\documentclass[11pt]{article} 

\usepackage{amsmath}
\usepackage{lmodern}
\usepackage[T1]{fontenc}
\usepackage[babel=true]{microtype}
\usepackage{amsxtra}%
\usepackage{amsfonts}%
\usepackage{amssymb}%
\usepackage{amsthm}
\usepackage{graphicx}
\usepackage{color,hyperref}
\hypersetup{colorlinks,breaklinks,
             linkcolor=blue,urlcolor=blue,
           anchorcolor=blue,citecolor=blue}
       
\usepackage{subfigure}
\usepackage{appendix}
\usepackage[small]{caption}
\usepackage{calc}

\usepackage[margin=0.75in]{geometry}

\newcommand{\eps}{\varepsilon}
\newcommand{\Z}{\mathbb{Z}}

\newcommand{\R}{\mathbb{R}}

\newcommand{\C}{\mathbb{C}}

\renewcommand{\phi}{\varphi}

\newcommand{\mcf}{\mathcal{F}}

\newcommand{\mcl}{\mathcal{L}}
\newcommand{\spn}{\text{\span}}

\usepackage{enumerate}
\renewcommand{\Re}{\mathrm{Re} \,}
\renewcommand{\Im}{\mathrm{Im} \,}

\newtheorem{thm}{Theorem}

\newtheorem*{thm*}{Theorem}

\newtheorem{prop}{Proposition}
\newtheorem{lemma}[prop]{Lemma}
\newtheorem{corollary}[prop]{Corollary}

\newtheorem{thmlocal}[prop]{Theorem}

\newtheorem{hyp}{Hypothesis}

\newtheorem{remark}[prop]{Remark}

\numberwithin{equation}{section}
\numberwithin{prop}{section}

\newcommand\blfootnote[1]{%
	\begingroup
	\renewcommand\thefootnote{}\footnote{#1}%
	\addtocounter{footnote}{-1}%
	\endgroup
}

\renewcommand{\spn}{\mathrm{span}}

\newcommand{\phih}{\phi^\mathrm{h}}
\newcommand{\psih}{\psi^\mathrm{h}}
\newcommand{\Deff}{D_\mathrm{eff}}
\newcommand{\phihp}{\phi^\mathrm{h, +}}
\newcommand{\phiIp}{\phi^{\mathrm{I},+}}
\newcommand{\vapp}{v^\mathrm{app}}
\newcommand{\Fres}{F_\mathrm{res}}
\newcommand{\bigo}{\mathrm{O}}

\newcommand{\phiI}{\phi^\mathrm{I}}
\newcommand{\phiII}{\phi^\mathrm{II}}

\newcommand{\FI}{\mathcal{F}^\mathrm{I}}
\newcommand{\FII}{\mathcal{F}^\mathrm{II}}
\newcommand{\Fh}{\mathcal{F}^\mathrm{h}}

\newcommand{\GI}{\mathcal{G}^\mathrm{I}}
\newcommand{\GII}{\mathcal{G}^\mathrm{II}}
\newcommand{\Gh}{\mathcal{G}^\mathrm{h}}

\newcommand{\tfi}{\tilde{\mathcal{F}}^\mathrm{I}}
\newcommand{\tfii}{\tilde{\mathcal{F}}^\mathrm{II}}
\newcommand{\tfh}{\tilde{\mathcal{F}}^\mathrm{h}}

\newcommand{\psiI}{\psi^\mathrm{I}}
\newcommand{\psiII}{\psi^\mathrm{II}}

\newcommand{\PI}{P^\mathrm{I}}
\newcommand{\Ph}{P^\mathrm{h}}
\newcommand{\PII}{P^\mathrm{II}}

\newcommand{\phiIIp}{\phi^{\mathrm{II},+}}

\pdfoutput=1

\begin{document}
\begin{center}
{\fontsize{15}{15}\fontseries{b}\selectfont{Front selection in reaction-diffusion systems via diffusive normal forms}}\\[0.2in]
Montie Avery \\[0.1in]
\textit{\footnotesize 
Boston University, Department of Mathematics and Statistics, 665 Commonwealth Ave, Boston, MA, 02215\\
}
\end{center}

\begin{abstract} 
	We show that propagation speeds in invasion processes modeled by reaction-diffusion systems are determined by marginal spectral stability conditions, as predicted by the \emph{marginal stability conjecture}. This conjecture was recently settled in scalar equations; here we give a full proof for the multi-component case. The main new difficulty lies in precisely characterizing diffusive dynamics in the leading edge of invasion fronts. To overcome this, we introduce coordinate transformations which allow us to recognize a leading order diffusive equation relying only on an assumption of generic marginal pointwise stability. We are then able to use self-similar variables to give a detailed description of diffusive dynamics in the leading edge, which we match with a traveling invasion front in the wake. We then establish front selection by controlling these matching errors in a nonlinear iteration scheme, relying on sharp estimates on the linearization about the invasion front. 
{\color{black} We briefly discuss applications to parametrically forced amplitude equations, competitive Lotka-Volterra systems, and a tumor growth model.} 
\end{abstract}

\blfootnote{Email: msavery@bu.edu. ORCID ID: 0000-0001-6524-1081. Data sharing not applicable to this article as no datasets were generated or analyzed during the current study.}

\section{Introduction}

The dynamics near an unstable state often play in important role in describing the formation of nontrivial coherent structures in physical systems. Instabilities may be observed either after a gradual change in system parameters induces a bifurcation, or after the sudden introduction of an external agent to which the system is unstable, such as in the spread of invasive species or epidemics in ecology. In spatially extended systems --- models posed on an unbounded domain in a mathematical idealization --- one expects small localized perturbations to an unstable state to grow until saturation at finite amplitude due to nonlinear effects, and then spread outward, producing a new selected stable state which invades the unstable state. A fundamental question is then to predict the invasion speed and features of the selected state in the wake. 

In the mathematics literature, rigorous predictions for invasion speeds are typically made by using comparison principles to estimate the speed by comparing with delicately constructed super- and sub- solutions \cite{Bramson1, Bramson2, Uchiyama, Comparison1, Comparison2, Lau, HamelPeriodic}. While these techniques can give very detailed information about propagation phenomena, including in higher spatial dimensions \cite{hamelnadirashvili, roquejoffre, berestyckinirenberg},  their use is limited to systems which obey the necessary comparison principles. Typically this restricts the applications to scalar, second order equations, although there are some cases when systems of more than one equation have enough special structure to allow the use of comparison principles \cite{BouinHendersonRyzhik1, BouinHendersonRyzhik2, LLWCooperative, LLWCompetitive}. We also note that when comparison principles are available, convergence results to selected fronts are typically global in nature, and fine aspects of this convergence remain active topics of research \cite{Graham, AnHendersonRyzhik1, AnHendersonRyzhik2}.

On the other hand, many interesting physical systems exhibiting invasion phenomena do not admit a comparison principle; see the extensive review \cite{vanSaarloosReview}. One then more generally predicts invasion speeds using the \emph{marginal stability conjecture} \cite{vanSaarloosReview, colleteckmann, ColletEckmannSH, deelanger, bers1983handbook, brevdo}, which postulates that nonlinear invasion speeds are determined by an appropriate notion of marginal spectral stability of an associated coherent structure describing the front interface. This conjecture was established very recently by the present author and Scheel in a framework of higher order scalar parabolic equations \cite{CAMS}. While this result made fundamental progress in establishing a rigorous model-independent framework for predicting invasion speeds, it did not directly treat multi-component systems, which appear in many relevant examples. The goal of this paper is to close that gap by proving the marginal stability conjecture for systems of parabolic equations.

To that end, we consider general reaction-diffusion systems of the form
\begin{align}
u_t = D u_{xx} + f(u), \quad u(x,t) \in \R^n, \quad x \in \R, \quad t > 0, \label{e: rd}
\end{align}
where the diffusion matrix $D \in \R^{n \times n}$ has strictly positive eigenvalues. We will assume that $u \equiv 0$ is an unstable equilibrium for this system, i.e. $f(0) = 0$ and $\Re \sigma_\mathrm{ess} (D\partial_x^2 + f'(0)) > 0$, where $\sigma_\mathrm{ess} (\mathcal{A})$ denotes the essential spectrum of an operator $\mathcal{A}$. The main result of this paper may then be informally stated as follows. 
\begin{thm*}
	Propagation speeds in reaction-diffusion systems are determined by marginal spectral stability of an associated invasion front. 
\end{thm*}
See Theorem \ref{t: main} for a more precise statement. 

\begin{remark}
	The results of this paper apply in the general semilinear parabolic case
	\begin{align}
	u_t = \mathcal{P}(\partial_x)u + f(u, \partial_x u, ..., \partial_x^{2m-1} u), \quad u = u (x,t) \in \R^n, \label{e: parabolic}
	\end{align}
	where $\mathcal{P}$ is a polynomial such that $\mathcal{P}(\partial_x)$ is an elliptic operator of order $2m$. The reaction-diffusion case \eqref{e: rd} is often the most relevant to applications, and so we focus on it to simplify the presentation.  We explain the modifications necessary for the general case in Remark \ref{rmk: parabolic}.  
\end{remark}

\begin{figure}
	\begin{center}
		\includegraphics[width=1\textwidth]{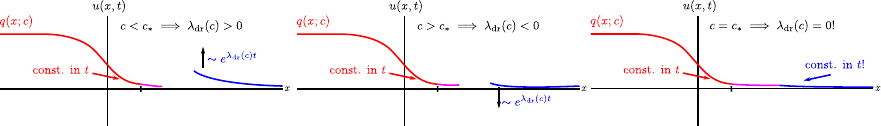}
		\caption{Schematic of the interior of a front (red) and the reconstructed front tail (blue) for $c < c_*$ (left), $c > c_*$ (middle) and $c = c_*$ (right). When $c = c_*$, the dynamics in the leading edge may be reconciled with the front in the wake.}
		\label{fig: matching}
	\end{center}
\end{figure}

\noindent \textbf{Background on front selection.} In a mathematical description, one often focuses on one-sided invasion processes, using \emph{steep} initial conditions which decay very rapidly to zero as $x \to \infty$ as a model for localized initial conditions, thereby focusing on a single front interface. A natural approach to predicting invasion speeds is then to investigate existence and stability of traveling front solutions $u(x,t) = q(x-ct;c)$ connecting $u = 0$ to a stable state $u_-$ in the wake. The difficulty is that invasion fronts typically exist for an open range of speeds $c$, and are all unstable in translation invariant function spaces due to the instability of the background state $u \equiv 0$. Using exponential weights to limit how perturbations can affect the tail decay of the front, one typically recovers spectral stability for a range of speeds $c \geq c_*$. However, these classes of perturbations do not allow one to consider steep initial data, for instance data supported on a half line, and solutions from steep initial data are typically observed to propagate with the minimal speed $c_*$. Thus, simple dynamical stability against restricted classes of perturbations is not sufficient to predict invasion speeds from localized or steep initial data.  

\noindent \textbf{Marginal stability as a selection mechanism}. To understand the invasion dynamics from steep initial data, in particular data which vanish for $x$ sufficiently large, a natural first candidate is to ``cut off the tail'' of a front solution $q(x;c)$ moving with speed $c$, by considering initial data of the form $\chi_L (x) q(x;c)$, where $\chi_L$ is some cutoff function that vanishes for $x$ sufficiently large. It is then natural to expect that the solution will rebuild its tail according to the dynamics at $+\infty$. To describe the solution, one then hopes to reconcile these leading edge dynamics with the dynamics in the bulk of the front, which should remain more or less constant in time in the frame moving with the front speed $c$. This is then only possible if the dynamics in the leading edge are neither growing nor decaying but also approximately constant in time. The dynamics in the leading edge are affected by the front speed $c$ which determines the choice of the co-moving frame: solutions in the leading edge exhibit pointwise temporal growth or decay $u(x,t) \sim e^{\nu_\mathrm{dr}(c) x} e^{ \lambda_\mathrm{dr} (c) t}$ depending on $c$. One then predicts invasion at the distinguished speed $c_*$ for which we have marginal pointwise stability in the leading edge, $\lambda_\mathrm{dr}(c_*) = 0$; see Figure \ref{fig: matching}. See \cite{CAMS, vanSaarloosReview, colleteckmann} for further details.


The proof of the analogous result for scalar higher order equations in \cite{CAMS} involves three main steps:
\begin{enumerate}
	\item The construction of an approximate solution which glues an invasion front traveling with the selected speed to a diffusive tail capturing the invasion dynamics of the leading edge. This approximate solution then governs the dynamics of open classes of genuine solutions. This construction was outlined via formal matched asymptotics in \cite{EbertvanSaarloos}, made rigorous in the scalar, second order case in \cite{NRRrefined}, and made rigorous for higher-order scalar equations in \cite{CAMS}. 
	\item Obtaining sharp estimates on the linearized dynamics near the corresponding invasion front, developed in \cite{SIMA} via a detailed analysis of the resolvent near the essential spectrum. 
	\item A delicate nonlinear argument which must control certain critical terms in the nonlinear iteration scheme resulting from the matching with the diffusive tail. 
\end{enumerate}

The most difficult step to adapt from the scalar to the systems case is the construction of the approximate solution with the diffusive tail in the leading edge. To study the dynamics in the leading edge, one passes to a moving frame with speed $c$ and linearizes about $u \equiv 0$, obtaining in the higher order scalar case
\begin{align}
u_t = \mathcal{P}(\partial_x)u + c \partial_x u + f'(0) u. \label{e: scalar linearized}
\end{align}
where $\mathcal{P}$ is a polynomial of order $2m$ satisfying the ellipticity condition $(-1)^{m+1} p_{2m} > 0$, where $p_{2m}$ is the highest order coefficient. The dynamics in the leading edge may be described by the \emph{dispersion relation} $d_c (\lambda, \nu) = 0$ obtained by substituting the Fourier-Laplace ansatz $u(x,t) \sim e^{\nu x} e^{\lambda t}$:
\begin{align}
d_c (\lambda, \nu) = \mathcal{P}(\nu) + c \nu + f'(0) - \lambda. \label{e: scalar dispersion relation}
\end{align}
 A crucial part of the marginal stability conjecture is the observation that pointwise stability in the linearization about $u \equiv 0$ may be predicted by the location of so-called \emph{pinched double roots} of the dispersion relation \cite{vanSaarloosReview, HolzerScheelPointwiseGrowth}. In particular, one typically assumes that there exists $\nu_* \in \R$ such that the dispersion relation has a double root at $(0, \nu_*)$, with generic expansion 
\begin{align}
	d_c (\lambda, \nu) = - \lambda + \alpha (\nu - \nu_*)^2 + \mathrm{O}\left((\nu-\nu_*)^3\right), \label{e: double root expansion}
\end{align}
valid near $(\lambda, \nu) = (0, \nu_*)$, where $\alpha > 0$. 
It follows that if $u$ solves the linearized problem \eqref{e: scalar linearized}, then the exponentially weighted variable $v(x,t) = e^{-\nu_* x} u(x,t)$ satisfies an equation of the form
\begin{align}
v_t = \alpha \partial_x^2 v + \mathrm{O}(\partial_x^3) v. \label{e: scalar leading edge}
\end{align}
This crucially relies on the fact that \eqref{e: scalar dispersion relation} implies that the symbol of the differential operator on the right hand side of \eqref{e: scalar linearized} can be expressed in terms of the dispersion relation $d_c (\lambda, \nu)$. Since the low-frequency dynamics can be seen to be dominant, one then views \eqref{e: scalar leading edge} as a perturbation of the heat equation, allowing the construction of a diffusive tail which can then be matched to an invasion front to obtain an approximate solution. 

However, in the case of multi-component reaction-diffusion systems \eqref{e: rd}, the relevant dispersion relation becomes 
\begin{align}
d_c (\lambda, \nu) = \det \left( D \nu^2 + c \nu I + f'(0) - \lambda I \right). 
\end{align}
The presence of the determinant significantly complicates the analysis in the leading edge: one can no longer directly express the symbol of the linearization about $u \equiv 0$
\begin{align}
M(\lambda, \nu, c) = D \nu^2 + c \nu I + f'(0) - \lambda I \label{e: original symbol def}
\end{align}
in terms of the dispersion relation as in the scalar case. We overcome this here by showing that, under only the assumption that the dispersion relation has a double root of the form \eqref{e: double root expansion}, for $\lambda, \nu$ small the symbol $M$ can be put via a change of coordinates into one of two \emph{diffusive normal forms}:
\begin{align}
M(\lambda, \nu-\eta_*, c) \sim \begin{pmatrix}
- \lambda + \alpha \nu^2 & \mathrm{O}(\lambda, \nu) \\
\mathrm{O}(\lambda, \nu^2) & \mathrm{O}(1) 
\end{pmatrix} \label{e: normal 1}
\end{align}
with $\alpha > 0$ or 
\begin{align}
M(\lambda, \nu-\eta_*, c) \sim \begin{pmatrix} - \lambda + \alpha \nu^2 &  \beta \nu + \mathrm{O}(\lambda, \nu^2) & \mathrm{O}(\lambda, \nu) \\
- \nu + \mathrm{O}(\lambda, \nu^2) & 1 + \mathrm{O}(\lambda, \nu) & \mathrm{O}(\lambda, \nu) \\
\mathrm{O}(\lambda, \nu^2) & \mathrm{O}(\lambda, \nu) & \mathrm{O}(1)
\end{pmatrix}, \label{e: normal 2}
\end{align}
with $\alpha + \beta > 0$ (in the case $n=2$, one simply ignores the last row and column of this matrix). 

In the case of \eqref{e: normal 1}, the leading order dynamics are then governed by a diffusion equation $v_t = \alpha v_{xx}$, with higher order perturbations resulting from coupling to the other components. In the second case \eqref{e: normal 2}, the leading order dynamics are governed by
\begin{align}
	u_t &= \alpha u_{xx} + \beta v_x, \\
	0 &= - u_x + v, 
\end{align}
leading to a diffusive equation of the form $u_t = (\alpha + \beta) u_{xx}$. Transforming to these normal forms therefore allows us construct a diffusive tail via a detailed analysis of the resulting equations in self-similar coordinates, as in the scalar case. We emphasize that the diffusive dynamics in the leading edge result directly from \eqref{e: double root expansion}, which captures generic marginal pointwise stability, and do not rely on the fact the the equation under consideration is second order or even parabolic. For instance, the construction here applies directly to the leading edge dynamics of the FitzHugh-Nagumo system considered in \cite{cartersch}, which is not strictly parabolic. 

We then follow the program outlined in \cite{CAMS} to match this diffusive tail with an invasion front in the wake. We adapt the sharp estimates for the linearization about this front from \cite{SIMA} to the multi-component case. With these two ingredients, a nonlinear stability argument for the approximate solution follows \emph{exactly} as in \cite{CAMS}, regardless of which of the two cases \eqref{e: normal 1} or \eqref{e: normal 2} we consider, demonstrating the robustness of this approach to front propagation which does not rely on comparison principles or the precise structure of the equation under consideration. 

Our main result, Theorem \ref{t: main}, universally reduces prediction of invasion speeds in reaction-diffusion systems to verifying spectral criteria, thereby reducing the infinite dimensional PDE problem to algebraic (finding double roots of the dispersion relation) and finite dimensional ODE (verifying existence and spectral stability of fronts) problems. This allows the use of a broad spectrum of new tools to make rigorous prediction of propagation speeds in invasion processes, including:
\begin{itemize}
	\item robust numerical methods \cite{ArndEigenvalues, Stablab} including validated numerics \cite{BeckJaquette} for spectral stability of nonlinear waves,
	\item geometric singular perturbation theory for existence \cite{Fenichel, cartersch} and spectral stability \cite{CarterdeRijkSandstede} of nonlinear waves in systems with multiple time scales,
	\item functional analytic methods regularizing singular perturbations and establishing existence and spectral stability of waves in singular limits \cite{AveryGarenaux, GohScheel, JensArnd},
	\item topological methods for existence of nonlinear waves \cite{vandenBerg, ArndCH1, ArndCH2},
	\item and numerical methods specifically designed for continuation and spectral stability of invasion fronts \cite{AveryHolzerScheel}. 
\end{itemize}

The remainder of the paper is organized as follows. In Section \ref{s: setup}, we formulate our conceptual assumptions and state our main results precisely. In Section \ref{s: overview}, we give some background on the marginal stability conjecture and linear spreading speeds. In Section \ref{s: colinear}, we construct the first normal form \eqref{e: normal 1}, and use self-similar coordinates to construct a diffusive tail and associated approximate solution. In Section \ref{s: linearly independent} we carry out the corresponding procedure for the second normal form \eqref{e: normal 2}. In Section \ref{s: linear estimates}, we adapt the linear estimates from \cite{SIMA} to the multi-component case, and then in Section \ref{s: stability argument} we explain how these estimates can be used to close a stability argument for the approximate solution, thereby proving Theorem \ref{t: main}. 
{\color{black} We conclude in Section \ref{s: discussion} with a discussion of applications and extensions of our results.} 

\subsection{Setup and main results}\label{s: setup}

We first give a precise mathematical formulation of the spectral assumptions encoded in the marginal stability conjecture, adapted from \cite{CAMS}. In the mathematical description of invasion processes, one often considers \emph{steep} initial data, which converge to zero very rapidly as $x \to \infty$, rather than initial data which is localized on both sides, for instance compactly supported. This captures the essential propagation dynamics while simplifying the mathematical description by focusing only on one front interface. We adopt this approach here.

We consider the reaction-diffusion system \eqref{e: rd}, in a moving frame with speed $c$
\begin{align}
u_t = Du_{xx} + c u_x + f(u). 
\end{align}
We assume that this system has at least two equilibria, $f(0) = f(u_-) = 0$. The linearization about $u \equiv 0$ is given by
\begin{align}
u_t = D u_{xx} + c u_x + f'(0) u, \label{e: rd lin}
\end{align}
with associated dispersion relation
\begin{align}
d_c (\lambda, \nu) = \det \left(D\nu^2 + c \nu I + f'(0) - \lambda I \right). \label{e: dispersion relation}
\end{align}
The location of \emph{pinched double roots} $(\nu_\mathrm{dr}(c), \lambda_\mathrm{dr}(c))$ of the dispersion relation predicts the temporal pointwise growth or decay $u(x,t) \sim e^{\nu_\mathrm{dr} (c) x} e^{\lambda_\mathrm{dr} (c) t}$ in the leading edge; see Section \ref{s: overview} below for further background on pinched double roots and linear spreading speeds. To capture marginal pointwise stability in the leading edge, we therefore assume that for some speed $c_*$, we have a marginally stable pinched double root $(\nu_\mathrm{dr} (c_*), \lambda_\mathrm{dr} (c_*)) = (\nu_*, 0)$. 

\begin{hyp}[Marginal pointwise stability] \label{hyp: dispersion reln}
	We assume there exists a speed $c_* > 0$ and $\nu_* < 0$ so that the following properties hold
	\begin{enumerate}[(i)]
		\item (Simple double root) For $\lambda$ small and $\nu$ near $\nu_*$, we have 
		\begin{align}
		d_{c_*} (\lambda, \nu) = d_{10} \lambda + d_{02} (\nu -\nu_*)^2 + \mathrm{O}\left((\nu-\nu_*)^3, \lambda^2, \lambda (\nu-\nu_*)\right), \label{e: hyp 1 simple pdr}
		\end{align}
		for some $d_{10}, d_{02} \in \R$ with $d_{10} d_{02} < 0$. 
		\item (Minimal critical spectrum) If $d_{c_*} (i \omega, ik + \nu_*) = 0$ for some $\omega, k \in \R$, then $\omega = k = 0$. 
		\item (No unstable spectrum) We have $d_{c_*}(\lambda, ik + \nu_*) \neq 0$ for any $k \in \R$ and any $\lambda \in \C$ with $\Re \lambda > 0$. 
	\end{enumerate}
\end{hyp}
Hypothesis \ref{hyp: dispersion reln}(i) assumes that we have a generic marginally stable double root for $c = c_*$. Conditions (ii) and (iii) guarantee that this double root is \emph{pinched} (again, see Section \ref{s: overview} for details), and that there are no unstable pinched double roots for $c = c_*$. Hypothesis \ref{hyp: dispersion reln} therefore captures marginal pointwise growth in the leading edge, and we refer to $c_*$ as the \emph{linear spreading speed}. We define $\eta_* = -\nu_*$, which captures the exponential decay rate of fronts in the leading edge. 

Next, we assume the existence of a front propagating with this linear spreading speed. 
\begin{hyp}[Existence of a critical front]\label{hyp: front existence}
	We assume there exists a solution to \eqref{e: rd} of the form 
	\begin{align}
	u(x,t) = q_* (x-c_*t), \quad \lim_{\xi \to -\infty} q_* (\xi) = u_-,\quad \lim_{\xi \to \infty} q_*(\xi) = 0. \label{e: fronts}
	\end{align}
	We refer to $q_*$ as the critical front. We further assume that $q_*$ has the generic asymptotics 
	\begin{align}
	q_* (x) = \left[ b (u_0 x + u_1) + a u_0 \right] e^{-\eta_* x} + \mathrm{O}(e^{-(\eta_*+ \eta) x}) \label{e: front asymptotics}
	\end{align}
	for some $a, b \in \R, u_0, u_1 \in \R^n$, and some $\eta > 0$.  
\end{hyp}
The asymptotics \eqref{e: front asymptotics} capture weak exponential decay resulting from the fact that Hypothesis \ref{hyp: dispersion reln} implies that the linearization about $u = 0$ in the first-order ODE formulation has a Jordan block of size $2$  \cite{HolzerScheelPointwiseGrowth}. This weak exponential decay then holds generically for fronts satisfying \eqref{e: fronts} \cite{AveryHolzerScheel, SIMA}. By replacing $q_*(x)$ by a suitable translate $q_*(x+\tilde{x}_0)$, we may without loss of generality assume that $b = 1$, which we do for the rest of the manuscript.  

We assume that the state $u_-$ selected by $q_*$ in the wake is strictly stable.  Via the Fourier transform, the associated dispersion relation
\begin{align}
d^- (\lambda, \nu) = \det (D \nu^2  + c_* \nu + f'(u_-) - \lambda I)
\end{align}
determines the spectrum $\Sigma^-$ of the linearization $D \partial_x^2 + c_* \partial_x + f'(u_-)$, with
\begin{align}
\Sigma^- = \{ \lambda \in \C: d^- (\lambda, ik) = 0 \text{ for some } k \in \R \}. 
\end{align}

\begin{hyp}[Stability in the wake]\label{hyp: stability on left}
	We assume that $\Re (\Sigma^-) < 0$. 
\end{hyp}

Hypothesis \ref{hyp: dispersion reln} gives a prediction for the invasion speed based on the linearization about $u \equiv 0$. If this prediction is correct, then the invasion process is referred to as \emph{linearly determined}, or \emph{pulled}. This prediction may fail, however; for instance, scalar equations of the form
	\begin{align}
	u_t = u_{xx} + u + \theta u^2 - u^3 
	\end{align}
may exhibit propagation from steep initial data at a rate faster than the linear spreading speed if $\theta$ is sufficiently large, a phenomenon referred to as \emph{nonlinearly determined} or \emph{pushed} invasion \cite{HadelerRothe, AveryHolzerScheel, vanSaarloosReview}. The marginal stability conjecture is easier to establish in the case of pushed invasion, since one can recover a spectral gap for the linearization while still allowing for perturbations which are sufficiently large to include steep initial data. We therefore focus here on pulled invasion. Pushed invasion occurs when the linearization about a pulled front has an unstable eigenvalue \cite{AveryHolzerScheel}, which we exclude as follows. 

Let 
\begin{align}
\mathcal{A} = D \partial_x^2 + c_* \partial_x + f'(q_*)
\end{align}
denote the linearization about the critical front $q_*$. Since $q_*(x)$ converges to the unstable state $u = 0$ as $x \to \infty$, the essential spectrum of $\mathcal{A}$ in a translation invariant space such as $L^p (\R)$ is unstable. We recover marginal spectral stability by restricting to exponentially localized perturbations, defining a smooth positive weight $\omega$ satisfying
\begin{align}
\omega(x) = \begin{cases}
e^{\eta_*x}, & x \geq 1, \\
1, & x \leq -1. 
\end{cases} \label{e: omega def}
\end{align}
We then define the weighted linearization via the conjugate operator
\begin{align}
\mcl g = \omega \mathcal{A} (\omega^{-1} g). 
\end{align}
As a consequence of Hypothesis \ref{hyp: dispersion reln}, Hypothesis \ref{hyp: stability on left}, and Palmer's theorem \cite{Palmer1, Palmer2}, the essential spectrum of $\mcl$ is marginally stable; see \cite[Figure 1]{CAMS}. To exclude pushed fronts, we assume that the point spectrum of $\mcl$ is stable. 

\begin{hyp}[No unstable point spectrum]\label{hyp: point spectrum}
	We assume that the weighted linearization $\mcl$ has no eigenvalues $\lambda$ with $\Re \lambda \geq 0$. We further assume that there is no bounded solution to the equation $\mcl u = 0$. 
\end{hyp}

\begin{remark}
	The case where there is a bounded solution to $\mcl u =0$ marks a transition from pulled to pushed front propagation; see \cite{AveryHolzerScheel} for further details. 
\end{remark}

We now define algebraic weights used to state our main result. Given $r_\pm \in \R$, define a smooth positive weight function $\rho_{r_-, r_+}$ satisfying
\begin{align}
\rho_{r_-, r_+} (x) = \begin{cases}
x^{r_-}, & x \leq -1, \\
x^{r_+}, & x \geq 1.  
\end{cases} \label{e: alg weights}
\end{align} 
We are now ready to state our main result. The result is phrased in terms of open sets of initial data $\mathcal{U}_\eps$, which are neighborhoods of an approximate solution which connects the critical front $q_*$ to a diffusive tail in the leading edge. 

\begin{thm}\label{t: main}
	Assume Hypotheses \ref{hyp: dispersion reln} through \ref{hyp: point spectrum} hold. Fix $0 < \mu < \frac{1}{8}$ and let $r = 2+\mu$. Then the critical front $q_*$ is selected in the sense of \cite[Definition 1]{CAMS}. More precisely, for each $\eps > 0$ there exists a set of initial data $\mathcal{U}_\eps \subseteq L^\infty (\R)$ such that the following hold.
	\begin{enumerate}
		\item For each $u_0 \in \mathcal{U}_\eps$, the solution $u$ to \eqref{e: rd} with initial data $u_0$ satisfies
		\begin{align}
		\sup_{x \in \R} |\rho_{0,-1} (x) \omega(x) [ u(x + \sigma(t), t) - q_*(x)]| < \eps, 
		\end{align}
		for all $t \geq t_*(u_0)$, sufficiently large, where 
		\begin{align}
		\sigma(t) = c_* t - \frac{3}{2 \eta_*} \log t + x_\infty (u_0)
		\end{align}
		for some $x_\infty (u_0) \in \R$. 
		\item $\mathcal{U}_\eps$ contains \emph{steep} initial data. More precisely, there exists $u_0 \in \mathcal{U}_\eps$ such that $u_0 (x) \equiv 0$ for $x$ sufficiently large. 
		\item $\mathcal{U}_\eps$ is open in the topology induced by the norm $\| f \| = \| \rho_{0, r} \omega f \|_{L^\infty}$. 
	\end{enumerate}
\end{thm}

Theorem \ref{t: main} guarantees that there are open sets of steep initial data which remain close to the critical front $q_*$  for all time in an appropriate moving coordinate frame, in particular guaranteeing that the propagation speed $\sigma'(t)$ is given by the linear spreading speed $c_*$ up to the universal correction $- \frac{3}{2 \eta_* t}$. This correction was first established for scalar second order equations by Bramson \cite{Bramson1, Bramson2}. The universality of this correction, with the expression depending only on the leading order expansion of the dispersion relation via $\eta_*$, was conjectured by Ebert and van Saarloos \cite{EbertvanSaarloos} and recently confirmed rigorously \cite{CAMS} for scalar higher order equations. In addition to predicting the propagation speed, Theorem \ref{t: main} therefore further confirms the universality of this correction for multi-component systems. 

\begin{remark}
	In Theorem \ref{t: main}, initial data is measured with weight $\rho_{0,r}\omega$, while the solution is measured with weight $\rho_{0,-1} \omega$. This is analogous to standard $L^1$-$L^\infty$ estimates for solutions to the heat equation, leveraging a loss of spatial localization to gain temporal decay. 
\end{remark}

{\color{black}
\begin{remark}
	Theorem \ref{t: main} is a local result, in that initial data in $\mathcal{U}_\eps$ are initially close to $q_*$. More precisely, as mentioned above, the proof of Theorem \ref{t: main} relies on the construction of an approximate solution $\vapp(\xi, t)$ which glues $q_*$ to a diffusive tail. When $t$ is large, $\vapp$ itself is close to $q_*$; see Corollaries \ref{c: colinear front approximated by vapp} and \ref{c: lin ind front approximated by vapp}. Initial data in $\mathcal{U}_\eps$ are required to be close in a weighted norm  to this approximate solution frozen at a fixed large time $T$; see \eqref{e: U eps def} for the precise definition of $\mathcal{U}_\eps$. Nonetheless, each $\mathcal{U}_\eps$ contains some steep initial data, which vanish for $x$ sufficiently large. Moreover, requirements on the tail structure of the initial condition are not an arbitrary technical condition: see Section \ref{s: discussion} for further discussion. 
\end{remark}

Due to this perturbative structure, we can ultimately close a nonlinear iteration argument which relies primarily on sharp estimates on the linearization $e^{\mcl t}$ about $q_*$. This is in contrast to global convergence results in parabolic equations, which typically rely on a priori $L^\infty$ estimates on solutions obtained via the maximum principle \cite{ChenQi, ColletXin}. 

}

{ \color{black} We mention that Hypotheses \ref{hyp: dispersion reln} through \ref{hyp: point spectrum} are open within the class of reaction-diffusion systems, by a straightforward adaptation of \cite[Theorem 2]{CAMS}.} 


\noindent \textbf{Acknowledgments.} The author is grateful to Margaret Beck for helpful feedback on the presentation of the manuscript. This work was supported by the National Science Foundation through NSF-DMS-2202714. Any opinions, findings, and conclusions or recommendations expressed in this material are those of the author and do not necessarily reflect the views of the National Science Foundation. 

\section{Background and overview}\label{s: overview}

In this section, we give some background on relevant concepts in front propagation. 

\subsection{Pinched double roots and marginal pointwise growth}

Here we briefly introduce the concept of pinched double roots of the dispersion relation and explain how they can be used to predict pointwise growth rates. We follow the recent treatment in \cite{HolzerScheelPointwiseGrowth}, but similar ideas can be found in the earlier works \cite{vanSaarloosReview, bers1983handbook, brevdo, deelanger, Lifshitz}. 

\begin{figure}
	\begin{center}
		\includegraphics[width=1\textwidth]{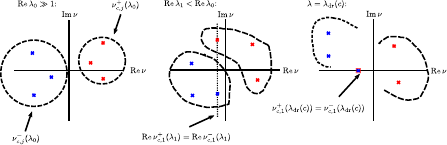}
		\caption{Schematic of spatial eigenvalues $\nu^\pm_{c, j}(\lambda)$ for $\lambda_0$ with $\Re \lambda_0 \gg 1$ (left), a sample $\lambda_1$ with $\Re \lambda_1 < \Re \lambda_0$ (middle), and at a pinched double root $\lambda = \lambda_\mathrm{dr}(c)$ (right). Red crosses denote spatial eigenvalues $\nu_{c,j}^+(\lambda)$ which were unstable for $\Re \lambda \gg 1$, while blue crosses denote spatial eigenvalues $\nu_{c,j}^-(\lambda)$ which were stable for $\Re \lambda \gg 1$. The dashed lines indicate contours which can be used to define Dunford integrals for the spectral projections $P^\mathrm{s/u}(\lambda)$. Notice that even when $\Re \nu_1^+(\lambda_1) = \Re \nu_1^- (\lambda_1)$, we can still separate the (previously) stable and unstable eigenvalues with smooth contours and therefore continue the spectral projections up to this point. When $\lambda = \lambda_\mathrm{dr}(c)$, there is no way to close the contours to separate the stable and unstable spatial eigenvalues, leading to a singularity of the continued spectral projections.}
		\label{fig: dunford}
	\end{center}
\end{figure}

Roots $\nu_c(\lambda)$ of the dispersion relation $d_c(\lambda, \nu)$ correspond precisely to eigenvalues of the matrix $A_c (\lambda)$ obtained by rewriting the resolvent equation
\begin{align}
(D \partial_x^2 + c \partial_x + f'(0) - \lambda) u = g
\end{align}
as a first order system in $U = (u, u_x) \in \R^{2n}$
\begin{align}
(\partial_x - A_c (\lambda)) U = \begin{pmatrix} 0 \\ D^{-1} g \end{pmatrix}. 
\end{align}
Associated to these equations are the resolvent kernel $G_\lambda (\xi)$ and the first order matrix Green's function $T_\lambda(\xi)$, which solve
\begin{align}
(D \partial_\xi^2 + c \partial_\xi + f'(0) - \lambda) G_\lambda = - \delta_0 I_n,  
\end{align}
and
\begin{align}
(\partial_\xi - A_c (\lambda)) T_\lambda = - \delta_0 I_{2n},
\end{align}
respectively, where $I_k$ is the identity matrix of size $k$. We call the eigenvalues $\nu_c(\lambda)$ of $A_c(\lambda)$ \emph{spatial eigenvalues}. 

Via the inverse Laplace transform, one may write the solution to the linearization \eqref{e: rd lin} with initial data $u_0(y) = g(y)$ in terms of the resolvent kernel as
\begin{align} 
u(x,t) = \frac{1}{2 \pi i} \int_\Gamma e^{\lambda t} \int_\R G_\lambda(x-y) g(y) \, dy d \lambda. \label{e: pointwise laplace}
\end{align}
In a classical semigroup approach, the sectorial contour $\Gamma$ is chosen to lie to the right of the essential spectrum of the linearization in a fixed function space. The essential spectrum depends on the choice of function space and may for instance be moved with exponential weights \cite{Sattinger}. However, we are interested here in measuring \emph{pointwise} growth or decay in a fixed window, and so do not want to fix a function space. Instead, when we restrict to initial data $g$ which is very localized, for instance compactly supported, one finds that the pointwise formula \eqref{e: pointwise laplace} is well-defined for any sectorial contour $\Gamma$ which lies to the right of the singularities of $\lambda \mapsto G_\lambda(\xi)$ for \emph{fixed} $\xi$. The boundary of the essential spectrum of $G_\lambda$ in a translation invariant function space is often associated with a loss of \emph{spatial} localization of $G_\lambda$, not a loss of analyticity in $\lambda$, so in many cases it is possible to push this contour past the essential spectrum in a given space if one is only interested in pointwise growth or decay in a finite window.

Pointwise temporal growth or decay of $u(x,t)$ can then be estimated by the location of pointwise singularities of $G_\lambda(\xi)$. The resolvent kernel $G_\lambda(\xi)$ may be recovered from the matrix Green's function $T_\lambda$, and the two have precisely the same singularities \cite[Lemma 2.1]{HolzerScheelPointwiseGrowth}. The matrix Green's function $T_\lambda$ may be constructed by using a splitting of stable and unstable eigenspaces of $A_c(\lambda)$ for $\Re \lambda \gg 1$. The stable and unstable eigenspaces may be constructed using a Dunford integral to define the associated spectral projections $P^\mathrm{s/u}(\lambda)$.  These projections may be analytically continued as $\Re \lambda$ decreases until the continued stable/unstable eigenspaces can no longer be separated \cite[Section 2]{HolzerScheelPointwiseGrowth}. Singularities of $T_\lambda$ then correspond to singularities of the continued spectral projections $P^\mathrm{s/u}(\lambda)$. A necessary condition for such a singularity is that an eigenvalue $\nu^-_c(\lambda)$ of $A_c(\lambda)$ which was stable for $\Re \lambda \gg 1$ collides with an eigenvalue $\nu^+_c(\lambda)$ which was unstable for $\Re \lambda \gg 1$. A pinched double root of the dispersion relation is, by definition, such a collision of eigenvalues: double root refers to the collision of these two eigenvalues, which correspond to roots of the dispersion relation, and ``pinched'' refers to the fact that the eigenvalues come in from different (stable/unstable) directions as $\Re \lambda$ is decreased. See Figure \ref{fig: dunford} for a schematic. 

In summary, all pointwise singularities of the resolvent kernel $G_\lambda (\xi)$, which can be used to estimate the pointwise growth of $u(x,t)$, correspond to pinched double roots of the dispersion relation. Not all pinched double roots give rise to singularities, as it is possible that two eigenvalues collide while having distinct limiting eigenspaces, so that the stable and unstable eigenspaces may still be separated \cite[Remark  5.5]{HolzerScheelPointwiseGrowth}. In Hypothesis \ref{hyp: dispersion reln}, we restrict to the generic case of \emph{simple} pinched double roots, which are robust \cite[Lemma 7.1]{CAMS} and always give rise to pointwise growth modes \cite[Lemma 4.4]{HolzerScheelPointwiseGrowth}. 

Hypothesis \ref{hyp: dispersion reln} therefore guarantees that for $c = c_*$, there is a pinched double root at $\lambda = 0$ giving rise to a pointwise growth mode, and no other unstable pinched double roots. One then predicts neither pointwise exponential growth nor pointwise exponential decay at this speed, a phenomenon we refer to as \emph{marginal pointwise stability}. See \cite{HolzerScheelPointwiseGrowth} for further background on pointwise growth modes and linear spreading speeds.

\subsection{The logarithmic delay}\label{s: log delay}

Even at the linear spreading speed,  the dynamics in the leading edge at the linear spreading speed are not quite stationary, but rather we have
\begin{align}
u(x,t) \sim u_0 \frac{x}{t^{3/2}} e^{-\eta_* x} e^{-x^2/(4 D_\mathrm{eff} t)}, \quad x \to \infty,
\end{align}
where $D_\mathrm{eff}$ is an effective diffusivity, and $u_0 \in \R^n$. These asymptotics explain the universal logarithmic delay $-\frac{3}{2 \eta_*} \log t$ in the position of the front: when $x = -\frac{3}{2 \eta_*} \log t$, we have $t^{-3/2} e^{-\eta_* x} = 1$, so that the dynamics in the leading edge are now approximately stationary at this delayed position, which allows for matching with the critical front in the wake. See Section \ref{s: colinear diffusive tail} for further details. 

\subsection{Preliminaries and notation}

We first provide a reformulation of Hypothesis \ref{hyp: dispersion reln}(i) which we rely on in constructing the diffusive normal forms \eqref{e: normal 1}, \eqref{e: normal 2} for the symbol of the linearization. First we introduce the notation
\begin{align}
A(\lambda, \nu) = M(\lambda, \nu -\eta_*, c_*) \label{e: symbol def}
\end{align}
for the symbol $A(\lambda, \nu)$ of the linearization about $u \equiv 0$, in the exponentially weighted space, where $M$ is the symbol in the original variables, given by \eqref{e: original symbol def}. 

We separate terms in $A(\lambda, \nu)$ by powers of $\lambda$, and $\nu$ as 
\begin{align}
A(\lambda, \nu) = A^0 + A^{10} \lambda + A^{01} \nu + A^{02} \nu^2. \label{e: A expansion}
\end{align}
\begin{hyp}[Simple double root] \label{hyp: DR matrix pencil}
	We assume there exist $u_0, u_1 \in \R^n$ such that 
	\begin{align}
	A^0 u_0 &= 0, \label{e: hyp kernel} \\ 
	A^0 u_1 + A^{01} u_0 &= 0. \label{e: hyp double root}
	\end{align}
	We further assume that $\ker A^0 = \spn (u_0)$ and let $e_\mathrm{ad}$ span the kernel of $(A^0)^T$. Finally, we assume
	\begin{align}
	\langle - u_0, e_\mathrm{ad} \rangle \langle A^{02} u_0 + A^{01} u_1, e_\mathrm{ad} \rangle < 0. \label{e: hyp simple DR}
	\end{align}
\end{hyp}

The following result, which is Lemma 2.1 of \cite{AveryHolzerScheel}, guarantees equivalence of Hypotheses \ref{hyp: dispersion reln}(i) and \ref{hyp: DR matrix pencil}. 

\begin{lemma}
	Hypothesis \ref{hyp: dispersion reln}(i) and \ref{hyp: DR matrix pencil} are equivalent. That is, Hypothesis \ref{hyp: dispersion reln}(i) implies Hypothesis \ref{hyp: DR matrix pencil}, and conversely if there exist $c_* > 0$ and $\nu_* < 0$ so that Hypothesis \ref{hyp: DR matrix pencil} holds, then Hypothesis \ref{hyp: dispersion reln}(i) holds as well. 
\end{lemma}
\begin{proof}
	We view the equation $A(\lambda,\nu) u = 0$ near $(\lambda, \nu) = 0$ either as a classical eigenvalue problem in $\lambda$, with $\nu = 0$, or as a nonlinear eigenvalue problem in $\nu$ with $\lambda = 0$. Hypothesis \ref{hyp: dispersion reln}(i) requires that $\lambda = 0$ be an algebraically simple eigenvalue to the classical eigenvalue problem $A(\lambda, 0)u = 0$, and that $\nu = 0$ is an algebraically double solution to the nonlinear eigenvalue problem $A(0, \nu)u = 0$, where algebraic multiplicity is defined as the order of the root of the determinant. The former condition is equivalent to the two conditions $\ker (A^0) = \spn (u_0)$ and $\langle u_0, e_\mathrm{ad} \rangle \neq 0$, which guarantee that the kernel is simple and that $u_0$ is not in the range of $A^0$ and hence does not belong to a nontrivial Jordan chain. 
	
	Conditions \eqref{e: hyp kernel}-\eqref{e: hyp double root} together with $\langle A^{02} u_0 + A^{01} u_1, e_\mathrm{ad} \rangle \neq 0$ from \eqref{e: hyp simple DR} guarantee that $\nu = 0$ has an associated Jordan chain of length 2 for the nonlinear eigenvalue problem $A(0, \nu) u = 0$. The length of this Jordan chain then corresponds with the algebraic multiplicity defined as the order of the root of the determinant \cite{Voss}, and so this is equivalent to the requirement in Hypothesis \ref{hyp: dispersion reln}(i) that $\nu = \nu_*$ be a double root of the associated dispersion relation. For further background on nonlinear eigenvalue problems, see \cite{Voss}. 
	
	It only remains to verify that \eqref{e: hyp simple DR} is equivalent to $d_{10}d_{02} < 0$ from Hypothesis \ref{hyp: dispersion reln}. This follows from a Lyapunov-Schmidt reduction, which recovers the dispersion relation from Hypothesis \ref{hyp: DR matrix pencil} by finding all solutions to $A(\lambda, \nu) u = 0$ in a neighborhood of $(\lambda, \nu) = (0,0)$ via a reduced scalar equation computed by projecting onto the cokernel. 
\end{proof}

Throughout the remainder of the paper, we assume that Hypotheses \ref{hyp: dispersion reln} through \ref{hyp: DR matrix pencil} hold. We can now understand the distinction between the two normal forms \eqref{e: normal 1} and \eqref{e: normal 2}. If the vectors $u_0$ and $u_1$ introduced in Hypothesis \ref{hyp: DR matrix pencil} are co-linear, then the symbol $A(\lambda,\nu)$ may be put in the normal form \eqref{e: normal 1}, while the other form \eqref{e: normal 2} applies when $u_0$ and $u_1$ are linearly independent. 

\noindent \textbf{Function spaces.} Recall the definition \eqref{e: alg weights} of the algebraic weight $\rho_{r_-, r_+}$. Given such a weight, we define the algebraically weighted Sobolev space $W^{k,p}_{r_-, r_+} (\R, \C^n)$ through the norm
\begin{align}
\| g \|_{W^{k,p}_{r_-, r_+}} = \| \rho_{r_-, r_+} g \|_{W^{k,p}}, 
\end{align}
where $W^{k,p} (\R, \C^n)$ is the standard Sobolev space of weakly differentiable functions up to order $k$ with integrability index $1 \leq p \leq \infty$. If $k = 0$, we write $W^{k,p}_{r_-, r_+} (\R, \C^n) = L^p_{r_-, r_+} (\R, \C^n)$, with corresponding notation for the norms. 

We will also need general exponential weights for the proofs of our linear estimates. Given $\eta_\pm \in \R$, we define a smooth positive weight function $\omega_{\eta_-, \eta_+}$ satisfying
\begin{align}
\omega_{\eta_-, \eta_+} (x) = \begin{cases}
e^{\eta_- x}, &x \leq -1, \\
e^{\eta_+ x}, &x \geq 1. 
\end{cases}
\end{align}
We then define an exponentially weighted Sobolev space $W^{k,p}_{\mathrm{exp}, \eta_-, \eta_+} (\R, \C^n)$ through the norm 
\begin{align}
\| g \|_{W^{k,p}_{\mathrm{exp}, \eta_-, \eta_+}} = \| \omega_{\eta_-, \eta_+} g \|_{W^{k,p}}. 
\end{align}
As for the algebraically weighted spaces, if $k = 0$ we let $W^{k,p}_{\mathrm{exp}, \eta_-, \eta_+} (\R, \C^n) = L^p_{\mathrm{exp}, \eta_-, \eta_+} (\R, \C^n)$ with corresponding notation for the norms. When $p = 2$, we let $H^k_{\mathrm{exp}, \eta_-, \eta_+} (\R, \C^n) := W^{k, p}_{\mathrm{exp}, \eta_-, \eta_+} (\R, \C^n)$. 

\noindent \textbf{Additional notation.} Given two Banach spaces $X$ and $Y$, we let $\mathcal{B}(X,Y)$ denote the space of bounded linear operators from $X$ to $Y$, equipped with the operator norm. Given $\delta > 0$, we let $B(0,\delta)$ denote the unit ball of radius $\delta$ centered at the origin in the complex plane. We may sometimes abuse notation slightly by writing a function $u(\cdot, t)$ as $u(t)$, suppressing the spatial dependence and viewing $u(t)$ as an element of some Banach space at a fixed time. We let $\langle x \rangle = (1+x^2)^{1/2}$.

\section{Constructing the approximate solution --- co-linear case}\label{s: colinear}
In this section, we analyze the diffusive dynamics in the leading edge under the assumption that Hypothesis \ref{hyp: DR matrix pencil} holds, with $u_0$ and $u_1$ co-linear. Without loss of generality, we may then take $u_0 = u_1$. 

\subsection{Diffusive normal form}
We show that in this case, the symbol $A(\lambda, \nu)$ may be transformed into the form \eqref{e: normal 1}
\begin{lemma}[Diffusive normal form --- co-linear case]\label{l: normal form colinear}
	Assume Hypothesis \ref{hyp: DR matrix pencil} holds with $u_0 = u_1$. Then there exist invertible matrices $S,Q \in \R^{n \times n}$ such that 
	\begin{align}
	B(\lambda, \nu) := S A(\lambda, \nu) Q = \begin{pmatrix}
	-\lambda + b_{11}^{02} \nu^2 & b_{12} (\lambda, \nu) \\
	b_{21} (\lambda, \nu) & b_{22}(\lambda, \nu), 
	\end{pmatrix}
	\label{e: colinear normal form}
	\end{align}
	where $ b_{11}^{02} > 0$, and $b_{21}(\lambda, \nu) \in \C^{n-1 \times 1}, b_{12}(\lambda, \nu) \in \C^{1 \times n-1}$, and $b_{22} (\lambda, \nu) \in \C^{n-1 \times n-1}$, are polynomials which satisfy 
	\begin{align}
	b_{12} (\lambda, \nu) &= b_{12}^{01} \nu + \mathrm{O}(\lambda, \nu^2) \\
	b_{21} (\lambda, \nu) &= b_{21}^{10} \lambda + b_{21}^{02} \nu^2,\\
	b_{22} (\lambda, \nu) &= b_{22}^{00} + b_{22}^{01} \nu + \mathrm{O}(\lambda, \nu^2),
	\end{align}
	where $b_{22}^{00} \in \C^{n-1 \times n-1}$ is invertible. 
\end{lemma}
\begin{proof}
	We first let $S$ and $Q$ be arbitrary invertible matrices, and denote 
	\begin{align}
	B(\lambda, \nu) = S A(\lambda, \nu) Q = \begin{pmatrix}
	b_{11} (\lambda, \nu) & b_{12} (\lambda, \nu) \\
	b_{21} (\lambda, \nu) & b_{22} (\lambda, \nu)
	\end{pmatrix}.
	\end{align}
	We then use our freedom to choose $S$ and $Q$ to eliminate terms in $b_{ij}$ which do not match the normal form \eqref{e: colinear normal form}. We expand $B$ as
	\begin{align}
	B(\lambda, \nu) = B^0 + B^{10} \lambda + B^{01} \nu + B^{02} \nu^2. \label{e: B expansion}
	\end{align}
	
	\emph{Step 1: Requiring $b_{11} (0,0) = 0, b_{12} (0,0) = 0$, and $b_{21} (0,0) = 0$.} Choose the first column of $Q$ to be equal to $u_0$, so that $Q e_0 = u_0$, where $e_0 = (1,0, ..., 0)^T$ is the first standard basis vector. This implies that $e_0$ is in the kernel of $B^0$, since 
	\begin{align}
	B^0 e_0 = S A^0 Q e_0 = S A^0 u_0,
	\end{align}
	and $u_0 \in \ker A^0$. Hence the first column of $B^0$ is equal to zero, i.e. $b_{11} (0,0) = 0$ and $b_{21} (0,0) = 0$. Similarly, choosing the first row of $S$ to be $c e_\mathrm{ad}^T$, for some $c \in \R$ which we will choose later, we find $S^T e_0 = c e_\mathrm{ad}$, and so $\ker ((B^0)^T) = \spn (e_0)$, which implies that $b_{12} (0,0) = 0$. 
	
	\emph{Step 2: Verifying $b_{11}(\lambda, \nu)$ and $b_{21} (\lambda, \nu)$ are $\mathrm{O}(\lambda,\nu^2).$} Since we are assuming that $u_1 = u_0 \in \ker A^0$, condition \eqref{e: hyp double root} of Hypothesis \ref{hyp: DR matrix pencil} reduces to 
	\begin{align}
	A^{01} u_0 = - A^0 u_0 = 0,. \label{e: co normal form A01 u0}
	\end{align}
	From this fact together with the expansion \eqref{e: B expansion} and the choice $Qe_0 = u_0$, we see that
	\begin{align*}
	B^{01} e_0 = S A^{01} Q e_0 = S A^{01} u_0 = 0, 
	\end{align*}
	so that the first column of $B^{01}$ is equal to zero, which implies together with the previous step that $b_{11} (\lambda, \nu)$ and $b_{21} (\lambda, \nu)$ are $\mathrm{O}(\lambda, \nu^2)$. We may therefore write $b_{11} (\lambda,\nu) = b_{11}^{10} \lambda + b_{11}^{02} \nu^2$ for some $b_{11}^{10}, b_{11}^{02} \in \R$. 
	
	\emph{Step 3: Requiring $b_{11}^{10} = -1$ and $b_{11}^{02}$ > 0.} Note that \eqref{e: co normal form A01 u0} implies that condition \eqref{e: hyp simple DR} of Hypothesis \ref{hyp: DR matrix pencil} reduces to 
	\begin{align}
	\langle -u_0, e_\mathrm{ad}\rangle \langle A^{02} u_0, e_\mathrm{ad}\rangle < 0. \label{e: co normal form step 3 1}
	\end{align}
	On the other hand, comparing the expansion \eqref{e: B expansion} for $B(\lambda, \nu)$ with the expansion \eqref{e: A expansion} for $A(\lambda, \nu)$ implies that 
	\begin{align}
	b_{11}^{10} = \langle - SQ e_0, e_0 \rangle, \quad b_{11}^{02} = \langle S A^{02} Q e_0, e_0 \rangle.
	\end{align}
	Using the conditions $S^T e_0 = c e_\mathrm{ad}$ and $Q e_0 = u_0$, we see that 
	\begin{align*}
	b_{11}^{10} = - \langle Q e_0, S^T e_0 \rangle = - c \langle u_0, e_\mathrm{ad} \rangle.
	\end{align*}
	Choosing $c = \langle u_0, e_\mathrm{ad} \rangle^{-1}$, we then obtain $b_{11}^{10} = -1$. Similarly, we have 
	\begin{align}
	b_{11}^{02} = \langle S A^{02} Q e_0, e_0 \rangle = \langle A^{02} Q e_0, S^T e_0 \rangle =  \langle u_0, e_\mathrm{ad} \rangle^{-1} \langle A^{02} u_0, e_\mathrm{ad} \rangle > 0
	\end{align}
	by \eqref{e: co normal form step 3 1}. 
	
	\emph{Step 4: Verifying $b_{22}^{00}$ is invertible.} By assumption, the kernel of $A^0$ is one dimensional, and so also the kernel of $B^0$ is one dimensional (since $S$ and $Q$ are invertible) and spanned by $e_0$. The submatrix $b_{22}^{00}$ must therefore be invertible, since a nontrivial kernel of this matrix would imply that the kernel of $B^0$ has dimension greater than one. 
\end{proof}

\subsection{Constructing the diffusive tail}\label{s: colinear diffusive tail}
We now use the normal form constructed in Lemma \ref{l: normal form colinear} to analyze the dynamics of the linearization about $u = 0$, and construct a diffusive tail which can be matched to the invasion front. As mentioned in Section \ref{s: log delay}, we need to incorporate a logarithmic delay in the position of the front in order to make the dynamics in the leading edge roughly constant in time in this frame. We therefore introduce the shifted variable
\begin{align}
y = x - c_*t + \frac{3}{2 \eta_*} \log (t+T) - \frac{3}{2 \eta_*} \log T. \label{e: y def}
\end{align}
The parameter $T$ will be chosen large so that the diffusive tail will be initially well-developed, ultimately allowing us to close a perturbative stability argument. We let $U(y,t) = u(x,t)$, so that the system \eqref{e: rd} becomes 
\begin{align}
U_t = D U_{yy} + \left(c_* - \frac{3}{2 \eta_* (t+T)} \right) U_y + f(U). \label{e: U eqn}
\end{align}
We then use the exponential weight $\omega$ defined by \eqref{e: omega def} to suppress the instability in the leading edge, defining $v(y,t) = \omega(y) U(y,t)$, so that $v$ solves 
\begin{align}
v_t = D \omega \partial_{yy} (\omega^{-1} v) + \left(c_* - \frac{3}{2 \eta_* (t+T)} \right) \omega \partial_y (\omega^{-1} v) + \omega f(\omega^{-1} v). 
\end{align}
We extract the linear term in $\omega f(\omega^{-1} v)$, so that we may rewrite this equation as 
\begin{align}
v_t = D \omega \partial_{yy} (\omega^{-1} v) + \left(c_* - \frac{3}{2 \eta_* (t+T)} \right) \omega \partial_y (\omega^{-1} v) + f'(0) v + \omega N(\omega^{-1} v), \label{e: v eqn}
\end{align}
where
\begin{align}
N(\omega^{-1} v) = f(\omega^{-1} v) - f'(0) \omega^{-1} v
\end{align}
satisfies
\begin{align}
| \omega N(\omega^{-1} v)| \leq C(B) \omega^{-1} |v|^2 
\end{align}
provided $\| \omega^{-1} v \| \leq B$, by Taylor's theorem. Since $\omega(y) = e^{\eta_*y}$ for $y \geq 1$, by the definition \eqref{e: symbol def} of the symbol $A(\lambda, \nu)$, the equation \eqref{e: v eqn} becomes, for $y \geq 1$,
\begin{align}
- A(\partial_t, \partial_y)  v = e^{\eta_* y} \left( - \frac{3}{2 \eta_* (t+T)} \right) \partial_y (e^{-\eta_* y} v) + e^{\eta_* y} N(e^{-\eta_*y}). 
\end{align}
Rewriting slightly, we obtain 
\begin{align}
F_\mathrm{res}^+ [v] := -A(\partial_t, \partial_y) v - \frac{3}{2 (t+T)} v +  \frac{3}{2 \eta_* (t+T)} v_y - e^{\eta_* y} N(e^{-\eta_* y} v) = 0. \label{e: v eqn pencil form}
\end{align}
We will see that the first two terms are dominant, and determine the leading order behavior of our diffusive tail. 

To transform into the normal form constructed above, we choose $Q$ and $S$ as in Lemma \ref{l: normal form colinear}, define $\Phi = Q^{-1} v$, apply $S$ to both sides of \eqref{e: v eqn pencil form}, and thereby see that for $y \geq 1$, $\Phi$ satisfies
\begin{align}
\tilde{F}_\mathrm{res}^+ [\Phi] := - B(\partial_t, \partial_y) \Phi 
 - \frac{3}{2 (t+T)} SQ \Phi + \frac{3}{2 \eta_* (t+T)} SQ \Phi_y - S e^{\eta_* y} N(e^{-\eta_* y} Q\Phi) = 0. \label{e: colinear Phi eqn}
\end{align}
To isolate the leading order diffusive dynamics, we let $\Phi = (\phi^\mathrm{I}, \phih)^T \in \R \times \R^{n-1}$. By Lemma \ref{l: normal form colinear}, the equation \eqref{e: colinear Phi eqn} then has the form
\begin{multline}
\begin{pmatrix}
\partial_t - b_{11}^{02} \partial_y^2 & -b_{12} (\partial_t, \partial_y) \\
-b_{21} (\partial_t, \partial_y) & -b_{22} (\partial_t, \partial_y) 
\end{pmatrix}
\begin{pmatrix}
\phiI \\ \phi^\mathrm{h} 
\end{pmatrix}
- \frac{3}{2 (t+T)} SQ \begin{pmatrix} \phiI \\ \phi^\mathrm{h} \end{pmatrix} + \frac{3}{2 \eta_* (t+T)} SQ \begin{pmatrix}\phiI_y \\ \phi^\mathrm{h}_y \end{pmatrix} \\ - S e^{\eta_*y} N \left( e^{-\eta_*y} Q \begin{pmatrix} \phiI \\ \phi^\mathrm{h} \end{pmatrix}  \right) = 0. \label{e: colinear system}
\end{multline}
The leading order dynamics are governed by the terms $\partial_t \phiI - b_{11}^{02} \partial_y^2 \phiI - \frac{3}{2 (t+T)} P^\mathrm{I} SQ (\phiI, 0)^T$, where $P^\mathrm{I} (f_1, ..., f_n)^T = f_1$. However, in order to construct an approximate solution with residual error sufficiently small to close a nonlinear stability argument, we must explicitly capture the effects of several higher order terms. To separate relevant from irrelevant terms, we define, using the expansions from Lemma \ref{l: normal form colinear},
\begin{align}
\tilde{b}_{12} (\partial_t, \partial_t) &= b_{12}(\partial_t, \partial_x) - b_{12}^{01} \partial_x = \mathrm{O}(\partial_t, \partial_x^2), \\
\tilde{b}_{22} (\partial_t, \partial_x) &= b_{22}(\partial_t, \partial_x) - b_{22}^{00} - b_{22}^{01} \partial_x = \mathrm{O}(\partial_t, \partial_x^2). 
\end{align}

We now heuristically explain which terms are relevant. Using that $P^\mathrm{I} SQ e_0 = 1$, we see that the leading order equation has the form
\begin{align*}
\phiI_t = \Deff \phiI + \frac{3}{2(t+T)} \phiI.
\end{align*}
Using self-similar variables to compute asymptotics of $\phiI$, we will see that this equation has a solution of the form
\begin{align}
\phiI \sim (y+y_0) e^{-(y+y_0)^2/(4 \Deff (t+T))}. \label{e: colinear formal leading order}
\end{align}
To close a nonlinear stability argument as in \cite{CAMS}, we need the residual error from inserting an approximate solution $\Phi^+$ into $\tilde{F}_\mathrm{res}^+$ to satisfy
\begin{align}
\sup_{y \in \R} \left| \langle y \rangle^{2 + \mu} \tilde{F}_\mathrm{res}^+ [\Phi^+] (y,t) \right| \leq \frac{C}{(t+T)^a} \label{e: residual desired}
\end{align}
for some $\mu, a > 0$, as in \cite[Lemma 2.4]{CAMS}. From \eqref{e: colinear system}, we expect that at leading order 
\begin{align}
\phih \sim -(b_{22}^{00})^{-1} b_{21} (\partial_t, \partial_y) \phiI. \label{e: colinear phih formal leading order}
\end{align}
From Lemma \ref{l: normal form colinear}, we know that $b_{21}(\partial_t, \partial_y) = \mathrm{O}(\partial_t, \partial_y^2)$. Note that if $\phiI$ has the form \eqref{e: colinear formal leading order}, then $\partial_t \phiI \sim \partial_y^2 \phiI$ on the diffusive length scale $y \sim \sqrt{t+T}$, and from a short calculation we find
\begin{align}
\langle y \rangle^{2 + \mu} \partial_y^k \phiI (y,t) \sim \frac{\langle y \rangle^{2 + \mu}}{(t+T)^\frac{k-1}{2}} e^{-(y+y_0)^2/(4 \Deff (t+T))} \sim (t+T)^\frac{3-k+\mu}{2}. 
\end{align}
Hence $\partial_y^k \phiI$ does not satisfy the estimate \eqref{e: residual desired} if $k \leq 3$. We must therefore keep track of all terms involving derivatives up to order 3 of $\phiI$  which are generated upon inserting the leading order asymptotics \eqref{e: colinear formal leading order} into \eqref{e: colinear system}. For instance, according to \eqref{e: colinear phih formal leading order}, we expect that $\phih \sim \partial_y^2 \phiI$, but inserting this into \eqref{e: colinear system} generates  additional terms of the form
\begin{align}
b_{22}(\partial_t, \partial_y) \phih \sim b_{22}^{00} \partial_y^2 \phiI + b_{22}^{01} \partial_y^3 \phiI + \mathrm{O}(\partial_t, \partial_y^2) \partial_y^2 \phiI,
\end{align}
the first two of which we must track explicitly, while the last term is sufficiently decaying to be included in a residual term satisfying an estimate of the form \eqref{e: residual desired}. 

\begin{remark}\label{rmk: parabolic}
	When considering higher-order parabolic systems of the form \eqref{e: parabolic}, the operators $b_{ij}(\partial_t, \partial_y)$ themselves contain terms with three or more space derivatives in them, and we must for instance expand
	\begin{align}
	b_{11}(\partial_t, \partial_y) = \partial_t - b_{11}^{02} \partial_y^2 - b_{11}^{03} \partial_y^3 + \mathrm{O}(\partial_y^4)
	\end{align}
	and keep track of the third derivative term explicitly, since $\langle y \rangle^{2 + \mu} \partial_y^3 \phiI$ is not decaying in time. Keeping track of these terms is the only modification needed to treat higher order parabolic systems. 
\end{remark}

We translate these heuristics into a precise, systematic argument using self-similar variables. First we rewrite \eqref{e: colinear system} slightly. We define $P^\mathrm{I}: \R^n \to \R$ and $P^\mathrm{h}:\R^n \to \R^{n-1}$ by 
\begin{align}
P^\mathrm{I} \begin{pmatrix} f_1 \\ \vdots \\ f_n \end{pmatrix} = f_1, \quad P^\mathrm{h} \begin{pmatrix} f_1, \\ \vdots\\  f_n \end{pmatrix} = \begin{pmatrix}f_2 \\ \vdots \\ f_n \end{pmatrix},
\end{align}
and note that by Lemma \ref{l: normal form colinear} 
\begin{align}
\PI S Q \begin{pmatrix} \phiI \\ \phih \end{pmatrix} = \phiI + s_{12}^T \phih,
\end{align}
and 
\begin{align}
\Ph SQ \begin{pmatrix} \phiI \\ \phih \end{pmatrix} = s_{21} \phiI + s_{22} \phih 
\end{align}
for some $s_{12}, s_{21} \in \R^{n-1}$, and $s_{22} \in \R^{n-1} \times \R^{n-1}$. We then write the system \eqref{e: colinear system} as the coupled equations
\begin{align}
\partial_t \phiI &= \Deff \partial_y^2 \phiI + \frac{3}{2 (t+T)} \phiI + \FI_1 (\phiI, \phih, y, t+T) + \FI_2  (\phiI, \phih, y, t+T) \label{e: colinear phiI eqn}\\
b_{22}^{00} \phih &= \left( - b_{21}^{10} \partial_t - b_{21}^{02} \partial_y^2 - \frac{3}{2 (t+T)} s_{21} \right) \phiI + \Fh_1  (\phiI, \phih, y, t+T) + \Fh_2  (\phiI, \phih, y, t+T), \label{e: colinear phih eqn}
\end{align}
where $\Deff = b_{11}^{02}$, and
\begin{align}
\FI_1  (\phiI, \phih, y, t+T) &= b_{12}^{01} \partial_y \phih - \frac{3}{2 \eta_* (t+T)} \phiI_y, \\
\FI_2  (\phiI, \phih, y, t+T) &= \left( \tilde{b}_{12} (\partial_t, \partial_y)  + \frac{3}{2(t+T)} s_{12}^T - \frac{3}{2 \eta_* (t+T)} s_{12}^T \partial_y \right) \phih + \PI S e^{\eta_* y} N \left( e^{-\eta_* y} Q \Phi \right), \\
\Fh_1  (\phiI, \phih, y, t+T) &= - b_{22}^{01} \partial_y \phih + \frac{3}{2 \eta_* (t+T)} s_{21} \partial_y \phiI,
\end{align}
and
\begin{align}
\Fh_2  (\phiI, \phih, y, t+T) &= \left( - \tilde{b}_{22} (\partial_t, \partial_y) - \frac{3}{2(t+T)} s_{22} + \frac{3}{2 \eta_* (t+T)} s_{22} \partial_y \right) \phih + \Ph S e^{\eta_* y} N \left(e^{-\eta_* y} Q \Phi \right). 
\end{align}
We have written the equations in this form since the first, explicit terms in \eqref{e: colinear phiI eqn}-\eqref{e: colinear phih eqn} represent the leading order parts in the equation, while $\mathcal{F}^\mathrm{I/h}_{j}, j = 1,2$ represent successively higher order corrections. 

We now introduce the scaling variables
\begin{align}
\tau = \log (t+T), \quad \xi = \frac{1}{\sqrt{\Deff}} \frac{y+y_0}{\sqrt{t+T}}
\end{align}
where $y_0 \in \R$ will be chosen later. We define
\begin{align}
\Psi(\xi, \tau) = \begin{pmatrix}
\psiI (\xi, \tau) \\ \psih(\xi, \tau) 
\end{pmatrix}
:= 
\begin{pmatrix}
\phiI(y, t) \\ \phih(y, t) 
\end{pmatrix}
 = \Phi(y,t). 
\end{align}
The new unknowns $\psiI, \psih$ solve the system
\begin{align}
\partial_\tau \psiI &= \partial_\xi^2 \psiI + \frac{1}{2} \xi \partial_\xi \psiI + \frac{3}{2} \psiI + \tfi_1 + \tfi_2 \label{e: colinear psiI eqn} \\
b_{22}^{00} e^\tau \psih &= \left( -b_{21}^{10} (\partial_\tau - \frac{1}{2} \xi \partial_\xi) - b_{21}^{02} \Deff^{-1} \partial_\xi^2 - \frac{3}{2} s_{21} \right) \psiI + \tfh_1 + \tfh_2. \label{e: colinear psih eqn}
\end{align}
The higher order correction terms $\tilde{\mathcal{F}}^\mathrm{I/h}_j$ depend on $(\psiI, \psih, \xi, \tau)$, but we will suppress the dependence on $\xi$ and $\tau$. They are given by 
\begin{align}
\tfi_1 (\psiI, \psih) &= b_{12}^{01} e^{\tau/2} \Deff^{-1/2} \partial_\xi \psih - \frac{3}{2 \eta_*} e^{-\tau/2} \Deff^{-1/2} \partial_\xi \psiI, \label{e: colinear tfi 1 def}
\end{align}
\begin{multline}
\tfi_2 (\psiI, \psih) = \left( e^\tau \tilde{b}_{12} \left(e^{-\tau} (\partial_\tau - \frac{1}{2} \xi \partial_\xi), e^{-\tau/2} \Deff^{-1/2} \partial_\xi \right) + \frac{3}{2} s_{12}^T - \frac{3}{2 \eta_*} e^{-\tau/2} s_{12}^T \Deff^{-1/2} \partial_\xi \right) \psih \\ + e^\tau \tilde{N}^\mathrm{I} (\xi, \tau, \psiI, \psih), 
\end{multline}
\begin{align}
\tfh_1 (\psiI, \psih) &= - b_{22}^{01} e^{\tau/2} \Deff^{-1/2} \partial_\xi \psih + \frac{3}{2 \eta_*} s_{21} e^{-\tau/2} \Deff^{-1/2} \partial_\xi \psiI, \label{e: colinear tfh 1 def}
\end{align}
and
\begin{multline}
\tfh_2 (\psiI, \psih) = \left( - e^\tau \tilde{b}_{22} \left( e^{-\tau} (\partial_\tau - \frac{1}{2} \xi \partial_\xi), e^{-\tau/2} \Deff^{-1/2} \partial_\xi\right) - \frac{3}{2} s_{22} + \frac{3}{2 \eta_*} s_{22} e^{-\tau/2} \Deff^{-1/2} \partial_\xi\right) \psih \\
+ e^\tau \tilde{N}^\mathrm{h} (\xi, \tau, \psiI, \psih),
\end{multline}
where
\begin{align}
\tilde{N}^\mathrm{I/h} (\xi, \tau, \psiI, \psih) = P^\mathrm{I/h} S e^{\eta_* y(\xi, \tau)} N(e^{-\eta_* y(\xi,\tau)} Q \Psi),
\end{align}
are the nonlinear terms, with 
\begin{align}
y(\xi, \tau) = \sqrt{\Deff} e^{\tau/2} \xi - y_0. 
\end{align}
The nonlinear terms are irrelevant in the regime we are interested in for constructing the diffusive tail --- see Lemma \ref{l: colinear nonlinearity estimates}. 

We emphasize that in deriving the system \eqref{e: colinear psiI eqn}-\eqref{e: colinear psih eqn}, we have multiplied both equations by $e^\tau$, so we will need to keep track of this factor when comparing residual errors of approximate solutions in the scaling and original coordinates. We uncover the leading order dynamics of \eqref{e: colinear psiI eqn}-\eqref{e: colinear psih eqn} by inserting the ansatz
\begin{align}
\psiI(\xi, \tau) &= e^{\tau/2} \psiI_0 (\xi) + \psiI_1 (\xi), \label{e: colinear ansatz 1} \\
\psih(\xi, \tau) &= e^{-\tau/2} \psih_0 (\xi) + e^{-\tau}\psih_1 (\xi). \label{e: colinear ansatz 2}
\end{align}
The ansatz for $\psiI$ has precisely the form used for scalar equations in \cite{CAMS, NRRrefined}, while $\psih$ is a new term which captures the higher order but still relevant dynamics resulting from the interactions between diffusive and strongly stable modes in the multi-component case. Note that the form of the ansatz captures the intuition discussed above that $\phih \sim \partial_y^2 \psiI$, and so decays faster by a factor of $(t+T)^{-1} = e^{-\tau}$. 

Inserting this ansatz into \eqref{e: colinear psiI eqn}-\eqref{e: colinear psih eqn}, we separate terms according to their order in powers of $e^{\tau/2}$. The leading order terms are on the order $e^{\tau/2}$, and keeping only these terms results in the system
\begin{align}
0 &= \partial_\xi^2 \psiI_0 + \frac{1}{2} \xi \partial_\xi \psiI_0 + \psiI_0 \label{e: colinear psi0 eqn}\\
b_{22}^{00} \psih_0 &= \left( -\frac{1}{2}b_{21}^{10} (1 - \xi \partial_\xi) - b_{21}^{02} \Deff^{-1} \partial_\xi^2 - \frac{3}{2} s_{21} \right) \psiI_0 \label{e: colinear psih0 eqn}. 
\end{align}
Since the state selected in the wake of the front is exponentially stable, the bulk of the front creates a strong absorption mechanism when viewed from the perspective of the leading edge. Therefore, we consider the equations for $\psiI_j$ and  $\psih_j,$ on the half-line $\xi > 0$ with homogeneous Dirichlet boundary condition $\psiI_j (0) = \psih_j (0) = 0$, as in \cite{CAMS, NRRrefined}. 

We define $L_\Delta = \partial_\xi^2 + \frac{1}{2} \xi \partial_\xi + 1$, as an operator on $\{ \xi > 0 \}$ with homogeneous Dirichlet boundary condition. Note that $L_\Delta$ is invariant under the reflection $\xi \mapsto -\xi$, and so imposing a homogeneous Dirichlet boundary condition is equivalent to restricting to odd functions. The spectrum of $L_\Delta$ is well-known: see for instance \cite[Appendix A]{GallayWayne}, or note that conjugating with weight $e^{\xi^2/8}$ transforms $L_\Delta$ into the quantum harmonic oscillator, with well known spectral properties; see e.g. \cite{Helffer}. In particular, $\lambda = 0$ is an eigenvalue to $L_\Delta$ on $L^2_\mathrm{odd}(\R)$, with eigenfunction 
\begin{align}
\psiI_0 (\xi) = \beta_0 \xi e^{-\xi^2/4} \label{e: colinear psi0 formula}
\end{align}
uniquely solving \eqref{e: colinear psi0 eqn} up to the arbitrary constant $\beta_0 \in \R$ which we leave free for now. Having solved for $\psiI_0$, leaving $\beta_0$ to be chosen later, we may then invert $b_{22}^{00}$ by Lemma \ref{l: normal form colinear} to solve \eqref{e: colinear psih0 eqn} for $\psih_0$, obtaining
\begin{align}
\psih_0 (\xi) = (b_{22}^{00})^{-1} \left( -\frac{1}{2}b_{21}^{10} (1 -  \xi \partial_\xi) - b_{21}^{02} \Deff^{-1} \partial_\xi^2 - \frac{3}{2} s_{21} \right) \psiI_0 (\xi). \label{e: colinear psih 0 formula}
\end{align}
In particular, $\psih_0$ satisfies the following elementary estimates.
\begin{lemma}\label{l: colinear psih 0 estimates}
	Let $\psiI_0$ be defined by \eqref{e: colinear psi0 formula} and $\psih_0$ be defined by \eqref{e: colinear psih 0 formula}. For each natural number $k$, there exists a constant $C_k > 0$ such that 
	\begin{align}
	|\partial_\xi^k \psih_0(\xi) | \leq C_k e^{-\xi^2/5}
	\end{align}
	for all $\xi \in \R$. 
\end{lemma}

We now collect the next order terms in $e^{\tau/2}$ resulting from inserting the ansatz \eqref{e: colinear ansatz 1}-\eqref{e: colinear ansatz 2} into the system \eqref{e: colinear psiI eqn}-\eqref{e: colinear psih eqn}, which are terms of order 1. We obtain the system
\begin{align}
\left(L_\Delta + \frac{1}{2} \right) \psiI_1 &= - \tfi_1 (e^{\tau/2} \psiI_0, e^{-\tau/2} \psih_0) =: \GI_1 (\xi). \label{e: colinear psiI 1 eqn} \\
b_{22}^{00} \psih_1 &= \left( \frac{1}{2}b_{21}^{10} \xi \partial_\xi - b_{21}^{02} \Deff^{-1} \partial_\xi^2 - \frac{3}{2} s_{21} \right) \psiI_1 + \tfh_1 (e^{\tau/2} \psiI_0, e^{-\tau/2} \psih_0) \label{e: colinear psih 1 eqn}
\end{align}
Note from the form of $\tfi_1$ in \eqref{e: colinear tfi 1 def} that the right hand side of \eqref{e: colinear psiI 1 eqn} is independent of $\tau$, so we denote it by $\GI_1(\xi)$ as above. It follows from Lemma \ref{l: colinear psih 0 estimates} and \eqref{e: colinear psi0 formula} that there exists a constant $C > 0$ such that
\begin{align}
|\GI_1(\xi)| \leq C e^{-\xi^2/6}. 
\end{align}
Note that we widen the variance of the Gaussian to absorb any polynomially growing terms resulting from repeated differentiation. 

We can solve \eqref{e: colinear psiI 1 eqn} provided $\lambda = -\frac{1}{2}$ is not an eigenvalue of $L_\Delta$. According to \cite[Appendix A]{GallayWayne}, $\lambda = -\frac{1}{2}$ is an eigenvalue to $L_\Delta$ on $L^2 (\R)$. However, the corresponding eigenfunction is even, so restricting to the odd subspace ---  or equivalently considering the equations on $\xi > 0$ with homogeneous Dirichlet boundary condition --- eliminates this eigenvalue. We may therefore solve \eqref{e: colinear psiI 1 eqn} for $\psiI_1$, using also that the right hand side is strongly localized so that the essential spectrum of $L_\Delta$ does not interfere. Conjugating with the Gaussian weight $e^{\xi^2/8}$ and exploiting properties of the quantum harmonic oscillator \cite{Helffer}, we further obtain Gaussian estimates on the solution and its derivatives, which we state in the following lemma. 
\begin{lemma}\label{l: colinear psiI 1}
	Let $\psiI_0$ be given by \eqref{e: colinear psi0 formula} and $\psih_0$ be given by \eqref{e: colinear psih 0 formula}. Then there exists a smooth solution $\psiI_1$ to \eqref{e: colinear psiI 1 eqn} such that for each natural number $k$, we have 
	\begin{align}
	|\partial_\xi^k \psiI_1 (\xi) | \leq C_k e^{-\xi^2/8} 
	\end{align}
	for all $\xi > 0$. Moreover, all derivatives of $\psiI_1$ extend continuously to $\xi = 0$. 
\end{lemma}
We note that this is just Lemma 2.3 of \cite{CAMS}, but with the additional $\psih_0$ terms present in \eqref{e: colinear psiI 1 eqn}. With $\psiI_1$ in hand, we may solve for $\psih_1$ by simply inverting the matrix $b_{22}^{00}$. Note from the form \eqref{e: colinear tfh 1 def} of $\tfh_1$ that
\begin{align}
\Gh_1(\xi) := \tfh_1 (e^{\tau/2} \psiI_0(\xi), e^{-\tau/2} \psih_0 (\xi)) 
\end{align}
does not depend on $\tau$, and so the whole right hand side of \eqref{e: colinear psih 1 eqn} is independent of $\tau$, and hence we find a solution $\psih_1(\xi)$ which depends only on $\xi$ and which satisfies the following elementary estimates.
\begin{lemma}\label{l: colinear psih 1}
	Let $\psiI_0, \psih_0,$ and $\psiI_1$ be given by \eqref{e: colinear psi0 formula}, \eqref{e: colinear psih 0 formula}, and Lemma \ref{l: colinear psiI 1}, respectively. Then there exists a smooth solution $\psih_1$ to \eqref{e: colinear psih 1 eqn} such that for each natural number $k$, we have 
	\begin{align}
	| \partial_\xi^k \psih_1 (\xi) | \leq C_k e^{-\xi^2/8}
	\end{align}
	for all $\xi > 0$. Moreover, all derivatives of $\psih_1$ extend continuously to $\xi = 0$. 
\end{lemma}

Having now chosen $\psi^\mathrm{I/h}_j, j = 0, 1$, we show that this construction produces an approximate solution of \eqref{e: colinear psiI eqn}-\eqref{e: colinear psih eqn} whose residual error will satisfy \eqref{e: residual desired} once we revert to $(y,t)-$coordinates. We define the nonlinear evaluation operator, with $\Psi = (\psiI, \psih)^T$, 
\begin{align}
\mcf_\mathrm{res} [\Psi] = \begin{pmatrix} \partial_\tau \psiI - \partial_\xi^2 \psiI - \frac{1}{2} \xi \partial_\xi \psiI - \frac{3}{2} \psiI - \tfi_1(\psiI, \psih) - \tfi_2 (\psiI, \psih) \\ b_{22}^{00} e^\tau \psih - \left( -b_{21}^{10} (\partial_\tau - \frac{1}{2} \xi \partial_\xi) - b_{21}^{02} \Deff^{-1} \partial_\xi^2 - \frac{3}{2} s_{21} \right) \psiI - \tfh_1(\psiI, \psih) - \tfh_2(\psiI, \psih),
\end{pmatrix}
\end{align}
so that exact solutions to \eqref{e: colinear psiI eqn}-\eqref{e: colinear psih eqn} satisfy $\mcf_\mathrm{res} [\Psi] = 0$. We now estimate the contributions from the nonlinear terms $\tilde{N}^\mathrm{I/h} (\xi, \tau, \psiI, \psih)$. We will only use the diffusive tail in our construction of an approximate solution for $y \geq (t+T)^\mu$, since we will match with the bulk of the front on this length scale. Hence, when estimating the nonlinear contributions here, we may restrict to the corresponding region in $\xi$, given by \eqref{e: xi restriction}. 
\begin{lemma}\label{l: colinear nonlinearity estimates}
	Fix $0 < \mu < \frac{1}{8}$. Assume there is a constant $C_2$ such that $|\psiI|+|\psih| \leq C_1$. Then there exists a constant $C_2$ such that 
	\begin{align}
	e^\tau |\tilde{N}^I (\xi, \tau, \psiI, \psih)| + e^\tau |\tilde{N}^\mathrm{h} (\xi, \tau, \psiI, \psih)| \leq C_2 e^{- 2 \tau} |\psiI + \psi_h|^2 \label{e: scaling variables nonlinearity estimate}
	\end{align}
	for all $\xi$ and $\tau> 0$ such that 
	\begin{align}
	\xi \geq \frac{1}{\sqrt{D_\mathrm{eff}}} \left( e^{(\mu-1/2) \tau} + y_0 e^{-\tau/2} \right). \label{e: xi restriction}
	\end{align}
\end{lemma}
\begin{proof}
	Note that $N$ is quadratic in its argument, so we have 
	\begin{align*}
	| e^{\eta_* y(\xi, \tau)} N (e^{-\eta_* y(\xi, \tau)} Q(\psiI,\psih)^T) | \leq C e^{-\eta_* y(\xi, \tau)} |Q (\psiI, \psih)^T|^2 \leq C e^{- \eta_* y (\xi, \tau)} |\psi+ \psih|^2. 
	\end{align*}
	The condition \eqref{e: xi restriction} implies that $y(\xi, \tau) \geq e^{\mu \tau}$, and so we have 
	\begin{align*}
	| e^{\eta_* y(\xi, \tau)} N (e^{-\eta_* y(\xi, \tau)} Q(\psiI,\psih)^T) | \leq C \exp \left[ - \eta_* e^{\mu \tau} \right] |\psiI + \psih|^2. 
	\end{align*} 
	Since this decay is super-exponential in $\tau$, the extra factor of $e^\tau$ in \eqref{e: scaling variables nonlinearity estimate} may be easily absorbed, leading to the desired estimate. 
\end{proof}

We are now ready to estimate the full residual error from constructing the diffusive tail according to the ansatz \eqref{e: colinear ansatz 1}-\eqref{e: colinear ansatz 2}. 
\begin{prop}[Diffusive tail in self-similar coordinates --- colinear case]
	Fix $0 < \mu < \frac{1}{8}$. Let $\psiI_0, \psih_0, \psiI_1,$ and $\psih_1$ be defined as in \eqref{e: colinear psi0 formula}, \eqref{e: colinear psih 0 formula}, Lemma \ref{l: colinear psiI 1}, and Lemma \ref{l: colinear psih 1}, respectively, and let 
	\begin{align}
	\Psi(\xi, \tau) = \begin{pmatrix}
	\psiI (\xi, \tau) \\ \psih(\xi, \tau)
	\end{pmatrix} = \begin{pmatrix}
	e^{\tau/2} \psiI_0 (\xi) + \psiI_1 (\xi) \\
	e^{-\tau/2} \psih_0(\xi) + e^{-\tau} \psih_1(\xi)
	\end{pmatrix}. 
	\end{align}
	Then there exists a constant $C > 0$ such that 
	\begin{align}
	| \mcf_\mathrm{res}[\Psi] (\xi, \tau)| \leq C e^{-\tau/2} e^{-\xi^2/8} \label{e: colinear Psi residual estimate}
	\end{align}
	for all $\tau > 0$ and all
	\begin{align}
	\xi \geq \frac{1}{\sqrt{\Deff}} \left( e^{(\mu-1/2) \tau} + y_0 e^{-\tau/2} \right). \label{e: tail xi restriction}
	\end{align}
\end{prop}
\begin{proof}
	By Lemma \ref{l: colinear nonlinearity estimates}, the formula \eqref{e: colinear psi0 formula} for $\psiI_0$, and the estimates on $\psiI_1, \psih_0, \psih_1$ from Lemmas \ref{l: colinear psih 0 estimates} through \ref{l: colinear psih 1}, we have 
	\begin{align*}
		e^\tau |\tilde{N}^I (\xi, \tau, \psiI, \psih)| + e^\tau |\tilde{N}^\mathrm{h} (\xi, \tau, \psiI, \psih)| \leq C e^{-\tau} e^{-\xi^2/8}
	\end{align*} 
	for some constant $C > 0$. The asymptotic expansions leading to the choice of $\psi_j$ and $\psih_j$ carried out above eliminates all terms of order $e^{\tau/2}$ and order $1$ from $\mcf_\mathrm{res}[\Psi]$. The only remaining terms are then bounded by $C e^{-\tau/2}( p(\xi) e^{-\xi^2/4} + e^{-\xi^2/8})$, where $p(\xi)$ is a polynomial resulting from differentiating $\psi_0 (\xi) = \beta_0 \xi e^{-\xi^2/8}$. Since there is a constant $C > 0$ such that $|p(\xi) e^{-\xi^2 /4}|\leq C e^{-\xi^2/8}$, the result follows. 
\end{proof}

We now translate this approximate solution and estimate on the residual error back into $(x,t)$ coordinates. 

\begin{corollary}[Diffusive tail --- normal form coordinates] \label{c: colinear normal form residual}
	Fix $0 < \mu < \frac{1}{8}$ and $\beta_0, y_0 \in \R$ and define 
	\begin{align}
	\phiIp(y,t) &= (t+T)^{1/2} \psiI_0 \left( \frac{1}{\sqrt{\Deff}} \frac{y+y_0}{\sqrt{t+T)}} \right) + \psiI_1 \left( \frac{1}{\sqrt{\Deff}} \frac{y+y_0}{\sqrt{t+T}} \right), \\
	\phihp(y,t) &= (t+T)^{-1/2} \psih_0  \left( \frac{1}{\sqrt{\Deff}} \frac{y+y_0}{\sqrt{t+T}} \right) + (t+T)^{-1} \psih_1  \left( \frac{1}{\sqrt{\Deff}} \frac{y+y_0}{\sqrt{t+T}} \right)
	\end{align}
	Then with $\Phi^+ = (\phiIp, \phihp)$ and $\tilde{F}^+_\mathrm{res} [\Phi]$ defined as in \eqref{e: colinear Phi eqn}, there exists a constant $C > 0$ such that for all $t > 0$ and all $y \geq (t+T)^\mu$, we have 
	\begin{align}
	| \tilde{F}^+_\mathrm{res} [\Phi^+] (y,t) | \leq \frac{C}{(t+T)^{3/2}} e^{-(y+y_0)^2/[8 \Deff (t+T)]}. 
	\end{align}
\end{corollary}
\begin{proof}
	Note that in transforming from the $\Phi$ equation to the ultimate system \eqref{e: colinear psiI eqn}-\eqref{e: colinear psih eqn} for $\Psi$, we multiplied by a factor of $e^\tau$. Hence in undoing this transformation we regain a factor of $e^{-\tau} = (t+T)^{-1}$ compared to the estimate \eqref{e: colinear Psi residual estimate}. 
\end{proof}

We now finally undo the normal form transformation $\Phi = Q^{-1} v$ to obtain estimates in the original coordinates. 
\begin{corollary}\label{c: colinear vplus residual}
	Fix $0 < \mu < \frac{1}{8}$ and $\beta_0, y_0 \in \R$ and define $v^+(y,t) = Q \Phi^+(y,t)$, with $\Phi^+$ chosen as in Corollary \ref{c: colinear normal form residual}. There exists a constant $C > 0$ such that for $t > 0$ and $y \geq (t+T)^{\mu}$, we have 
	\begin{align}
	|F^+_\mathrm{res}[v^+] (y,t)| \leq \frac{C}{(t+T)^{3/2}} e^{-(y+y_0)^2/[8\Deff (t+T)]}.
	\end{align}
\end{corollary}
\begin{proof}
	Note that $F_\mathrm{res}^+ [v] = S^{-1} \tilde{F}_\mathrm{res}^+ [\Phi]$. The result then readily follows from Corollary \ref{c: colinear normal form residual}. 
\end{proof}

\subsection{Matching the diffusive tail to the front}\label{s: colinear matching}
With a detailed description of the diffusive dynamics in the leading edge in hand, we now match this diffusive tail to the invasion front in the wake, obtaining an approximate solution which will govern the propagation dynamics. The argument of this section refines that of \cite[Section 2.3]{CAMS}, following that argument but with two main differences. The first is that we simplify the proof by omitting the additional shift $\zeta(t+T)$ used in \cite{CAMS}, which turns out to be unnecessary for closing the argument. The other difference is that we correct a small gap in the proof of \cite[Proposition 2.5]{CAMS}, where equation (2.39) treats the evaluation operator $F_\mathrm{res} [v]$ as though it were linear. To be completely accurate, $F_\mathrm{res}[v]$ is nonlinear but the effect of the nonlinearity is exponentially small in $(t+T)$ in the relevant regime $(t+T)^\mu \leq y \leq (t+T)^\mu + 1$, so that the same result still holds. A corrected argument is given in the proof of Proposition \ref{p: colinear residual} below. 

For matching with the diffusive tail in the leading edge, we define 
\begin{align}
v^-(y) = \omega(y) q_*(y). 
\end{align}

\begin{lemma}[Pointwise matching]\label{l: colinear pointwise matching}
	Fix $0 < \mu < \frac{1}{8}$, and let $\beta_0 = \sqrt{\Deff}$ and $y_0 = 1+a$. Then there exists a constant $C >0$ such that for all $T$ sufficiently large and $t>0$, we have
	\begin{align}
	|v^- ((t+T)^\mu) - v^+ ((t+T)^\mu, t) | \leq \frac{C}{(t+T)^{\mu-1/2}}. \label{e: colinear matching estimate}
	\end{align}
	Moreover, the derivatives satisfy, for $T$ sufficiently large and $t > 0$,
	\begin{align}
	| v^-_y((t+T)^\mu) - v^+_y((t+T)^\mu, t) | \leq \frac{C}{(t+T)^{\mu-1/2}}, \label{e: colinear matching derivative estimate}
	\end{align}
	and for each integer $k \geq 2$, there is a constant $C_k > 0$ such that
	\begin{align}
	|\partial_y^k v^+ (y, t) 1_{\{(t+T)^\mu \leq y \leq (t+T)^\mu +1\}}| &\leq \frac{C_k}{(t+T)^{\mu- 1/2}}, \label{e:colinear higher derivative estimate 1}\\
	|\partial_y^k v^- (y)  1_{\{(t+T)^\mu \leq y \leq (t+T)^\mu +1\}}| &\leq \frac{C_k}{(t+T)^{\mu-1/2}}. \label{e: colinear higher derivative estimate 2}
	\end{align}
\end{lemma}
\begin{proof}
	Recall from Corollary \ref{c: colinear vplus residual} that $v^+(y,t) = Q \Phi^+(y,t)$. We split $v^+$ as
	\begin{align*}
	v^+(y,t) =  Q \begin{pmatrix} \phiIp(y,t) \\ 0 \end{pmatrix} + Q \begin{pmatrix} 0 \\ \phihp(y,t) \end{pmatrix}. 
	\end{align*}
	Recall from Lemma \ref{l: normal form colinear} that $Q e_0 = u_0$, where $e_0 = (1, 0, ..., 0)^T$, and so we have 
	\begin{align*}
	Q \begin{pmatrix} \phiIp(y,t) \\ 0 \end{pmatrix} = u_0 \phi^+(y,t). 
	\end{align*}
	Using Corollary \ref{c: colinear normal form residual} and the expression \eqref{e: colinear psi0 formula} to evaluate $\phiIp(y,t)$ at $y = (t+T)^\mu$, we obtain 
	\begin{multline*}
	u_0 \phiIp ((t+T)^\mu, t) =  u_0 \frac{\beta_0}{\sqrt{\Deff}} ((t+T)^\mu + y_0)\exp \left[-\frac{1}{4 \Deff} (t+T)^{-1} \left( (t+T)^{2\mu} + 2 y_0 (t+T)^{\mu} + y_0^2  \right) \right] \\ + u_0 \psiI_1 \left( \frac{1}{\sqrt{\Deff}} \frac{(t+T)^\mu + y_0}{\sqrt{t+T}} \right). 
	\end{multline*}
	Taylor expanding the exponential and $\psiI_1$, using the fact that $\psiI_1(0) = 0$, we obtain 
	\begin{align*}
	u_0 \phiIp ((t+T)^\mu, t) &= u_0 \frac{\beta_0}{\sqrt{\Deff}} ((t+T)^\mu + y_0) \left( 1 + \mathrm{O}((t+T)^{2\mu-1}) \right) + \mathrm{O} ((t+T)^{\mu-1/2}) \\
	&= u_0 \frac{\beta_0}{\sqrt{\Deff}} [ (t+T)^\mu + y_0 ] + \mathrm{O}((t+T)^{\mu-1/2}),
	\end{align*}
	using also that $3 \mu -1 < \mu - 1/2$ since $\mu < 1/8$. This leading term is the term we will use to match with $v^-$. Using the front asymptotics \eqref{e: front asymptotics}, we have
	\begin{align}
	v^-((t+T)^\mu) = u_0 [ (t+T)^\mu + (1 + a) u_0] + \mathrm{O}\left(e^{-\eta_0 (t+T)^\mu} \right), \label{e: colinear v minus expansion}
	\end{align}
	recalling that we are assuming $u_0 = u_1$ in this section. Therefore, choosing $\beta_0 = \sqrt{\Deff}$ and $y_0 = 1+a$, we have 
	\begin{align}
	u_0 \phiIp ((t+T)^\mu, t) - v^-((t+T)^\mu) = \mathrm{O} ((t+T)^{\mu-1/2}). \label{e: colinear leading edge matching 1}
	\end{align}
	It follows from Corollary \ref{c: colinear normal form residual} and Lemma \ref{l: colinear psih 0 estimates} that 
	\begin{align}
	\left| \partial_y^k Q \begin{pmatrix} 0 \\ \phihp ((t+T)^\mu, t) \end{pmatrix}  \right| \leq C
	(t+T)^{-1/2}, \label{e: colinear leading edge matching 2}
	\end{align}
	for $k = 0, 1,$ or 2, and together \eqref{e: colinear leading edge matching 1} and \eqref{e: colinear leading edge matching 2} for $k = 0$ imply \eqref{e: colinear matching estimate}. 
	
	Having chosen $\beta_0$ and $y_0$, we now evaluate the derivatives of $v^\pm$ at $y = (t+T)^\mu$ and see that 
	\begin{align*}
	u_0 \phi^+_y((t+T)^\mu, t) = 1 + \mathrm{O}((t+T)^{\mu-1/2}), \quad v^-((t+T)^\mu) = 1 + \mathrm{O}(e^{-\eta_* (t+T)^\mu}).
	\end{align*}
	Together with the estimate \eqref{e: colinear leading edge matching 2} for $k= 1$, this implies \eqref{e: colinear matching derivative estimate}. The estimates for higher derivatives follow similarly. 
\end{proof}

We now construct an approximate solution $v^\mathrm{app}$ by gluing together $v^+$ and $v^-$. To measure the residual error of the approximate solution on the whole real line, we define 
\begin{align}
\Fres [v] = v_t - D \omega \partial_y^2 (\omega^{-1} v) - \left( c_* - \frac{3}{2 \eta_* (t+T)} \right) \omega \partial_y (\omega^{-1} v) - f'(0) v - \omega N(\omega^{-1} v). \label{e: F res def}
\end{align}
However, we cannot simply match $v^-(y)$ with $v^+(y,t)$ exactly at $y = (t+T)^\mu$, since according to Lemma \ref{l: colinear pointwise matching} doing so would produce a discontinuity at $y = (t+T)^\mu$, which is decaying in time but nonetheless present. Differentiating such an approximate solution upon substituting it into $F_\mathrm{res}$ would produce terms involving the Dirac delta and its derivatives. To avoid having to consider distribution-valued solutions, we instead glue $v^+$ and $v^-$ smoothly across the interval $y \in [(t+T)^\mu, (t+T)^\mu +1]$. We first define a smooth positive cutoff function $\chi_0$ satisfying 
\begin{align}
\chi_0 (y) = \begin{cases}
1, & y \leq 0, \\
0, & y \geq 1,
\end{cases}
\end{align}
and $0 \leq \chi_0 \leq 1$. We then define the time-varying cutoff 
\begin{align}
\chi(y, t) = \chi_0 (y - (t+T)^\mu). \label{e: chi t def}
\end{align}
We now use this cutoff to define our approximate solution as 
\begin{align}
\vapp(y,t) = \chi(y,t) v^-(y) + (1-\chi(y,t)) v^+ (y,t).
\end{align}
Note that $\vapp$ is smooth, since every term on the right hand side is smooth. 

\begin{prop}\label{p: colinear residual}
	Let $\beta_0 = \sqrt{\Deff}$ and $y_0 = 1+a$. Fix $0 < \mu < \frac{1}{8}$ and let $r = 2 + \mu$. There exists a constant $C > 0$ such that for all $T$ sufficiently large and $t > 0$, we have 
	\begin{align}
	\| F_\mathrm{res} [v^\mathrm{app}] (\cdot, t) \|_{L^\infty_{0, r}} \leq \frac{C}{(t+T)^{1/2-4\mu}}. \label{e: colinear residual estimate}
	\end{align}
	We denote $R (y,t) = F_\mathrm{res}[v](y,t)$. 
\end{prop}
\begin{proof}
	Note that $v^\mathrm{app}(y,t) = v^+(y,t)$ for $y \geq (t+T)^\mu + 1$, and so by Corollary \ref{c: colinear vplus residual}, we have
	\begin{align}
	\| 1_{\{ \cdot \geq (t+T)^\mu + 1\}} F_\mathrm{res} [v^\mathrm{app}] (\cdot, t) \|_{L^\infty_{0, r}} &= \| 1_{\{ \cdot \geq (t+T)^\mu + 1\}} F_\mathrm{res} [v^+(\cdot, t)] \|_{L^\infty_{0, r}} \\
	&\leq \frac{C}{(t+T)^{1/2 - \mu/2}}. 
	\end{align}
	On the other hand, for $y \leq (t+T)^\mu$, we have $v^\mathrm{app} (y,t) = v^-(y) = \omega(y) q_*(y)$, and note that $q_*$ solves the equation \eqref{e: U eqn} in the absence of the term $-\frac{3}{2 \eta_*} (t+T)^{-1}$, which implies that
	\begin{align*}
		F_\mathrm{res} [v^-] = \omega \frac{3}{2\eta_* (t+T)} \partial_y q_*.
	\end{align*}
	In particular, using that $|\omega (y)\partial_y q_*(y)| \leq C \rho_{0,1}(y)$ by \eqref{e: front asymptotics}, we have
	\begin{align*}
	\| 1_{\{\cdot \leq (t+T)^\mu\}} F_\mathrm{res} [v^\mathrm{app}] (\cdot, t) \| \leq C \frac{(t+T)^\mu}{t+T} \leq \frac{C}{(t+T)^{1/2-\mu}}. 
	\end{align*}
	
	It only remains to estimate $F_\mathrm{res} [v^\mathrm{app}] (y,t)$ on the intermediate region $(t+T)^\mu \leq y \leq (t+T)^\mu + 1$. To handle this region, we first decompose $F_\mathrm{res}[v]$ into linear and nonlinear parts, writing
	\begin{align}
	F_\mathrm{res}[v] = F_\mathrm{res}^\mathrm{lin}[v] + F_\mathrm{res}^\mathrm{nl} [v],
	\end{align}
	with
	\begin{align}
	F_\mathrm{res}^\mathrm{lin} [v] = v_t - D \omega \partial_y^2 (\omega^{-1} v) - \left( c_* - \frac{3}{2 \eta_* (t+T)} \right) \omega \partial_y (\omega^{-1} v) - f'(0) v, 
	\end{align}
	and
	\begin{align}
	F_\mathrm{res}^\mathrm{nl} [v] = - \omega N(\omega^{-1} v). 
	\end{align}
	First we show that the nonlinearity is irrelevant on this length scale. Since $\omega^{-1} v^\mathrm{app}$ is uniformly bounded, by Taylor's theorem there exists a constant $C > 0$ such that
	\begin{align*}
	|F_\mathrm{res}^\mathrm{nl} [v^\mathrm{app}]| \leq C \omega^{-1} |v^\mathrm{app}|^2. 
	\end{align*}
	Since $\omega(y) = e^{\eta_* y}$ for $y \geq 1$, we then have
	\begin{align*}
		1_{\{ (t+T)^\mu \leq y \leq (t+T)^\mu + 1 \}} |F_\mathrm{res}^\mathrm{nl} [v^\mathrm{app}] (y,t)| &\leq C 1_{\{ (t+T)^\mu \leq y \leq (t+T)^\mu + 1 \}} e^{-\eta_* (t+T)^\mu} |v^\mathrm{app}|^2 \\
		&\leq \frac{C}{(t+T)^{\mu - 1/2}},
	\end{align*}
	and so the nonlinearity is irrelevant in this regime, as claimed. 
	
	For the linear part, we then write
	\begin{align*}
	F_\mathrm{res}^\mathrm{lin} [v^\mathrm{app}] = \chi F_\mathrm{res}^\mathrm{lin} [v^-] + (1-\chi) F_\mathrm{res}^\mathrm{lin} [v^+] + [F_\mathrm{res}^\mathrm{lin}, \chi] (v^- - v^+),
	\end{align*}
	with commutator
	\begin{align*}
	[F_\mathrm{res}^\mathrm{lin}, \chi] v = F_\mathrm{res}^\mathrm{lin} [\chi v] - \chi F_\mathrm{res}^\mathrm{lin} [v]. 
	\end{align*}
	It follows from the same arguments for the $y \geq (t+T)^\mu + 1$ region above that
	\begin{align*}
	\left\| 1_{\{ (t+T)^\mu \leq \cdot \leq (t+T)^\mu + 1 \}} \left( \chi(\cdot, t) F_\mathrm{res}^\mathrm{lin} [v^-](\cdot, t) + (1-\chi(\cdot, t)) F_\mathrm{res}^\mathrm{lin} [v^+](\cdot, t) \right) \right\|_{L^\infty_{0, r}} \leq \frac{C}{(t+T)^{1/2-\mu}}. 
	\end{align*}
	It then only remains to control the terms involving the commutator $[F_\mathrm{res}^\mathrm{lin}, \chi]$, which is a first order differential operator in space, with smooth, bounded coefficients. Hence
	\begin{align}
	\left\|  1_{\{ (t+T)^\mu \leq \cdot \leq (t+T)^\mu + 1 \}} \left( [F_\mathrm{res}^\mathrm{lin}, \chi] (v^- - v^+) (\cdot, t) \right) \right\|_{L^\infty_{0, r}} \leq C \|  1_{\{ (t+T)^\mu \leq \cdot \leq (t+T)^\mu + 1 \}} |(1 + \partial_y) [v^- (\cdot, t) - v^+(\cdot, t)] \|_{L^\infty_{0, r}}. 
	\end{align}
	We now use the estimates of Lemma \ref{l: colinear pointwise matching} to expand $v^- - v^+$ on the region of interest. By Taylor's theorem and Lemma \ref{l: colinear pointwise matching}, we have
	\begin{align}
	v^+(y,t) &= v^+((t+T)^\mu, t) + v^+_y((t+T)^\mu, t) + \mathrm{O}((t+T)^{\mu-1/2}), \\
	v^-(y,t) &= v^-((t+T)^\mu, t) + v^-_y((t+T)^\mu, t) + \mathrm{O}((t+T)^{\mu-1/2})
	\end{align} 
	for $(t+T)^\mu \leq y \leq (t+T)^\mu + 1$. Taking the difference of these two expansions and using Lemma \ref{l: colinear pointwise matching} to estimate the difference between individual terms, we find
	\begin{align*}
	\langle y \rangle^{2 + \mu}1_{\{ (t+T)^\mu \leq y \leq (t+T)^\mu + 1 \}} | v^- (y, t) - v^+(y, t)| &\leq C (t+T)^{2 \mu + \mu^2} (t+T)^{\mu-1/2} \\
	&\leq \frac{C}{(t+T)^{1/2 - 4 \mu}}. 
	\end{align*}
	Arguing similarly, we obtain an analogous estimate on $\partial_y(v^-(y,t) - v^+(y,t))$ in this region. Hence we conclude
	\begin{align*}
	\left\|  1_{\{ (t+T)^\mu \leq y \leq (t+T)^\mu + 1 \}} \left( [F_\mathrm{res}^\mathrm{lin}, \chi] (v^- - v^+) (\cdot, t) \right) \right\|_{L^\infty_{0, r}} \leq \frac{C}{(t+T)^{1/2-4\mu}},
	\end{align*}
	which completes the proof of the desired estimates. 
\end{proof}

For later use, we record the following result which says that $q_*$ is well approximated by $\omega^{-1} \vapp$ for large times, which follows from Lemma \ref{l: colinear pointwise matching} and the choice of $v^+$ in Corollary \ref{c: colinear vplus residual}. 
\begin{corollary}\label{c: colinear front approximated by vapp}
	Fix $0 < \mu < \frac{1}{8}$ and let $r = 2 + \mu$. There exists a constant $C>0$ such that
	\begin{align}
	\| q_* - \omega^{-1} \vapp(\cdot, t) \|_{L^\infty_{0, r}} \leq \frac{C}{(t+T)^{1/2-4\mu}}
	\end{align}
	for all $t>0$ provided $T$ is sufficiently large. 
\end{corollary}

\section{Constructing the approximate solution --- $u_0$, $u_1$ linearly independent}\label{s: linearly independent}

We will prove Theorem \ref{t: main} in the case where $u_0$ and $u_1$ are co-linear by establishing an appropriate nonlinear stability result for the approximate solution $v_\mathrm{app}$ constructed Section \ref{s: colinear}. This stability argument, however, will only rely on the estimates of Proposition \ref{p: colinear residual} and Corollary \ref{c: colinear front approximated by vapp} together with sharp estimates on the semigroup $e^{\mcl t}$ generated by the linearization about $q_*$. In the case where $u_0$ and $u_1$ are linearly independent, we will construct a modified approximate solution with a different form, but importantly this approximate solution will still obey the estimates of Proposition \ref{p: colinear residual} and Corollary \ref{c: colinear front approximated by vapp}. We will also have the same estimates on the semigroup $e^{\mcl t}$ in this case, so \emph{exactly the same} nonlinear stability argument will apply here as in the co-linear case.

Hence, we first construct an approximate solution $v_\mathrm{app}$ in the linearly independent case, rather than completing the proof of Theorem \ref{t: main} in the colinear case right away, since the proof will be the same in both cases once we have obtained the analogous results of Proposition \ref{p: colinear residual} and Corollary \ref{c: colinear front approximated by vapp}. 

We therefore now assume that Hypothesis \ref{hyp: DR matrix pencil} holds, but with $u_0$ and $u_1$ linearly independent. We first construct a diffusive normal form for the matrix symbol. 

\subsection{Diffusive normal form}
We show that in this case, the symbol $A(\lambda, \nu)$ may be transformed into the form \eqref{e: normal 2}. 
\begin{lemma}[Diffusive normal form --- linearly independent case]\label{l: lin ind normal form}
	Assume now that $u_0$ and $u_1$ are linearly independent. Then there exist invertible matrices $S, Q \in \R^{n \times n}$ such that, for $n \geq 3$, 
	\begin{align}
	B(\lambda, \nu) := S A(\lambda, \nu) Q = \begin{pmatrix}
	-\lambda + b_{11}^{02} \nu^2 & b_{12}^{01} \nu + \tilde{b}_{12}(\lambda, \nu) & b_{13}(\lambda, \nu) \\
	- \nu + \tilde{b}_{21}(\lambda, \nu) & 1 + \tilde{b}_{22} (\lambda, \nu) & b_{23}(\lambda, \nu) \\
	b_{31}(\lambda, \nu) & b_{32}(\lambda, \nu) & b_{33}(\lambda, \nu) 
	\end{pmatrix}, \label{e: lin ind normal form}
	\end{align}
	where $b_{11}^{02}, b_{12}^{02} \in \R$ satisfy $b_{11}^{02} + b_{12}^{01} > 0$, and the polynomials
	\begin{align*}
	\tilde{b}_{12}(\lambda, \nu), \tilde{b}_{21}(\lambda, \nu), \tilde{b}_{22}(\lambda, \nu) &\in \C, \\
	b_{13}(\lambda, \nu), b_{23} (\lambda, \nu) &\in \C^{1 \times n-2}, \\
	b_{31}(\lambda, \nu), b_{32} (\lambda, \nu) &\in \C^{n-2 \times 1}, 
	\end{align*}
	and $b_{33}(\lambda, \nu) \in \C^{n-2 \times n-2}$ satisfy
	\begin{align}
	\tilde{b}_{12} (\lambda, \nu) &= b_{12}^{10} \lambda + b_{12}^{02} \nu^2, \\
	b_{13} (\lambda, \nu) &= b_{13}^{01} \nu + \mathrm{O}(\lambda, \nu^2) =: b_{13}^{01} \nu + \tilde{b}_{13}(\lambda, \nu) \\
	\tilde{b}_{21} (\lambda, \nu) &= b_{12}^{10} \lambda + b_{21}^{02} \nu^2, \\
	\tilde{b}_{22} (\lambda, \nu) &= b_{22}^{01} \nu + b_{22}^{10} \lambda + b_{22}^{02} \nu^2  \\
	b_{23} (\lambda, \nu) &= b_{23}^{01} \nu + \mathrm{O}(\lambda, \nu^2) =: b_{23}^{01} \nu + \tilde{b}_{23} (\lambda, \nu) \\
	b_{31} (\lambda, \nu) &= b_{31}^{10} \lambda + b_{31}^{02} \nu^2 , \\
	b_{32} (\lambda, \nu) &= b_{32}^{01} \nu + b_{32}^{10} \lambda + b_{32}^{02} \nu^2 \\
	b_{33}(\lambda, \nu) &= b_{33}^{00} + b_{33}^{01} \nu + \mathrm{O}(\lambda, \nu^2) =: b_{33}^{00} + b_{33}^{01} \nu + \tilde{b}_{33} (\lambda, \nu),
	\end{align}
	and $b_{33}^{00}$ is invertible. If $n = 2$, then the form \eqref{e: lin ind normal form} is simply replaced by
	\begin{align}
	B(\lambda, \nu) = S A(\lambda, \nu) Q = \begin{pmatrix}
	- \lambda + b_{11}^{02} \nu^2 & b_{12}^{01} \nu + \tilde{b}_{12}(\lambda, \nu) \\
	-\nu + \tilde{b}_{21}(\lambda, \nu) & 1 + \tilde{b}_{22}(\lambda, \nu). 
	\end{pmatrix}
	\end{align}
\end{lemma}
\begin{proof}
	We assume $n \geq 3$, since the $n = 2$ case may be recovered by simply ignoring all parts of the argument involving terms $b_{j3}$. As in the proof of Lemma \ref{l: normal form colinear} in the co-linear case, we first let $S$ and $Q$ be arbitrary invertible matrices, and write 
	\begin{align}
	B(\lambda, \nu) = S A(\lambda, \nu) Q = \begin{pmatrix}
	b_{11} (\lambda, \nu) & b_{12}(\lambda, \nu) & b_{13}(\lambda, \nu) \\
	b_{21} (\lambda, \nu) & b_{22} (\lambda, \nu) & b_{23} (\lambda, \nu) \\
	b_{31}(\lambda, \nu) & b_{32}(\lambda, \nu) &  b_{33}(\lambda, \nu),
	\end{pmatrix}
	\end{align}
	and sequentially modify $S$ and $Q$ to eliminate undesirable terms in $B(\lambda, \nu)$. We also again expand $B$ as 
	\begin{align}
	B(\lambda, \nu) = B^0 + B^{10} \lambda + B^{01} \nu + B^{02} \nu^2. \label{e: lin ind B expansion}
	\end{align}
	
	\emph{Step 1: $b_{11}(0,0) = b_{12} (0,0) = b_{21} (0,0) = 0, b_{13}(0,0) = 0$, and $b_{31}(0,0) = 0$}. 
	
	As in the co-linear case, choosing $Q$ and $S$ such that $Q e_0 = u_0$ and $(S^0)^T e_0 = c e_\mathrm{ad}$ for some constant $c \in \R$, where $e_0 = (1, 0, ..., 0)^T$, guarantees that $\ker B_0 = \ker ((B^0)^T) = \spn(e_0)$, i.e. that the first row and column of $B^0$ vanish, as desired. 
	
	\emph{Step 2: $b_{22}(\lambda, \nu) = 1 + \mathrm{O}(\lambda, \nu)$, and $b_{32} (\lambda, \nu) = \bigo(\lambda, \nu)$}. 
	
	We first modify the second column of $Q$ so that $Q e_1 = u_1$, where $e_1 = (0, 1, 0, ..., 0)^T$. We can do this while maintaining invertibility of $Q$ since $u_0$ and $u_1$ are linearly independent. The desired condition is equivalent to $B^0 e_1 = S A^0 Q e_1 = e_1$, which then in turns becomes 
	\begin{align*}
	SA^0 u_1 = e_1. 
	\end{align*}
	Note that the condition $S^T e_0 = ce_\mathrm{ad}$ implies that the first entry of the vector $S A^0 u_1$ is equal to $\langle S A^0 u_1, e_0 \rangle = \langle A^0 u_1, c e_\mathrm{ad} \rangle = \langle u_1, c (A^0)^T e_\mathrm{ad} \rangle = 0$. This also shows that the first row of $S$, which is proportional to $e_\mathrm{ad}^T$, is linearly independent from $A^0 u_1$. We can therefore choose to second row of $S$ to be 
	\begin{align}
	s_2^T = \frac{(A^0 u_1)^T}{|A^0 u_1|^2} \label{e: s2 formula}
	\end{align}
	while maintaining invertibility of $S$. We then let $\{ s_3, ..., s_n \}$ denote a basis for the orthogonal complement in $\R^n$ of $\spn (e_\mathrm{ad}, s_2)$, and set 
	\begin{align}
	S = \begin{pmatrix}
	c e_\mathrm{ad}^T \\
	s_2^T \\
	s_3^T \\
	\vdots \\
	s_n^T
	\end{pmatrix}. 
	\end{align}
	Since each row of $S$ is orthogonal to $A^0 u_1$ except the second row, it follows that $S A^0 u_1 = e_1$, as desired.
	
	\emph{Step 3: $b_{21} (\lambda, \nu) = - \nu + \bigo(\lambda, \nu^2)$ and $b_{31}(\lambda, \nu) = \bigo (\lambda, \nu^2)$. } 
	
	From examining the expansions of $A$ and $B$, we find that \eqref{e: hyp double root} of Hypothesis \ref{hyp: DR matrix pencil} implies that 
	\begin{align*}
	B^{01} e_0 = - B^0 e_1. 
	\end{align*}
	Having shown in the previous step that $B^0 e_1 = e_1$, we conclude that $B^{01} e_0 = - e_1$, which implies the desired conditions. 
	
	\emph{Step 4: $b_{23}(0,0) = 0$}. The condition $b_{23}(0,0) = 0$ is equivalent to $\langle (B^0)^T e_1, e_j \rangle = 0$, for $j = 2, 3, ..., n-1$ where $e_j \in \R^n$ is the vector whose $(j+1)$-th entry is equal to 1, with all other entries equal to zero. Using that $B^0 = S A^0 Q$, and $S^T e_1 = \alpha A^0 u_1$ for $\alpha = |A^0 u_1|^{-2}$ from \eqref{e: s2 formula}, we rewrite this condition as 
	\begin{align}
	\langle (A^0)^T A^0 u_1, Q e_j \rangle = 0, \quad j = 2, 3, ... n-1. \label{e: normal form li step 4}
	\end{align}
	That is, we need the last $n-2$ columns of $Q$ to be orthogonal to the vector $(A^0)^T A^0 u_1$, while maintaining invertibility of $Q$. Note that since $Q e_0 = u_0$ and $A^0 u_0 = 0$, we have
	\begin{align*}
	\langle (A^0)^T A^0 u_1, Q e_0 \rangle = \langle A^0 u_1, A^0 u_0 \rangle = 0,
	\end{align*}
	while
	\begin{align*}
	\langle (A^0)^T A^0 u_1, Q e_1 \rangle = \langle A^0 u_1, A^0 u_1 \rangle \neq 0
	\end{align*}
	since $Q e_1 = u_1$. That is, the vector $(A^0)^T A^0 u_1$ is orthogonal to the first column of $Q$ but not to the second column of $Q$. Since we are assuming $n \geq 3$, the orthogonal complement $E^\perp$ of the vector $(A^0)^T A^0 u_1$ in $\R^n$ has dimension $n-1 \geq 2$. Let $\{ v_2, ..., v_{n-1}\}$ denote a basis for $E^\perp \setminus \spn (Q e_0)$, which has dimension $n-2$. We then choose
	\begin{align}
	Q = \begin{pmatrix} 
	u_0 & u_1 & v_2 & ... & v_{n-1}.
	\end{pmatrix}
	\end{align}
	The columns of $Q$ are then linearly independent, so $Q$ is invertible, and the last $n-2$ columns are orthogonal to $(A^0)^T A^0 u_1$, by construction, which guarantees the desired condition holds.

	\emph{Step 5: Verifying $b_{11}^{02} + b_{12}^{01} > 0$ and $b_{11}(\lambda, \nu) = -\lambda+ \bigo(\nu^2)$.} 
	
	We observe from \eqref{e: lin ind B expansion} that 
	\begin{align*}
	b_{11}^{02} = \langle B^{02} e_0, e_0 \rangle, \quad b_{12}^{01} = \langle B^{01} e_1, e_0 \rangle. 
	\end{align*}
	Matching terms in the expansions of $A$ and $B$, and using that $Q e_0 = u_0$, $Q e_1 = u_1$, and $S^T e_0 = c e_\mathrm{ad}$, we conclude 
	\begin{align*}
	b_{11}^{02} =  \langle B^{02} e_0, e_0 \rangle = \langle S A^{02} Q e_0, e_0 \rangle = c \langle A^{02} u_0, e_\mathrm{ad} \rangle 
	\end{align*}
	and 
	\begin{align*}
	b_{12}^{01} = \langle B^{01} e_1, e_0 \rangle = c\langle A^{01} u_1, e_\mathrm{ad} \rangle. 
	\end{align*}
	Hence \eqref{e: hyp simple DR} of Hypothesis \ref{hyp: DR matrix pencil} implies that 
	\begin{align*}
	- c \langle u_0, e_\mathrm{ad} \rangle (b_{11}^{02} + b_{12}^{01}) < 0. 
	\end{align*}
	Choosing $c = \langle u_0, e_\mathrm{ad} \rangle^{-1}$ then implies that $b_{11}^{02} + b_{12}^{01} > 0$. We now observe from the expansions \eqref{e: lin ind B expansion} and \eqref{e: A expansion} for $A$ and $B$ that 
	\begin{align*}
	b_{11}^{10} = \langle B^{10} e_0, e_0 \rangle = -\langle S Q e_0, e_0 \rangle = - \langle u_0, c e_\mathrm{ad} \rangle = -1,
	\end{align*}
	as desired. 
	
	\emph{Step 6: Verifying that $b_{33}(0,0)$ is invertible.} 
	
	Note that since $b_{13}(0,0) = b_{23}(0,0) = 0$, if $b_{33}(0,0)$ any nontrivial kernel of $b_{33}(0,0)$ would contribute to the kernel of $B^0$. However, the kernel of $B^0$ is one-dimensional and spanned by $e_0$, so this is not possible, and hence $b_{33}(0,0)$ must be invertible.
\end{proof}

\subsection{Constructing the diffusive tail --- $u_0$, $u_1$ linearly independent}\label{s: lin ind diffusive tail}
Defining $v$ and $\Phi = Q^{-1} v$ as in Section \ref{s: colinear diffusive tail}, we again see that for $y \geq 1$, $\Phi(y,t)$ solves the equation
\begin{align*}
- B(\partial_t, \partial_y) \Phi - \frac{3}{2 (t+T)} SQ \Phi + \frac{3}{2 \eta_*(t+T)} SQ \Phi_x = S e^{\eta_* y} N(e^{-\eta_* y} Q \Phi), 
\end{align*}
but $S$ and $Q$ are now chosen as in Lemma \ref{l: lin ind normal form}. As a result, writing $\Phi = (\phiI, \phiII, \phih)$ we find equations of the form
\begin{align}
\partial_t \phiI &= b_{11}^{02} \partial_y^2 \phiI + \frac{3}{2(t+T)} \phiI + b_{12}^{01} \partial_y \phiII +  \mcf^\mathrm{I}_1 + \mcf^\mathrm{I}_2 \label{e: lin ind phiI eqn} \\
\phiII &= \partial_y \phiI + \FII_1 + \FII_2 + \FII_3 \\
b_{33}^{00} \phih &= \Fh_0 + \Fh_1 + \Fh_2. \label{e: lin ind phih eqn}
\end{align}
The leading order dynamics are governed by
\begin{align}
\partial_t \phiI &= b_{11}^{02} \partial_y^2 \phiI + \frac{3}{2 (t+T)} \phiI + b_{12}^{01} \partial_y \phiII \label{e: lin ind leading order 1} \\
\phiII &= \partial_y \phiI \label{e: lin ind leading order 2}\\
b_{33}^{00} \phih &= \Fh_0 (\phiI, \phiII), \label{e: lin ind leading order 3}
\end{align}
while the terms $\mathcal{F}^\mathrm{I/II/h}_j (\phiI, \phiII, \phih, y, t+T), j = 1,2,3$ encode higher order corrections, and we give explicit expressions for these terms in Appendix \ref{app: lin ind}. Inserting \eqref{e: lin ind leading order 2} into \eqref{e: lin ind leading order 1}, we find the diffusion equation
\begin{align}
\partial_t \phiI= (b_{11}^{02} + b_{12}^{01}) \partial_y^2 \phiI + \frac{3}{2 (t+T)} \phiI,
\end{align} 
and hence we define the effective diffusivity $\Deff = b_{11}^{02} + b_{12}^{01}$. 

The term involving $b_{12}^{01} \partial_y^2$ arises from diffusive effects which come from interactions of distinct components. For instance, the system 
\begin{align}
u_t &= v_x, \\
v_t & = u_x - v
\end{align}
may be viewed as a perturbation of the heat equation, since the second equation implies that $v$ exponentially relaxes to $u_x$, and when $v = u_x$ the first equation becomes the heat equation. Indeed, the dispersion relation for this system satisfies Hypothesis \ref{hyp: dispersion reln}(i), despite the system not appearing parabolic. The diffusive normal form constructed in Lemma \ref{l: lin ind normal form} gives a systematic way to recognize and interpret these ``hidden diffusive effects'' even in systems with many components. 

To precisely unravel the leading order diffusive dynamics, as in Section \ref{s: colinear}, we introduce the self-similar variables
\begin{align}
\tau = \log (t+T), \quad \xi = \frac{1}{\sqrt{\Deff}} \frac{y+y_0}{\sqrt{t+T}},
\end{align}
where $\Deff = b_{11}^{02} + b_{12}^{01}$, and define 
\begin{align}
\Psi(\xi, \tau) = \begin{pmatrix}
\psiI (\xi, \tau) \\ \psiII(\xi, \tau) \\ \psih(\xi, \tau) 
\end{pmatrix}
:= 
\begin{pmatrix}
\phiI(y,t) \\
\phiII(y,t) \\
\phih(y,t) 
\end{pmatrix} = \Phi(y,t).
\end{align}
The new unknowns $(\psiI, \psiII, \psih)$ then solve the system
\begin{align}
\partial_\tau \psiI &= b_{11}^{02} \Deff^{-1} \partial_\xi^2 \psiI+ \frac{1}{2} \xi \partial_\xi \psiI + \frac{3}{2} \psiI + b_{12}^{01} e^{\tau/2} \Deff^{-1/2} \partial_\xi \psiII + \tfi_1 + \tfi_2 \label{e: lin ind psi I eqn} \\
e^{\tau} \psiII &= e^{\tau/2} \Deff^{-1/2} \partial_\xi \psiI + \tfii_1 + \tfii_2 + \tfii_3, \label{e: lin ind psi II eqn} \\
e^\tau b_{33}^{00} \psih &= \tfh_0 + \tfh_1 + \tfh_2, \label{e: lin ind psi h eqn}
\end{align}
where $\tfi_j(\psiI, \psiII, \psih, \xi, \tau), \tfii_j(\psiI, \psiII, \psih, \xi, \tau),$ and $\tfii_j(\psiI, \psiII, \psih, \xi, \tau)$, are defined in Appendix \ref{app: lin ind self sim}. 

We again emphasize that we have multiplied through by $e^{\tau}$ in deriving \eqref{e: lin ind psi I eqn} through \eqref{e: lin ind psi h eqn}, and so will regain a factor of $e^{-\tau} = (t+T)^{-1}$ when estimating the residual error of approximate solutions in the original variables. Observe that if we neglect all the higher order corrections $\tilde{\mathcal{F}}^\mathrm{I/II/h}_j$, we obtain the leading order system
\begin{align*}
\partial_\tau \psiI &= b_{11}^{02} \Deff^{-1} \partial_\xi^2  \psiI + \frac{1}{2} \xi \partial_\xi \psiI + \frac{3}{2} \psiI + b_{12}^{01} e^{\tau/2} \Deff^{-1/2} \partial_\xi \psiII, \\
e^{\tau/2} \psiII &= \Deff^{-1/2} \partial_\xi \psiI. 
\end{align*}
Note that if we solve the second equation for $\psiII$ and insert this into the first equation, we obtain the diffusive equation
\begin{align}
\partial_\tau \psiI = \partial_\xi^2 \psiI + \frac{1}{2} \xi \partial_\xi \psiI+ \frac{3}{2} \psi,
\end{align}
using that $(b_{11}^{02} + b_{12}^{01}) \Deff^{-1} = 1$. 

To precisely separate this leading order diffusive behavior from the higher order corrections, we make the ansatz
\begin{align}
\psiI (\tau, \xi) &= e^{\tau/2} \psiI_0 (\xi) + \psiI_1 (\xi), \label{e: lin ind psiI ansatz}\\
\psiII(\tau, \xi) &= \psiII_0(\xi) + e^{-\tau/2} \psiII_1 (\xi) + e^{-\tau} \psiII_2 (\xi), \label{e: lin ind psiII ansatz} \\
\psih(\tau, \xi) &= e^{-\tau/2} \psih_0 (\xi) + e^{-\tau} \psih_1 (\xi). \label{e: lin ind psih ansatz}
\end{align}
To capture the absorption effect in the wake of the front, we again consider the resulting equations for $\psi^{\mathrm{I}/\mathrm{II}/\mathrm{h}}_j$ only for $\xi > 0$, with homogeneous Dirichlet boundary condition at $\xi = 0$. Inserting this ansatz into \eqref{e: lin ind psi I eqn}-\eqref{e: lin ind psi h eqn} and first collecting only the leading order terms in $e^{\tau/2}$, we obtain the system
\begin{align}
0 &= b_{11}^{02} \Deff^{-1} \partial_\xi^2  \psiI_0 + \frac{1}{2} \xi \partial_\xi \psiI_0 + \psiI_0 + b_{12}^{01} \Deff^{-1/2} \partial_\xi  \psiII_0, \\
\psiII_0 &= \Deff^{-1/2} \partial_\xi  \psiI_0, \label{e: lin ind psi II 0 eqn}\\
b_{33}^{00} \psih_0 &= e^{-\tau/2} \tfh_0 (e^{\tau/2} \psiI_0 (\xi), \psiII_0 (\xi), \xi, \tau) =: \Gh_0(\xi), \label{e: lin ind psi h 0 eqn}
\end{align}
where we have noted from the expression \eqref{e: tilde F h 0 expression} that $\tfh_0$ does not depend on $\psih$, and that the right hand side of \eqref{e: lin ind psi h 0 eqn} is in fact independent of $\tau$. As suggested in the heuristic argument above, we insert the second equation into the first, and find that $\psiI_0$ solves
\begin{align}
0 = L_\Delta \psiI_0 =: \left(\partial_\xi^2 + \frac{1}{2} \xi \partial_\xi + 1 \right) \psiI_0, \label{e: lin ind psiI0 eqn},
\end{align}
using the fact that $\Deff = b_{11}^{02} + b_{12}^{01}$. As in Section \ref{s: colinear}, the unique solution to this equation, with boundary condition $\psiI_0(0) = 0$, is
\begin{align}
\psiI_0 (\xi) = \beta_0 \xi e^{-\xi^2/4},
\end{align}
up to the arbitrary factor $\beta_0 \in \R$. 
Hence from \eqref{e: lin ind psi II 0 eqn}, we conclude 
\begin{align}
\psiII_0(\xi) = \Deff^{-1/2} \partial_\xi \psi^I_0 (\xi) = \beta_0 \Deff^{-1/2} \partial_\xi \left( \xi e^{-\xi^2/4} \right). \label{e: lin ind psi II 0 expression}
\end{align}
Also since $b_{33}^{00}$ is invertible by Lemma \ref{l: lin ind normal form}, we can now solve \eqref{e: lin ind psi h 0 eqn} for $\psih_0$, obtaining 
\begin{align}
\psih_0 (\xi) = (b_{33}^{00})^{-1} \Gh_0 (\xi),
\end{align}
which satisfies the estimate
\begin{align}
| \partial_\xi^k \psih_0(\xi) | \leq C e^{-\xi^2/6} \label{e: lin ind psi h 0 estimate}
\end{align}
for $k = 0, 1, 2$. 

Having solved for $\psiI_0, \psiII_0$, and $\psih_0$, we now collect the next order terms resulting from inserting our ansatz into \eqref{e: lin ind psi I eqn}-\eqref{e: lin ind psi h eqn}, and thereby obtain the following equations for $\psiI_1, \psiII_1,$ and $\psih_1$: 
\begin{align}
\left(b_{11}^{02} \Deff^{-1} \partial_\xi^2 + \frac{1}{2} \xi \partial_\xi + \frac{3}{2} \right) \psiI_1 &= -b_{12}^{01}  \Deff^{-1/2} \partial_\xi \psiII_1  - \tfi_1 \left(e^{\tau/2 }\psiI_0, \psiII_0,  e^{-\tau/2} \psih_0, \xi, \tau \right) \label{e: lin ind psi I 1 eqn} \\
e^{\tau/2} \psiII_1 &= e^{\tau/2} \Deff^{-1/2} \partial_\xi \psiI_1 + \tfii_1 \left( e^{\tau/2 }\psiI_0, \psiII_0,  e^{-\tau/2} \psih_0, \xi, \tau \right), \\
e^\tau b_{33}^{00} \psih_1 &= \tfh_1 \left( e^{\tau/2 }\psiI_0, \psiII_0,  e^{-\tau/2} \psih_0, \xi, \tau \right) \label{e: lin ind psi h 1 eqn}
\end{align}
Solving the second equation for $\psiII_1$ in terms of $\psiI_1$ and $\psi^{\mathrm{I}/\mathrm{II}/\mathrm{h}}_0$ and inserting into the first equation, we obtain 
\begin{align}
\left( L_\Delta + \frac{1}{2} \right) \psiI_1 = - b_{12}^{01} \Deff^{-1/2} e^{-\tau/2} \partial_\xi \tfii_1 \left( e^{\tau/2 }\psiI_0, \psiII_0,  e^{-\tau/2} \psih_0, \xi, \tau \right) - \tfi_1 \left(e^{\tau/2 }\psiI_0, \psiII_0,  e^{-\tau/2} \psih_0, \xi, \tau \right). \label{e: lin ind psiI 1 eqn}
\end{align}
Note that the right hand side now only depends on $\psiI_0, \psiII_0$, and $\psih_0$, which we have already chosen. Examining the formulas \eqref{e: tf ii 1 expression} and \eqref{e: tfi 1 expression} for $\tfii_1$ and $\tfi_1$, we see that the right hand side
\begin{align}
\GI_1 (\xi) := - b_{12}^{01} \Deff^{-1/2} e^{-\tau/2} \partial_\xi \tfii_1 \left( e^{\tau/2 }\psiI_0, \psiII_0,  e^{-\tau/2} \psih_0, \xi, \tau \right) - \tfi_1 \left(e^{\tau/2 }\psiI_0 (\xi), \psiII_0 (\xi),  e^{-\tau/2} \psih_0 (\xi), \xi, \tau \right)
\end{align}
is independent of $\tau$. As pointed out in Section \ref{s: colinear diffusive tail}, $\lambda = -\frac{1}{2}$ is not in the spectrum of $L_\Delta$ on $L^2_\mathrm{odd} (\R_+)$, and so we can invert $(L_\Delta + \frac{1}{2})$ to solve \eqref{e: lin ind psi I 1 eqn} for $\psiI_1(\xi)$. We can then solve for $\psiII_1$ and $\psih_1$ as
\begin{align}
\psiII_1 (\xi) &= \GII_1 (\xi), \label{e: lin ind psi I 0 expression}\\
\psih_1 (\xi) &= (b_{33}^{00})^{-1} \Gh_1(\xi), \label{e: lin ind psi h 0 expression}
\end{align}
where 
\begin{align}
\GII_1 (\xi) := \Deff^{-1/2} \partial_\xi \psiI_1 (\xi) + e^{-\tau/2} \tfii_1 \left( e^{\tau/2} \psiI_0 (\xi), \psiII_0 (\xi), e^{-\tau/2} \psih_0 (\xi), \xi, \tau) \right)
\end{align}
and
\begin{align}
\Gh_1(\xi) = e^{-\tau} \tfh_1 \left( e^{\tau/2 }\psiI_0 (\xi), \psiII_0 (\xi),  e^{-\tau/2} \psih_0 (\xi), \xi, \tau \right)
\end{align}
are in fact independent of $\tau$ by the formulas \eqref{e: tf ii 1 expression}, \eqref{e: tfh 1 expression}. 

Exploiting the relationship of $L_\Delta$ to the quantum harmonic oscillator as in Section \ref{s: colinear diffusive tail}, we obtain the following estimates on $\psi^\mathrm{I/II/h}_1$. 

\begin{lemma}\label{l: lin ind psi 1 estimates}
	Let $\psiI_0, \psiII_0$, and $\psih_0$ given by \eqref{e: lin ind psi I 0 expression}, \eqref{e: lin ind psi II 0 expression}, and \eqref{e: lin ind psi h 0 expression}, respectively. Then there exist smooth functions $\psi^{\mathrm{I}/\mathrm{II}/\mathrm{h}}$, defined for $\xi \geq 0$, which solve \eqref{e: lin ind psi I 1 eqn}-\eqref{e: lin ind psi h 1 eqn} and satisfy the estimates
	\begin{align}
	| \partial_\xi^k \psiI_1 | &\leq C_k e^{-\xi^2/6}, \\
	| \partial_\xi^k\psiII_1 | & \leq C_k e^{-\xi^2/6}, \\
	| \partial_\xi^k\psih_1 | &\leq C_k e^{-\xi^2/6}
	\end{align}
	for all $\xi \geq 0$ and all natural numbers $k$. 
\end{lemma}

We must include one more term in the expansion for $\psiII$, since we have not yet eliminated the leading order terms in $\tfii_2$, which are $\mathrm{O}(1)$ in $e^{\tau/2}$. Collecting all $\mathrm{O}(1)$ terms in the equation for $\psiII$ after inserting our ansatz, we find 
\begin{align}
\psiII_2 (\xi) = \tfii_1 (\psiI_1, e^{-\tau/2} \psiII_1, \xi, \tau) + \tfii_2 (e^{\tau/2} \psiI_0, \psiII_0, e^{-\tau/2} \psih_0, \xi, \tau). 
\end{align}
Note again that the right hand side is in fact independent of $\tau$. We record the following elementary estimates on $\psiII_2$.
\begin{lemma}\label{l: lin ind psiII 2 estimate}
	Let $\psi^\mathrm{I/II/h}_j, j = 1,2$ be chosen as in Lemma \ref{l: lin ind psi 1 estimates}. There exists a constant $C > 0$ such that 
	\begin{align}
	| \psiII_2 (\xi) | \leq C e^{-\xi^2 /6}. 
	\end{align}
\end{lemma}

Having now chosen all the terms in the ansatz \eqref{e: lin ind psiI ansatz}-\eqref{e: lin ind psih ansatz}, we are ready to estimate the residual error resulting from this approximate solution. To do this, we define for $\Psi = (\psiI, \psiII, \psih)$, 
\begin{align}
\mcf_\mathrm{res} [\Psi] = \begin{pmatrix}
\partial_\tau - b_{11}^{02} \Deff^{-1} \partial_\xi^2 \psiI - \frac{1}{2} \xi \partial_\xi \psiI - \frac{3}{2} \psiI - b_{12}^{01} e^{\tau/2} \Deff^{-1/2} \partial_\xi \psiII - \tfi_1 (\psiI, \psiII, \psih) - \tfi_2 (\psiI, \psiII, \psih) \\
e^\tau \psiII- e^{\tau/2} \Deff^{-1/2} \partial_\xi \psiI - \tfii_1 (\psiI, \psiII, \psih) - \tfii_2 (\psiI, \psiII, \psih) - \tfii_3 (\psiI, \psiII, \psih) \\
e^\tau b_{33}^{00} \psih - \tfh_0 (\psiI, \psiII, \psih) - \tfh_1 (\psiI, \psiII, \psih) - \tfh_2 (\psiI, \psiII, \psih),
\end{pmatrix}
\end{align}
so that $\Psi$ solves \eqref{e: lin ind psi I eqn}-\eqref{e: lin ind psi h eqn} if and only if $\mcf_\mathrm{res}[\Psi] = 0$. Summarizing the results of this section and estimating the nonlinear terms as in Section \ref{s: colinear}, we arrive at the following estimate on the residual error. 

\begin{prop}[Diffusive tail in self-similar coordinates --- linearly independent case]\label{p: lin ind diffusive tail}
	Fix $\beta_0 \in \R, y_0 \in \R$, and $0 < \mu < \frac{1}{8}$, and let $\psi^\mathrm{I/II/h}_j$ be defined as above. Let
	\begin{align}
	\Psi(\xi, \tau) = \begin{pmatrix}
	e^{\tau/2} \psiI_0 (\xi) + \psiI_1 (\xi) \\ 
	\psiII_0 (\xi) + e^{-\tau/2} \psiII_1(\xi) + e^{-\tau} \psiII_2 (\xi), \\
	e^{-\tau/2} \psih_0 (\xi) + e^{-\tau} \psih_1(\xi). 
	\end{pmatrix}
	\end{align}
	There exists a constant $C > 0$ such that
	\begin{align}
	\left| \mcf_\mathrm{res} [\Psi] (\xi, \tau) \right| \leq C e^{-\tau/2} e^{-\xi^2/8}
	\end{align}
	for all $\tau > 0$ and all
	\begin{align}
	\xi \geq \frac{1}{\sqrt{\Deff}} \left( e^{(\mu - 1/2)\tau} + y_0 e^{-\tau/2} \right). 
	\end{align}
\end{prop}

Reverting back to the coordinates $(y,t)$, we obtain the following estimates on the approximate solution, recalling that we gain a factor of $(t+T)^{-1}$ since we multiplied by $e^\tau$ in deriving the system \eqref{e: lin ind psi I eqn}-\eqref{e: lin ind psi h eqn}. 

\begin{corollary}[Diffusive tail in normal form coordinates --- linearly independent case]\label{c: lin ind diffusive tail normal form}
	Fix $\beta_0 \in \R, y_0 \in \R,$ and $0 < \mu < \frac{1}{8}$. With $\Psi = (\psiI, \psiII, \psih)$ as in Proposition \ref{p: lin ind diffusive tail}, define 
	\begin{align}
	\phiIp (y,t) &= (t+T)^{1/2} \psiI_0 \left( \xi(y,t)  \right) + \psiI_1 \left( \xi(y,t) \right), \\
	\phiIIp (y,t) &= \psiII_0 \left( \xi(y,t)  \right) + (t+T)^{-1/2} \psiII_1 \left( \xi(y,t)  \right)  + (t+T)^{-1} \psiII_2 \left( \xi(y,t) \right)  \\
	\phihp(y,t) &= (t+T)^{-1/2} \psih_0 \left( \xi(y,t) \right)  + (t+T)^{-1} \psih_1\left( \xi(y,t) \right),
	\end{align}
	where 
	\begin{align}
	\xi(y,t) =  \frac{1}{\sqrt{\Deff}} \frac{y+y_0}{\sqrt{t+T}}. 
	\end{align}
	Then with $\Phi^+ = (\phiIp, \phiIIp, \phihp)$ and $\tilde{F}^+_\mathrm{res} [\Phi]$ defined as in \eqref{e: colinear Phi eqn}, with $Q$, $S$, and $B(\partial_t, \partial_y)$ as in Lemma \ref{l: lin ind normal form}, there exists a constant $C > 0$ such that for all $t > 0$ and all $y \geq (t+T)^\mu$, we have 
	\begin{align}
	| \tilde{F}^+_\mathrm{res} [\Phi^+](y,t)| \leq \frac{C}{(t+T)^{3/2}} e^{-(y+y_0)^2/[8 \Deff (t+T)]}. 
	\end{align}
\end{corollary}
Undoing the normal form transformation by considering $v = Q \Phi$, we finally obtain the following estimates on the residual error in the original coordinates.
\begin{corollary}\label{c: lin ind diffusive tail original coordinates}
	Fix $\beta_0 \in \R, y_0 \in \R,$ and $0 < \mu < \frac{1}{8}$. Define $v^+(y,t) = Q \Phi^+(y,t)$, with $\Phi^+$ chosen as in Corollary \ref{c: lin ind diffusive tail normal form}. Let $F^+_\mathrm{res} [v]$ be defined by \eqref{e: v eqn pencil form}. There exists a constant $C > 0$ such that for $t > 0$ and $y \geq (t+T)^\mu$, we have 
	\begin{align}
	| F^+_\mathrm{res} [v^+](y,t) | \leq \frac{C}{(t+T)^{3/2}} e^{-(y+y_0)^2/[8 \Deff (t+T)]}. 
	\end{align}
\end{corollary}

\subsection{Matching the diffusive tail to the front}
As in Section \ref{s: colinear}, we define $v^-(y) = \omega(y) q_*(y)$, and we will match this with the diffusive tail $v^+(y,t)$ on the intermediate length scale $y \sim (t+T)^\mu$. We first measure the error made by matching exactly at $y = (t+T)^\mu$.

\begin{lemma}[Pointwise matching --- $u_0, u_1$ linearly independent]\label{l: lin ind pointwise matching}
	Let $\beta_0 = \sqrt{\Deff}$ and $y_0 = a$, and fix $0 < \mu < \frac{1}{8}$. Assume $u_0$ and $u_1$ are linearly independent. The conclusions of Lemma \ref{l: colinear pointwise matching} remain valid in this case, with $v^+(y,t)$ constructed as in Section \ref{s: lin ind diffusive tail}. 
\end{lemma}
\begin{proof}
	Recalling that $v^+(y,t) = Q \Phi^+(y,t)$ with $Q$ defined in Lemma \ref{l: lin ind normal form} and $\Phi^+$ defined in Corollary \ref{c: lin ind diffusive tail original coordinates}, we write
	\begin{align*}
	v^+(y,t) &= Q \begin{pmatrix}
	\phiIp(y,t) \\ 0 \\ 0 
	\end{pmatrix}
	+ Q \begin{pmatrix} 0 \\ \phiIIp(y,t) \\ 0 \end{pmatrix}
	+ Q \begin{pmatrix}0 \\ 0 \\ \phihp (y,t) \end{pmatrix} \\
	&= u_0 \phiIp(y,t) + u_1 \phiIIp(y,t) + Q \begin{pmatrix}0 \\ 0 \\ \phihp (y,t) \end{pmatrix},
	\end{align*}
	using $Qe_0 = u_0$ and $Q e_1 = u_1$ from the construction of $Q$ in Lemma \ref{l: lin ind normal form} Using the formulas from Corollary \ref{c: lin ind diffusive tail original coordinates} for $\phiIp$ and $\phiIIp$ together with the expressions \eqref{e: lin ind psi I 0 expression} and \eqref{e: lin ind psi II 0 expression} for $\psiI_0$ and $\psiII_0$, and Taylor expanding the exponentials, we obtain the following expansion
	\begin{align}
	u_0 \phiIp((t+T)^\mu, t) + u_1 \phiIIp((t+T)^\mu, t) = u_0 \frac{\beta_0}{\sqrt{\Deff}} \left[ (t+T)^\mu + y_0 \right]  + u_1 \frac{\beta_0}{\sqrt{\Deff}} + \mathrm{O}[(t+T)^{\mu-1/2}]. \label{e: lin ind intermediate tail asymptotics}
	\end{align}
	The higher order terms are estimated using the Gaussian estimates on $\psi^\mathrm{I/II}_j, j = 1, 2$ from Lemmas \ref{l: lin ind psi 1 estimates} and \ref{l: lin ind psiII 2 estimate}. 
	By \eqref{e: front asymptotics}, the front interior satisfies
	\begin{align}
	v^-((t+T)^\mu) = u_0 (t+T)^\mu + u_1 + a u_0 + \mathrm{O}[(t+T)^{\mu-1/2}] \label{e: lin ind intermediate front asymptotics}
	\end{align}
	for all $t > 0$ provided $T$ is sufficiently large. Notice that if we choose $\beta_0 = \sqrt{\Deff}$ and $y_0 = a$, then the expressions \eqref{e: lin ind intermediate tail asymptotics} and \eqref{e: lin ind intermediate front asymptotics} match up to an $\mathrm{O}[(t+T)^{\mu-1/2}]$ error. From the expression for $\phihp$ from Corollary \ref{c: lin ind diffusive tail original coordinates}, the Gaussian estimate on $\psih_1$ from Lemma \ref{l: lin ind psi 1 estimates}, and the Gaussian estimate \eqref{e: lin ind psi h 0 estimate}, we obtain the estimate
	\begin{align*}
	\left| Q \begin{pmatrix} 0 \\ 0 \\ \phihp((t+T)^\mu, t) \end{pmatrix} \right| \leq C (t+T)^{-1/2}. 
	\end{align*}
	Combining with the expansions \eqref{e: lin ind intermediate tail asymptotics}-\eqref{e: lin ind intermediate front asymptotics}, we thereby obtain the desired estimate
	\begin{align*}
	| v^+((t+T)^\mu, t) - v^-((t+T)^\mu) | \leq C (t+T)^{\mu-1/2}. 
	\end{align*}
	The estimates on derivatives follow similarly. 
\end{proof}

Having measured the pointwise matching error at $y = (t+T)^\mu$, we now smoothly glue the diffusive tail to the interior of the front exactly as in Section \ref{s: colinear}. Let $\chi(y,t)$ be defined by \eqref{e: chi t def}. We then define our approximate solution as 
\begin{align}
v^\mathrm{app} (y,t) = \chi(y,t) v^-(y) + (1-\chi(y,t)) v^+(y,t). 
\end{align}
Letting $F_\mathrm{res}[v]$ be defined as in \eqref{e: F res def} to measure the residual error of approximate solutions, we obtain the following estimates.
\begin{prop}\label{p: lin ind residual estimate}
	Let $\beta_0 = \sqrt{\Deff}$ and $y_0 = a$. Fix $0 < \mu < \frac{1}{8}$ and let $r = 2 + \mu$. There exists a constant $C > 0$ such that for all $T$ sufficiently large and $t > 0$, we have
	\begin{align}
	\| F_\mathrm{res}[v^\mathrm{app}](\cdot, t) \|_{L^\infty_{0, r}} \leq \frac{C}{(t+T)^{1/2-4\mu}}. 
	\end{align}
	We let $R(y,t) = F_\mathrm{res}(y,t)$. 
\end{prop}
With the pointwise matching estimates of Lemma \ref{l: lin ind pointwise matching} in hand, the proof of Proposition \ref{p: lin ind residual estimate} is identical to that of Proposition \ref{p: colinear residual}. We further obtain the following corollary which measures how well $\omega^{-1} v^\mathrm{app}(y,t)$ resembles the critical front $q_*(y)$. 

\begin{corollary}\label{c: lin ind front approximated by vapp}
	Fix $0 < \mu < \frac{1}{8}$ and let $r = 2 + \mu$. There exists a constant $C>0$ such that
	\begin{align}
	\| q_* - \omega^{-1} \vapp(\cdot, t) \|_{L^\infty_{0, r}} \leq \frac{C}{(t+T)^{1/2-4\mu}}
	\end{align}
	for all $t>0$ provided $T$ is sufficiently large. 
\end{corollary}

\section{Linear estimates}\label{s: linear estimates}
In the previous two sections, we have constructed approximate solutions with good residual error, measured in Proposition \ref{p: colinear residual} when $u_0$ and $u_1$ are co-linear and Proposition \ref{p: lin ind residual estimate} when $u_0$ and $u_1$ are linearly independent. The one remaining ingredient needed to prove Theorem \ref{t: main} according to the program in \cite{CAMS} is sharp estimates on the linearized evolution $e^{\mcl t}$ near the critical front $q_*$. The key estimates are roughly of the form
\begin{align}
\| e^{\mcl t} g \|_{L^\infty_{0,-1}} &\leq \frac{C}{(1+t)^{3/2}} \| g \|_{L^\infty_{0, r}}, \label{e: linear estimate 1} \\
\| \partial_x e^{\mcl t} g \|_{L^\infty_{0,r}} &\leq \frac{C}{t^{3/2-r/2}} \| g\|_{L^\infty_{0,r}}, \label{e: linear estimate 2}
\end{align}
which capture two separate mechanisms for diffusive decay. The first captures sharp $t^{-3/2}$ decay by giving up spatial localization. This is analogous to the standard $t^{-1/2}$ decay estimate for the heat evolution from $L^1$ to $L^\infty$ , but with improved decay rates thanks to the absorption mechanism created by the stable state in the wake of the front; see \cite{SIMA}. The second preserves spatial localization but achieves decay of derivatives through exploiting diffusive smoothing. 

These and other closely related decay estimates may be proven via a detailed analysis of the resolvent near the origin, as in \cite{SIMA, CAMS} which treat the scalar, higher order case. In this scalar case, the estimate \eqref{e: linear estimate 1} was originally proven in $L^2$-based spaces in \cite{SIMA}, while the linear analysis in \cite{CAMS} adapts this estimate to $L^1$ and $L^\infty$ based spaces and further proves derivative estimates of the form \eqref{e: linear estimate 2} needed to close the front selection argument. 

The proofs are almost the same in the multi-component case as in \cite{CAMS, SIMA}, and we explain the essential modifications here. 

\subsection{Resolvent estimates}
The general strategy is to first analyze the limiting resolvent equation
\begin{align}
	(\mcl_+ - \gamma^2) u = g, \label{e: far field resolvent equation}
\end{align}
for $g \in L^1_{1,1} (\R)$, where 
\begin{align}
	\mcl_+ = D (\partial_y - \eta_*)^2 + c_* (\partial_y - \eta_*) + f'(0) 
\end{align}
captures the limiting dynamics of $\mcl$ as $y \to +\infty$. In solving \eqref{e: far field resolvent equation}, we restrict to odd data $g$. This is because the leading order far-field dynamics are essentially governed by the heat equation, and for the heat equation restricting to odd initial data is equivalent to enforcing a homogeneous Dirichlet boundary condition at $y = 0$, which models absorption in the wake due to the exponential stability of $u^-$. 

We are then able to transfer estimates of the limiting resolvent to the full resolvent $(\mcl-\gamma^2)^{-1}$ via a \emph{far-field/core decomposition}, in which we capture the limiting dynamics at $+\infty$ in an explicit ansatz, and then solve for localized corrections by exploiting Fredholm properties of $\mcl$ in exponentially weighted function spaces; see \cite[Section 3.2]{CAMS} or \cite[Section 3.1]{SIMA}. Detailed estimates on the resolvent near $\gamma^2 = 0$ may then be transferred to sharp temporal decay estimates through the inverse Laplace transform, using carefully chosen contours which follow the essential spectrum of $\mcl$; see for instance \cite[Section 4]{CAMS}. 

The only place in the proofs of the linear estimates in \cite{CAMS, SIMA} which relies on the restriction to scalar equations is in the proof of \cite[Lemma 2.2]{SIMA}, which is the key step in describing the far-field resolvent. The proof of \cite[Lemma 2.2]{SIMA} only relies on the restriction to the scalar case in guaranteeing that the results are sharp, but we nonetheless give an adaptation here which guarantees that our linear estimates are sharp in the multi-component case. 
\begin{lemma}\label{l: center projections pole}
	Let $\mathcal{M}(\gamma^2)$ be the matrix obtained by reformulating the resolvent equation \eqref{e: far field resolvent equation} as a first-order system
	\begin{align}
	(\partial_y - \mathcal{M}(\gamma^2)) U = \begin{pmatrix} 0 \\ D^{-1} g \end{pmatrix} \label{e: far field resolvent eqn first order}
	\end{align}
	in $U = (u,u_y) \in \R^{2n}$. 
	
	\begin{enumerate}[i.]
		\item For $\gamma$ small, $\mathcal{M}(\gamma^2)$ has precisely two eigenvalues $\nu^\pm(\gamma)$ in a neighborhood of the origin, which are analytic in $\gamma$ and satisfy the expansions
		\begin{align}
		\nu^\pm(\gamma) = \pm \sqrt{-d_{10}d_{02}^{-1} } \gamma + \mathrm{O}(\gamma^2). 
		\end{align}
		\item Let $P^\mathrm{cs/cu}(\gamma)$ denote the spectral projections onto the eigenspaces of $\mathcal{M}(\gamma^2)$ associated with $\nu^\mp(\gamma)$, respectively, which are one-dimensional provided $\gamma \neq 0$. Then $P^\mathrm{cs/cu}(\gamma)$ are meromorphic in $\gamma$ in a neighborhood of the origin, with expansions
		\begin{align}
		P^\mathrm{cs/cu}(\gamma) = \mp \frac{1}{\gamma} P_\mathrm{pole} + \mathrm{O}(1) 
		\end{align}
		for some matrix $P_\mathrm{pole}\in \C^{2n \times 2n}$.
	\end{enumerate}
\end{lemma}
\begin{proof}
	The proof is identical to that of \cite[Lemma 2.2]{SIMA}, since the only step there which relies on the restriction to scalar equations is in proving that the top right entry of $P_\mathrm{pole}$ is nonzero, but here we are not yet claiming that any specific entries of $P_\mathrm{pole}$ are nontrivial. 
\end{proof}

The strategy of \cite[Section 2]{SIMA} is to recover the solution to the resolvent equation \eqref{e: far field resolvent equation} from the first order formulation \eqref{e: far field resolvent eqn first order}. Let $T_\gamma$ denote the matrix Green's function associated to this first order formulation, which solves
\begin{align}
	(\partial_y - \mathcal{M}(\gamma^2)) T_\gamma = - \delta_0 I_{2n},
\end{align}
where $I_{2n}$ denotes the identity matrix of size $2n$. The solution to \eqref{e: far field resolvent equation} is then given by
\begin{align}
u(y;\gamma) = \int_\R \Pi_1 T_\gamma(y-\zeta) \Lambda_1 g(\zeta) \, d \zeta, \label{e: far field resolvent solution}
\end{align}
where $\Pi_1 : \C^{2n} \to \C^n$ and $\Lambda_1 : \C^n \to \C^{2n}$ are defined by 
\begin{align}
	\Pi_1 \begin{pmatrix}
	g_1 \\ g_2 
	\end{pmatrix} = g_1,
	\quad 
	\Lambda_1 g = \begin{pmatrix} 0 \\ g \end{pmatrix}. 
\end{align}
Following the argument of \cite[Section 2]{SIMA}, we decompose $T_\gamma$ as
\begin{align}
	T_\gamma(\zeta) = -\left( e^{\nu^-(\gamma) \zeta} P^\mathrm{cs}(\gamma) + e^{\mathcal{M}(\gamma^2) \zeta} P^\mathrm{ss}(\gamma^2) \right) 1_{\{\zeta \geq 0\}} + \left( e^{\nu^+(\gamma) \zeta} P^\mathrm{cu}(\gamma) + e^{\mathcal{M}(\gamma^2) \zeta} P^\mathrm{uu}(\gamma^2) \right) 1_{\{\zeta < 0\}}, \label{e: T gamma decomp}
\end{align}
where $P^\mathrm{ss/uu}(\gamma^2)$ are the spectral projections onto the stable/unstable eigenvalues of $\mathcal{M}(\gamma^2)$ which are bounded away from the imaginary axis for $\gamma$ small. Using the Dunford integral, one readily sees that these projections are analytic in $\gamma^2$ in a neighborhood of the origin. The leading order time dynamics correspond to the terms in $u(y;\gamma)$ which are most singular in $\gamma$. The only terms in \eqref{e: T gamma decomp} which are singular in $\gamma$ are those involving $P_\mathrm{pole}$. The linear estimates we state here will then be sharp provided the terms involving $P_\mathrm{pole}$ have a nontrivial contribution to $u(x;\gamma)$ through the formula \eqref{e: far field resolvent solution}, which is guaranteed by the following lemma, which we prove in Appendix \ref{app: sharpness}.
\begin{lemma}\label{l: sharpness}
	$\Pi_1 P_\mathrm{pole} \Lambda_1 \neq 0 $. 
\end{lemma}

The leading order dynamics are then approximately given by
\begin{align*}
T_\gamma(\zeta) \approx \frac{1}{\gamma} P_\mathrm{pole} \left( e^{\nu^+(\gamma) \zeta} 1_{\{\zeta < 0\}} + e^{\nu^-(\gamma) \zeta} 1_{\{\zeta \geq 0\}} \right). 
\end{align*} 
Using that $-\nu^+(\gamma) \approx  \nu^-(\gamma) \approx - \nu_1 \gamma$ from Lemma \ref{l: center projections pole}, we then have
\begin{align*}
	T_\gamma(\zeta) \approx \frac{1}{\gamma}  e^{-\nu_1 \gamma |\zeta|} P_\mathrm{pole} =: G^\mathrm{heat}_\gamma(\zeta) P_\mathrm{pole}
\end{align*}
with $\nu_1 = \sqrt{-d_{10}d_{02}^{01}}$. Note that the scalar function $G^\mathrm{heat}_\gamma$ is precisely, up to a constant multiple, the Laplace transform of the fundamental solution to the heat equation with an appropriate diffusion coefficient. By \eqref{e: far field resolvent solution}, the solution to $(\mcl_+ - \gamma^2) u = g$ is then given to leading order by
\begin{align*}
	u(y;\gamma) \approx \Pi_1 P_\mathrm{pole} \Lambda_1 \int_\R G^\mathrm{heat}_\gamma (y-\zeta) g(\zeta) \, d \zeta,
\end{align*}
thereby explaining the sense in which the leading order dynamics are governed by the heat equation.

These approximations are made rigorous, with error estimates, in \cite[Section 2]{SIMA}. We can now completely follow that argument, together with that of \cite[Proposition 3.2]{CAMS} establishing estimates on derivatives, to obtain the following description of the far-field resolvent $(\mcl_+ - \gamma^2)^{-1}$. 

\begin{prop}[Far-field resolvent estimates]\label{p: far field resolvent expansion}
	Let $r > 2$. There exist positive constants $C$ and $\delta$ and a limiting operator $R_0^+ : L^1_{1,1} (\R) \to W^{1,\infty}_{-1, -1} (\R)$ such that for all odd functions $g \in L^1_{1,1}(\R)$, we have
	\begin{align}
	\| (\mcl_+ - \gamma^2)^{-1} g - R_0^+ g \|_{W^{1,\infty}_{-1,-1}} &\leq C |\gamma| \| g \|_{L^1_{1,1}}, \\
	\| (\mcl_+ - \gamma^2)^{-1} g - R_0^+ g \|_{W^{1,1}_{-r, -r}} & \leq C |\gamma| \|g\|_{L^1_{1,1}}
	\end{align}
	for all $\gamma \in B(0,\delta)$ such that $\gamma^2$ is to the right of $\sigma_\mathrm{ess}(\mcl_+)$. 
	
	Furthermore, provided $\delta$ is sufficiently small, we have for all odd functions $g \in L^\infty_{r,r} (\R)$,
	\begin{align}
		\| \partial_x (\mcl_+-\gamma^2)^{-1} g \|_{L^1_{1,1}} &\leq \frac{C}{|\gamma|} \| g \|_{L^\infty_{r,r}}, \\
		\| \partial_x (\mcl_+ - \gamma^2)^{-1} g \|_{L^\infty_{r,r}} &\leq \frac{C}{|\gamma|^{r-1}} \| g \|_{L^\infty_{r,r}}
	\end{align}
	for all $\gamma \in B(0,\delta)$ such that $\Re \gamma \geq \frac{1}{2} | \Im \gamma|$. 
\end{prop}

The far-field/core decomposition used in \cite{SIMA, CAMS} to transfer these estimates to the full resolvent $(\mcl - \gamma^2)^{-1}$ relies only on Fredholm properties implied by Hypotheses \ref{hyp: dispersion reln}-\ref{hyp: point spectrum}. Repeating these arguments exactly, we then arrive at the following description of the full resolvent.
\begin{prop}[Full resolvent estimates]\label{p: full resolvent estimates}
	Fix $r > 2$. There exist positive constants $C$ and $\delta$ and a bounded limiting operator $R_0 : L^1_{0,1} (\R) \to W^{1,\infty}_{0, -1} (\R)$ such that
	\begin{align}
	\| (\mcl - \gamma^2)^{-1} g - R_0 g \|_{W^{1, \infty}_{0, -1}} \leq C |\gamma| \| g \|_{L^1_{0,1}}
	\end{align}
	for all $\gamma \in B(0,\delta)$ such that $\gamma^2$ is to the right of $\sigma_\mathrm{ess}(\mcl_+)$. 
	
	Furthermore, provided $\delta$ is sufficiently small, we have for all odd functions $g \in L^\infty_{0, r}(\R)$, 
	\begin{align}
	\| \partial_x (\mcl-\gamma^2)^{-1} g \|_{L^1_{0,1}} &\leq \frac{C}{|\gamma|} \| g \|_{L^\infty_{0, r}}, \\
	\| \partial_x (\mcl - \gamma^2)^{-1} g \|_{L^\infty_{0, r}} &\leq \frac{C}{|\gamma|^{r-1}} \| g \|_{L^\infty_{0,r}}
	\end{align}
	for all $\gamma \in B(0,\delta)$ such that $\Re \gamma \geq \frac{1}{2} | \Im |\gamma|$. 
\end{prop}

\subsection{Linear time decay estimates}
The assumption that all eigenvalues of $D$ are positive guarantees that the operator $\mcl : W^{2,p} (\R) \to L^p(\R)$ is sectorial for all $1 \leq p \leq \infty$; see for instance \cite[Chapter 3]{Lunardi}. We can therefore write the linear evolution via the inverse Laplace transform
\begin{align}
	e^{\mcl t} = - \frac{1}{2 \pi i} \int_\Gamma e^{\lambda t} (\mcl - \lambda)^{-1} \, d \lambda,
\end{align}
where $\Gamma$ is a sectorial contour to the right of the essential spectrum of $\mcl$. Note that for $p = \infty$, $\mcl$ is not densely defined, and so we rely on the characterization of analytic semigroups whose generators are not densely defined in \cite{Lunardi}. In particular, strong continuity at $t = 0$ holds only after preconditioning with the resolvent, i.e. we have
\begin{align}
\lim_{t \to 0^+} (\mcl - \lambda)^{-1} e^{\mcl t} u_0 = u_0 \label{e: initial data}
\end{align}
for any $\lambda$ in the resolvent set of $\mcl$. 

The resolvent estimates of the preceding section can be translated to sharp temporal decay estimates by shifting the contour $\Gamma$ to closely follow the essential spectrum of $\mcl$ near the origin, as in \cite{CAMS, SIMA}. Precisely following the arguments of \cite[Section 4]{CAMS}, we obtain the following estimates on $e^{\mcl t}$. 

\begin{prop}[Linear estimates]\label{p: linear estimates}
	Fix $r > 2$. There exists a constant $C > 0$ such that the following estimates hold.
	\begin{enumerate}[i.]
		\item (Large time estimates) For any $z_0 \in L^1_{0,1} (\R)$, we have
		\begin{align}
			\| e^{\mcl t } z_0 \|_{L^\infty_{0, -1}} \leq \frac{C}{t^{3/2}} \| z_0 \|_{L^1_{0,1}}
		\end{align}
		for all $t > 0$, while for any $z_0 \in L^\infty_{0,r}(\R)$, we have
		\begin{align}
			\| e^{\mcl t } z_0 \|_{L^\infty_{0, -1}} \leq \frac{C}{(1+t)^{3/2}} \| z_0 \|_{L^\infty_{0,r}}, \
		\end{align}
		for all $t > 0$. 
		\item (Large time derivative estimates) For any $z_0 \in L^\infty_{0,r}(\R)$, we have
		\begin{align}
			\| \partial_x e^{\mcl t} z_0 \|_{L^\infty_{0, r}} &\leq \frac{C}{t^{3/2-r/2}} \| z_0 \|_{L^\infty_{0, r}}, \\
			\| \partial_x e^{\mcl t} z_0 \|_{L^1_{0,1}} &\leq \frac{C}{t^{1/2}} \| z_0 \|_{L^\infty_{0, r}}
		\end{align}
		for all $t > 0$. 
		\item (Small time regularity estimates) For any $z_0 \in L^1(\R)$, we have
		\begin{align}
			\| e^{\mcl t} \|_{L^\infty} \frac{C}{t^{1/2}} \| z_0 \|_{L^1} 
		\end{align}
		for $0 < t < 2$. For any $z_0 \in L^\infty_r(\R)$, we have
		\begin{align}
			\| \partial_x e^{\mcl t} \|_{L^\infty_{0,r}} \leq \frac{C}{t^{1/2}} \| z_0 \|_{L^\infty_{0, r}}, 
		\end{align}
		for $0 < t < 2$. For any $z_0 \in L^1_{0,1} (\R)$, we have
		\begin{align}
					\| \partial_x e^{\mcl t} \|_{L^1_{0,1}} \leq \frac{C}{t^{1/2}} \| z_0 \|_{L^1_{0, 1}}, 
		\end{align}
		for $0 < t < 2$. 
	\end{enumerate}
\end{prop}

\section{Stability argument and consequences for front propagation}\label{s: stability argument}
We prove Theorem \ref{t: main} by obtaining global in time control of perturbations to the approximate solution $v^\mathrm{app}$, which was constructed in Sections \ref{s: colinear} in the case where $u_0$ and $u_1$ are co-linear and in Section \ref{s: linearly independent} in the case where $u_0$ and $u_1$ are linearly independent. Having established good estimates on the residual error of $v^\mathrm{app}$ in Proposition \ref{p: colinear residual} and Proposition \ref{p: lin ind residual estimate}, and sharp estimates on $e^{\mcl t}$ in either case in Proposition \ref{p: linear estimates}, the proof of Theorem \ref{t: main} now follows \emph{exactly} as in \cite{CAMS}. For completeness, we sketch the proof here. The argument from here on is identical regardless of whether $u_0$ and $u_1$ are co-linear or linearly independent, and so we no longer distinguish between these cases.

\subsection{Stability argument}
We let $v$ solve $F_\mathrm{res} [v] = 0$, which is the original equation \eqref{e: rd} in the moving frame accounting for the linear spreading speed and the logarithmic delay, and after conjugation with the exponential weight $\omega$. We let $v^\mathrm{app}$ be the approximate solution constructed in either Section \ref{s: colinear} or Section \ref{s: linearly independent}. We then define the perturbation $w = v - v^\mathrm{app}$, which satisfies
\begin{multline}
	w_t = \mcl w - f'(q_*) w - \frac{3}{2 \eta_* (t+T)} [ \omega (\omega^{-1})' w + w_y] + \omega f(\omega^{-1} w + \omega^{-1} v^\mathrm{app}) - \omega f(\omega^{-1} v^\mathrm{app}) - R, \label{e: w eqn}
\end{multline}
where $R(y,t) = F_\mathrm{res}[v^\mathrm{app}](y,t)$. Note that we have added and subtracted $f'(q_*)w$ so that the principal linear part of this equation can be written as $\mcl w$. We want to view this equation as a perturbation of $w_t = \mcl w$ and aim to control all other terms via a nonlinear iteration argument on the associated variation of constants formula. 

We define
\begin{align}
N(\omega^{-1} w; y, t) = f(\omega^{-1} (w + v^\mathrm{app})) - f(\omega^{-1} v^\mathrm{app}) - f'(\omega^{-1} v^\mathrm{app}) \omega^{-1} w, 
\end{align}
so that we may rewrite \eqref{e: w eqn} as 
\begin{align}
w_t = \mcl w -  \frac{3}{2 \eta_* (t+T)} [ \omega (\omega^{-1})' w + w_y] + (f'(\omega^{-1} v^\mathrm{app}) - f'(q_*)) w - R + \omega N(\omega^{-1} w). \label{e: w eqn rewritten}
\end{align}
We usually write $N(\omega^{-1} w; y, t) = N(\omega^{-1} w)$, suppressing the explicit dependence on time and space. Note that $\omega^{-1} v^\mathrm{app}$ is uniformly bounded, so by Taylor's theorem there exists non-decreasing function $K: \R_+ \to \R_+$ such that 
\begin{align}
|\omega N(\omega^{-1} w)| \leq K(B) \omega^{-1} | w|^2 \label{e: nonlinearity Taylors thm}
\end{align}
provided $\|\omega^{-1} w\|_{L^\infty} \leq B.$ 
We also have
\begin{align}
\| (f'(\omega^{-1} v^\mathrm{app}) - f'(q_*)) w \| \leq C | \omega^{-1} v^\mathrm{app}-q_* | |w|
\end{align}
by smoothness of $f$ together with Corollaries \ref{c: colinear front approximated by vapp} and \ref{c: lin ind front approximated by vapp}, which imply that $| \omega^{-1} v^\mathrm{app} - q_*| \ll 1$ provided $T$ is large. 

The difficulty with viewing \eqref{e: w eqn rewritten} as a perturbation of $w_t = \mcl w$ is that the term $-\frac{3}{2 \eta_* (t+T)} \omega (\omega^{-1})' w$ is \emph{critical}, and prevents solutions even with localized initial data from decaying in time, while localized solutions to $w_t = \mcl w$ decay with rate $t^{-3/2}$. In order to be able to close a perturbative argument, we make the change of variables $z (y,t) = (t+T)^{-3/2} w(y,t)$ as in \cite{CAMS}, and find that $z$ solves
\begin{multline}
	z_t = \mcl z - \frac{3}{2 \eta_* (t+T)} [\omega (\omega^{-1})' + \eta_*] z - \frac{3}{2 \eta_* (t+T)} z_x + (f'(\omega^{-1} v^\mathrm{app}) - f'(q_*))z- (t+T)^{-3/2} R \\ + (t+T)^{-3/2} \omega N (\omega^{-1} (t+T)^{3/2} z). \label{e: z eqn}
\end{multline}
We note that the equation \eqref{e: z eqn} is locally well-posed in, for instance, $L^\infty_{0,-1}(\R)$ by standard theory of semilinear parabolic equations \cite{Henry, Lunardi}, with the understanding that the initial data is attained in the sense of \eqref{e: initial data}. 

The term $-\frac{3}{2 \eta_* (t+T)} [\omega (\omega^{-1})' + \eta_*] z$ is now supported only for $x \leq 1$, since $\omega(x) (\omega^{-1})'(x) = - \eta_*$ for $x \geq 1$. Since the dynamics of \eqref{e: z eqn} are dominated by the behavior as $x \to \infty$, this term is essentially removed from the equation, and in particular no longer obstructs decay. We are then able to prove, following \cite[Section 5]{CAMS}, that $\| z(t) \|_{L^\infty_{0, -1}} \sim (t+T)^{-3/2}$, so that $w = (t+T)^{3/2} z$ remains small for all time. 

The stability argument for the $z$-equation is not simple, even though we have removed the most critical term by changing from the $w$ to the $z$ equation. For instance, the nonlinear terms in the $z$ equation roughly satisfy
\begin{align*}
	(t+T)^{-3/2} \omega N(\omega^{-1} (t+T)^{3/2} z ) \sim (t+T)^{3/2} \omega^{-1} z^2,
\end{align*}
so that the nonlinearity carries a temporally growing coefficient and hence is now marginally relevant, since it is of the same order as $z$ if $z$ has the expected decay rate $z \sim (t+T)^{-3/2}$. Simultaneously handling this as well as the linear term $-\frac{3}{2 \eta_*(t+T)} z_x$, which is difficult to treat perturbatively since it does not gain any spatial localization, requires several bootstrap steps. In the end, we close the argument by controlling the norm template function
\begin{multline}
\Theta(t) = \sup_{0 < s < t} \bigg[ (s+T)^{3/2} \| z(s) \|_{L^\infty_{0, -1}} + 1_{\{0< s < 1\}} T^{3/2} s^{1/2} \bigg(  \| z_x (s)\|_{L^1_{0,1}} + \| z_x (s) \|_{L^\infty_{0, r}} \bigg) \\ + 1_{\{s \geq 1\}} T^{1/2} \bigg( (s+T)^{1/2} \| z_x (s)\|_{L^1_{0,1}} + (s+T)^\beta \| z_x \|_{L^\infty_{0,r}} \bigg) \bigg],
\end{multline}
where $r = 2 + \mu$ with $ 0 < \mu < \frac{1}{8}$ fixed, and $\beta = \frac{1}{2} - \frac{\mu}{2}$. Following the argument of \cite{CAMS}, we arrive at the following control of $\Theta(t)$.

\begin{prop}\label{p: theta control}
	Let $r = 2 + \mu$ with $0 < \mu < \frac{1}{8}$, and let $z$ solve \eqref{e: z eqn} with initial data $z_0 \in L^\infty_{0, r}(\R)$. Let $t_* \in (0, \infty]$ denote the maximal time of existence of $z(t)$ in $L^\infty_{0, -1}(\R)$. Define
	\begin{align}
	R_0 = \sup_{T\geq T_*} \sup_{s > 0} (s+T)^{1/2-4\mu} \| R(s; T) \|_{L^\infty_{0, r}},\label{e: R0 def}
	\end{align}
	which is finite for some $T_*$ sufficiently large by Corollary \ref{c: colinear normal form residual} in the co-linear case and Proposition \ref{p: lin ind residual estimate} in the linearly independent case. 
	There exist constants $C_0, C_1$, and $C_2$ independent of $z_0$ such that 
	\begin{align}
	\Theta(t) \leq C_0 \left( T^{3/2} \| z_0 \|_{L^\infty_{0,r}} + T^{-1/2+4\mu} R_0 \right) + \frac{C_1}{T^{1/2-4\mu}} \Theta(t) + C_2 K (B \Theta(t)) \Theta(t)^2
	\end{align}
	for all $t > 0$, and $T$ sufficiently large, where $B = \| \rho_{0,1} \omega^{-1} \|_{L^\infty}$ and $K$ is defined by \eqref{e: nonlinearity Taylors thm}. 
\end{prop}
\begin{proof}
	Note that \eqref{e: z eqn} has the exact same form as \cite[equation (5.4)]{CAMS}. The proof of \cite[Proposition 5.2]{CAMS} relies only on estimates on $R$, $| \omega^{-1} v^\mathrm{app} - q_*|$ and $e^{\mcl t}$ which are here encoded in Corollary \ref{c: colinear normal form residual}, Proposition \ref{p: lin ind residual estimate}, Corollary \ref{c: colinear front approximated by vapp}, Corollary \ref{c: lin ind front approximated by vapp}, and Proposition \ref{p: linear estimates}. We then obtain the proposition by applying exactly the proof of \cite[Proposition 5.2]{CAMS}. 
\end{proof}

Applying a standard nonlinear iteration argument (see for instance \cite[Section 5.4]{CAMS}), we then obtain the following control of $z(t)$.

\begin{corollary}
	Fix $0 < \mu < \frac{1}{8}$ and let $r = 2 + \mu$. Let $R_0$ be defined by \eqref{e: R0 def}, and let $T \geq T_*$ so that $R_0$ is finite. There exist positive constants $C$ and $\eps$ such that for all $z_0 \in L^\infty_{0, r} (\R)$ with 
	\begin{align}
	T^{3/2} \| z_0 \|_{L^\infty_{0,r}} + T^{-1/2 + 4 \mu} R_0 < \eps,
	\end{align}
	then the solution $z(t)$ to \eqref{e: z eqn} with initial data $z_0$ exists globally in time in $L^\infty_{0,-1}(\R)$ and satisfies
	\begin{align}
	\| z(t) \|_{L^\infty_{0, -1}} \leq \frac{C}{(t+T)^{3/2}} \left( T^{3/2} \|z_0 \|_{L^\infty_{0,r}} + T^{-1/2 + 4 \mu} R_0 \right). 
	\end{align}
\end{corollary}
Note that $T^{3/2} z_0 = w_0$, so assuming smallness of $T^{3/2} z_0$ is just assuming smallness of $w_0$. Indeed, undoing the change of variables $z = (t+T)^{-3/2} w$, we obtain the following control of $w$. 
\begin{corollary}\label{c: w control}
	Fix $0 < \mu < \frac{1}{8}$, let $r = 2 + \mu$, and assume $T \geq T_*$ with $T_*$ as in Proposition \ref{p: theta control}. There exist positive constants $C$ and $\eps$ such that if 
	\begin{align}
	\| w_0 \|_{L^\infty_{0, r}} + T^{-1/2+4\mu} R_0 < \eps,
	\end{align}
	then the solution $w(y,t)$ to \eqref{e: w eqn rewritten} exists for all positive time and satisfies
	\begin{align}
	\| w(t) \|_{L^\infty_{0, -1}} \leq C \left(\| w_0 \|_{L^\infty_{0, r}} + T^{-1/2 + 4 \mu} R_0 \right)
	\end{align}
	for all $t > 0$. 
\end{corollary}

\subsection{Consequences for front propagation --- proof of Theorem \ref{t: main}}
The proof of Theorem \ref{t: main} from this point is identical to \cite[Section 6]{CAMS}, but we reproduce it for completeness. Recall that $w$ is defined through
\begin{align}
\omega (y) U(y,t) = v^\mathrm{app}(y,t; T) + w(y,t),
\end{align}
where $U(y,t)$ solves \eqref{e: U eqn}, which is just the original equation \eqref{e: rd} in the moving frame defined by \eqref{e: y def}. The control of $w$ in Corollary \ref{c: w control} then translates to the following description of $U$.
\begin{corollary}\label{c: U control}
	Fix $0 < \mu < \frac{1}{8}$, let $r = 2 + \mu$, and assume $T \geq T_*$ with $T_*$ as in Proposition \ref{p: theta control}. Let $U(y,t)$ solve \eqref{e: U eqn} with initial data $U_0$. There exist positive constants $C$ and $\eps$ such that if 
	\begin{align}
	\| \omega U(\cdot, t) - v^\mathrm{app} (\cdot, 0; T) \|_{L^\infty_{0, r}} + T^{-1/2 + 4\mu} R_0 < \eps, 
	\end{align}
	then 
	\begin{align}
	\| \omega U(\cdot, t) - v^\mathrm{app}(\cdot, t; T) \|_{L^\infty_{0, -1}} \leq C \left( \| \omega U_0 - v^\mathrm{app}(\cdot, 0; T) \|_{L^\infty_{0, r}} + T^{-1/2 + 4 \mu} R_0 \right) 
	\end{align}
	for all $t > 0$. 
\end{corollary}
\begin{proof}[Proof of Theorem \ref{t: main}]
	Fix $\eps > 0$ small enough so that Corollary \ref{c: U control} holds. Then choose $T \geq T_*$ large enough so that $T^{-1/2 + 4 \mu} R_0 < \frac{\eps}{4}$. Define
	\begin{align}
	\mathcal{U}_\eps = \left\{ U_0 : \omega U_0 \in L^\infty_{0, r} (\R) \text{ with } \| \omega U_0 - v^\mathrm{app}(0,0;T) \|_{L^\infty_{0, r}} + T^{-1/2+4 \mu} R_0 < \frac{\eps}{2} \right\}. \label{e: U eps def}
	\end{align}
	Notice that $\mathcal{U}_\eps$ is clearly open in the norm $\| U_0 \| = \| \rho_{0,r} \omega U_0 \|_{L^\infty}$. To verify that $\mathcal{U}_\eps$ contains steep initial data, we define
	\begin{align}
	U_0^*(y) = \begin{cases}
	\frac{1}{\omega(y)} v^\mathrm{app} (y,t), & y < T^{1/2 + \mu} - y_0, \\
	0, & y \geq T^{1/2+\mu} - y_0. 
	\end{cases}
	\end{align}
	From the analysis of Sections \ref{s: colinear} and \ref{s: linearly independent}, we recall that 
	\begin{align*}
	|v^\mathrm{app} (y,t)| \leq C e^{-y^2/[8 \Deff (t+T)]} 
	\end{align*}
	for $y \geq (t+T)^\mu$, which implies that
	\begin{align*}
	\| \omega U_0^* (\cdot) - v^\mathrm{app} (\cdot, 0; T) \|_{L^\infty_{0, r}} \leq C T^{(\frac{1}{2} + \mu)(2 + \mu)} e^{-c_1 T^{2\mu}}
	\end{align*}
	for some constants $C, c_1 > 0$. Hence, provided $T$ is sufficiently large, we have $U_0^* \in \mathcal{U}_\eps$. 
	
	Let $u$ solve the original system \eqref{e: rd} with initial data $u_0$, and recall that $u(y+\tilde{\sigma}_T (t), t) = U(y,t)$, with 
	\begin{align}
	\tilde{\sigma}_T (t) = c_* t - \frac{3}{2 \eta_*} \log (t+T) + \frac{3}{2 \eta_*} \log T. 
	\end{align}
	Applying Corollary \ref{c: U control}, we then obtain
	\begin{align}
	\sup_{y \in \R} \left| \rho_{0,-1} (y) \omega (y) \left( u(y+\tilde{\sigma}_T (t), t) - \omega(y)^{-1} v^\mathrm{app}(y, t; T) \right)\right| < \frac{\eps}{2} 
	\end{align}
	for all $u_0 \in \mathcal{U}_\eps$. By Corollary \ref{c: colinear front approximated by vapp} in the co-linear case or Corollary \ref{c: lin ind front approximated by vapp} in the linearly independent case, $v^\mathrm{app}$ is a good approximation to the critical front $q_*$, so that
	\begin{align}
	\sup_{y \in \R} \left| \rho_{0,-1} (y) \omega (y) \left( u(y+\tilde{\sigma}_T (t), t) - q_*(y)\right)\right| < \frac{3 \eps}{4} \label{e: thm 1 proof 1}
	\end{align}
	provided $T$ is sufficiently large. Defining
	\begin{align}
	\sigma(t) = c_* t - \frac{3}{2 \eta_*} \log t + \frac{3}{2 \eta_*} \log T,
	\end{align}
	we notice that for any fixed $T$
	\begin{align}
	\left| \sigma(t) - \tilde{\sigma}_T (t) \right| = \left| \log \left( \frac{1}{1+ t/T} \right) \right| \to 0 \text{ as } t \to \infty. 
	\end{align}
	Since $u$ is smooth for $t > 0$ by parabolic regularity, we can use the mean-value theorem to replace $\tilde{\sigma}_T(t)$ with $\sigma(t)$ in the estimate \eqref{e: thm 1 proof 1} for $T$ sufficiently large, obtaining
	\begin{align}
	\sup_{y \in \R} \left| \rho_{-1} (y) \omega (y) \left( u(y+\sigma(t), t) (t), t) - q_*(y)\right)\right| < \eps
	\end{align}
	for $t$ sufficiently large depending on $T$.  This is the estimate of Theorem \ref{t: main}, with 
	\begin{align}
	x_\infty (u_0) = - \frac{3}{2 \eta_*} \log T. 
	\end{align}
	This estimate applies for all $u_0 \in \mathcal{U}_\eps$ provided $t$ and $T$ are sufficiently large, and we have already verified that $\mathcal{U}_\eps$ is open and contains steep initial data, and hence the proof of Theorem \ref{t: main} is complete. 
\end{proof}

\section{Examples and discussion}\label{s: discussion} 
{\color{black} We now explore specific models in which our hypotheses may be verified, and conclude with a discussion of extensions of our results. }

\subsection{Amplitude equations with parametric forcing.} The complex Ginzburg-Landau equation
\begin{align}
A_t = (1+i \alpha) A_{xx} + \mu A -(1+i \gamma) A|A|^2, \quad A = A(x,t) \in \C \label{e: CGL}
\end{align}
is a universal modulation equation describing weakly nonlinear spatiotemporal dynamics in many spatially extended systems \cite{AransonKramer}. The real coefficient version, with $\alpha = \gamma = 0$, describes leading order dynamics near a Turing instability in a pattern-forming system \cite{CrossHohenberg}, and many existing predictions on the behavior of pattern-forming invasion fronts are based on the Ginzburg-Landau approximation \cite{vanSaarloosReview}. Rigorously establishing front selection in \eqref{e: CGL} and related pattern-forming models is a long-standing open problem \cite{vanSaarloosReview, deelanger, colleteckmann, ColletEckmannSH}. The results in this paper do not directly apply to \eqref{e: CGL}, since states selected in the wake of invasion fronts in \eqref{e: CGL} have only marginally stable essential spectrum, due to invariance under the gauge symmetry $A \mapsto e^{i \phi} A$, so that Hypothesis \ref{hyp: stability on left} is violated. We nonetheless expect the methods here to be of use in analyzing \eqref{e: CGL} and other pattern-forming systems. See \cite{AveryScheelGL} for some progress in this direction which establishes sharp decay estimates for localized perturbations to critical fronts in \eqref{e: CGL} with $\alpha = \gamma = 0$. 

Periodic forcing of pattern-forming systems has been of substantial interest as a method to experimentally control patterns and produce new types of coherent structures \cite{Parametric1, Parametric2, Parametric3, ParametricForcingLiquidCrystal, ParametricForcingRayleighBenard}. Properly tuned periodic forcing resonant to the natural system frequencies introduces terms to \eqref{e: CGL} which break the gauge symmetry. For a concrete example, consider the parametrically forced Swift-Hohenberg equation
\begin{align}
u_t = -(1+\partial_x^2) u + \mu u - u^3 + \beta \cos (k x) u,  \quad u = u(x,t) \in \R, \quad \mu > 0,
\end{align}
a prototypical model for dynamics near a pattern-forming Turing instability. Making the formal ansatz $u(x,t) = \eps A(\eps^2 x, \eps t) e^{i x} + \bar{A}(\eps^2 x, \eps t) e^{-ix}$ with $\eps = \sqrt{\mu}$, one finds the parametrically forced Ginzburg-Landau equation
\begin{align}
A_t = A_{xx} + A -A|A|^2 + \beta \bar{A}^{k-1}, \quad A = A(x,t) \in \C \label{e: GL parametric}
\end{align}
as a solvability condition, after rescaling coefficients \cite{ParametricMeron}. The parameter $k \in \Z$ is chosen to be a resonant multiple of the linearly selected wavenumber $1$ in the Swift-Hohenberg equation. The most relevant resonance in experiments is often the $2:1$ resonance, so we focus on the case $k = 2$.
We let $A = u + iv$ to obtain the two-component system
\begin{align}
u_t &= u_{xx} + u - u(u^2 + v^2) + \beta u \label{e: GL parametric u} \\
v_t &= v_{xx} + v - v(u^2 + v^2) - \beta v. \label{e: GL parametric v}
\end{align}
Without loss of generality, we restrict to $\beta \geq 0$, since $\beta \mapsto -\beta$ just swaps $u$ and $v$.
\begin{thmlocal}
	The system \eqref{e: GL parametric u}-\eqref{e: GL parametric v} satisfies Hypotheses \ref{hyp: dispersion reln} through \ref{hyp: point spectrum} for all $\beta > 0$. 
\end{thmlocal}

\begin{proof}
	From a short calculation, one finds the linear spreading speed $c_* = 2 \sqrt{1 + \beta}$ and concludes that Hypothesis \ref{hyp: dispersion reln} is satisfied for all $\beta > 0$: when $\beta = 0$, the dispersion relation in the leading edge instead has the form $d_{c_*} (\lambda, \nu) = (\lambda-(\nu-\nu_*)^2)^2$ with $c_* = 2$, while for $\beta \neq 0$ this double double root splits into two simple pinched double roots, and choosing $c_* = 2 \sqrt{1+\beta}$ we find one marginally stable and one stable pinched double root. 
	
	The real subspace $v= 0 $ is invariant in \eqref{e: GL parametric u}-\eqref{e: GL parametric v}, and so we find a critical front solution $(u, v) = (q_*, 0)$, where $q_*$ solves the Fisher-KPP type traveling wave equation
	\begin{align}
	q_*'' + c_* q_*' + (1+\beta) q_* - q_*^3 = 0, \quad \lim_{x \to - \infty} q_*(x) = \sqrt{1 + \beta}, \quad \lim_{x \to \infty} q_*(x) = 0. 
	\end{align}
	This front $q_*$ can be constructed via classical phase plane methods, and satisfies the generic asymptotics $q_*(x) \sim (ax + b)e^{-\sqrt{1+\beta} x}$, so that Hypothesis \ref{hyp: front existence} is satisfied for \eqref{e: GL parametric} for all $\beta > 0$. From a short calculation, we find that the essential spectrum of the linearization about the state $u_- = (\sqrt{1+\beta},0)$ selected in the wake of $q_*$ is strongly stable for all $\beta > 0$, so that Hypothesis \ref{hyp: stability on left} is satisfied. We emphasize that spectral stability in the wake is a consequence of breaking the gauge symmetry of \eqref{e: CGL}. 
	
	It only remains to verify Hypothesis \ref{hyp: point spectrum}. The linearization about the critical front $(q_*, 0)$ in the frame moving with the linear spreading speed is given by
	\begin{align}
	\mathcal{A} = \begin{pmatrix}
	\partial_{xx} + c_* \partial_x + (1+ \beta - 3 q_*^2) & 0 \\
	0 & \partial_{xx}+ c_* \partial_x + (1 - \beta - q_*^2) 
	\end{pmatrix} =: \begin{pmatrix}
	\mathcal{A}_u & 0 \\ 0 & \mathcal{A}_v
	\end{pmatrix}. 
	\end{align}
	It follows from translation invariance of \eqref{e: GL parametric u}-\eqref{e: GL parametric v} that $\mathcal{A}_u q_*' = 0$, and we also observe that $(\mathcal{A}_v+2 \beta) q_*' = 0$. We define the critical weight $\omega$ by \eqref{e: omega def} with $\eta_* = \sqrt{1+\beta}$, and the associated conjugate operators
	\begin{align}
	\mcl = \begin{pmatrix}
	\mcl_u & 0 \\ 0 & \mcl_v
	\end{pmatrix} := 
	\begin{pmatrix}
	\omega \mathcal{A}_u \omega^{-1} & 0 \\
	0 & \omega \mathcal{A}_v \omega^{-1}
	\end{pmatrix}. 
	\end{align}
	Observe that $\mcl_u [\omega q_*'] = (\mcl_v + 2 \beta) [\omega q_*'] = 0$, and that $\omega q_*'$ is non-vanishing. It follows from a Sturm-Liouville argument (see \cite[Theorem 5.5]{Sattinger}) that $\mcl_u$ has no unstable point spectrum --- on, say, $L^2 (\R)$ --- and $\mcl_v$ has no eigenvalues $\lambda$ with $\Re \lambda > - 2 \beta$. Using basic ODE techniques, one can further show that there are no solutions to $\mcl_u u = 0$ or $\mcl_v v = 0$ which are bounded on $\R$. In the former case, this follows from the fact that $\omega q_*' (x) \sim x, x \to \infty$ is the unique solution up to a constant multiple which is bounded on the left. In the latter case, $\lambda = 0$ is to the right of the essential spectrum of $\mcl_v$, and so bounded solutions to $\mcl_v v  = 0$ are necessary exponentially localized and hence must be eigenfunctions, which we have already excluded. We conclude that the system \eqref{e: GL parametric u}-\eqref{e: GL parametric v} satisfies Hypothesis \ref{hyp: point spectrum}, which completes the proof of the theorem. 
\end{proof}

\subsection{Competitive Lotka-Volterra systems.} A large body of mathematical work on invasion processes is motivated by ecological systems, and in particular understanding the spread of invasive species. As an illustrative example, we consider here the Lotka-Volterra competition model
\begin{align}
	u_t &= u_{xx} + u (1-u-a_1 v) \label{e: LV 1} \\
	v_t &= \sigma v_{xx} + r v (1-a_2 u -v). \label{e: LV 2}
\end{align}
When $a_1 < 1 < a_2$, the uniform equilibrium $(u, v) \equiv (0, 1)$ is unstable, and one finds the linear spreading speed $c_* = 2 \sqrt{1-a_1}$. Owing to the special competitive structure of the reaction terms, \eqref{e: LV 1}-\eqref{e: LV 2} retains a comparison principle which can be used to estimate propagation speeds from steep data \cite{Alhasanat, LLWCompetitive, LLWCooperative}. In particular, it was shown in \cite{LLWCompetitive} that propagation occurs with the linear spreading speed for 
\begin{align}
0 < a_1 < 1 < a_2, \quad 0 < \sigma < 2, \quad (a_1 a_2 - M) r \leq M (2-\sigma) (1-a_1), \label{e: LV pulled assumption}
\end{align}
where $M = \max \{1, 2(1-a_1)\}$. It was shown in \cite{FayeHolzerLV} that Hypotheses \ref{hyp: dispersion reln} through \ref{hyp: point spectrum} are satisfied for \eqref{e: LV 1}-\eqref{e: LV 2} provided \eqref{e: LV pulled assumption} holds, and so our results recover propagation at the linear spreading speed in this regime. 

It was conjectured in \cite{Hosono} that the condition \eqref{e: LV pulled assumption} is asymptotically sharp in the large $r$ limit: for instance, fixing $a_1 = \frac{1}{2}, \sigma = 1$, it was conjectured that there exists a function $\Lambda(r) = 2 + \mathrm{o}(1), r \to \infty$ such that the invasion process in \eqref{e: LV 1}-\eqref{e: LV 2} is pulled for $a_2 < \Lambda(r)$, and pushed for $a_2 > \Lambda(r)$. Numerical evidence for this conjecture is not entirely conclusive, however; see \cite[Section 6]{AveryHolzerScheel}. 

One may ask why it is useful to discuss \eqref{e: LV 1}-\eqref{e: LV 2} in the context of our approach, when this system admits a comparison principle and so is amenable to more classical approaches. We find value in our approach even in this setting for the following reasons:
\begin{enumerate}
	\item Our results apply to open classes of systems, and so apply to perturbations of \eqref{e: LV 1}-\eqref{e: LV 2} which incorporate non-competitive effects and hence break the comparison structure. Indeed, it is acknowledged in \cite{LLWCooperative} that real ecological systems are rarely purely competitive. 
	\item Even in the purely competitive setting, it may be easier to verify our spectral assumptions using monotonicity properties of the associated eigenvalue problem than to construct sufficiently detailed super- and sub- solutions to the full nonlinear problem. 
	\item One may in principal verify our conceptual assumptions using geometric dynamical systems methods including Evans function techniques and geometric singular perturbation theory (in appropriate limits). Our results thereby open problems of invasion processes in ecological systems to a wider range of approaches. For instance, we expect that a singular perturbation approach may be useful in verifying our assumptions in the large $r$ limit, which would then allow one to rigorously determine the asymptotics of the pushed-pulled transition curve $\Lambda(r)$. 
\end{enumerate}

\subsection{A model for tumor growth}

The following model for dynamics of tumors driven by cancer stem cells was introduced in \cite{Tumor1}
\begin{align}
	u_t &= D u_{xx} + p_s \gamma_u F(u+v) u \label{e: tumor 0 1} \\
	v_t &= D v_{xx} + (1-p_s) \gamma_u F(u+v) u + \gamma_v F(u+v) v - \alpha v. \label{e: tumor 0 2}
\end{align}
The variable $u$ denotes the concentration of cancer stem cells, which reproduce at rate $\gamma_u$ by dividing into either two cancer stem cells, with probability $p_s$, or one cancer stem cell and one tumor cell, with probability $1-p_s$. The variable $v$ denotes the concentration of tumor cells, which can only reproduce by splitting into two tumor cells, with rate $\gamma_v$, and die off at rate $\alpha$ (for instance, due to an external treatment which targets tumor cells rather than cancer stem cells). The function $F$ models competition for resources between tumor cells and cancer stem cells, and is assumed to be monotonically decreasing and satisfy $F(0) = 1$. 

Of particular interest is the \emph{tumor invasion paradox}, where increasing the tumor death rate $\alpha$ (for instance, by using a more aggressive treatment regimen) may actually accelerate spatial spread of the tumor \cite{Tumor2}. Experimental observations indicate that the cancer stem self self-renewal probability $p_s$ is often small \cite{Tumor2}. Setting $p_s = \eps$, assuming cell movement is slow with $D = \eps d$, assuming $\gamma_u = \gamma_v = 1$, and defining $\tau = \eps t$, we find 
\begin{align}
u_\tau &= d u_{xx} + F(u+v) u, \nonumber \\
\eps v_\tau &= \eps d v_{xx} + (1-\eps) F(u+v) u + F(u+v) v - \alpha v. \label{e: tumor 2}
\end{align}
Setting $\eps = 0$, one may solve the second, algebraic equation for $v = v_\alpha (u)$, and inserting this into the first equation leads \cite{Tumor2} to the Fisher-KPP type equation
\begin{align}
u_\tau = d u_{xx} + F(u +v_\alpha(u)) u. \label{e: tumor kpp}
\end{align}
In \cite{Tumor1}, it is shown that the propagation speed in \eqref{e: tumor kpp} is
\begin{align}
	c_*^\mathrm{kpp} (\alpha) = \begin{cases}
	2 \sqrt{d}, &\alpha \geq 1, \\
	2 \sqrt{d \alpha}, & 0 < \alpha < 1.  
	\end{cases}	
\end{align}
One then sees the tumor invasion paradox for $0 < \alpha < 1$: in this regime, increasing $\alpha$ increases the invasion speed. The analysis of \cite{Tumor2} does not, however, extend to the case $\eps \neq 0$. The authors there view the reduction to \eqref{e: tumor kpp} as a slow manifold reduction, but since there is not a well-developed geometric singular perturbation theory for partial differential equations, it is not clear how to track a slow manifold for $\eps \neq 0 $ in this setting. 

On the other hand, Theorem \ref{t: main} reduces predicting invasion speeds from an infinite-dimensional PDE problem to finite dimensional ODEs capturing existence and spectral stability of invasion fronts, for which there is a well-developed geometric singular perturbation theory \cite{Fenichel}. {\color{black}Using this approach, one can prove that for fixed $0 < \alpha < 1$ or $\alpha > 1$, \eqref{e: tumor 2} satisfies Hypotheses \ref{hyp: dispersion reln} through \ref{hyp: point spectrum} for $\eps$ sufficiently small, and thereby rigorously describe this tumor growth process for $\eps \neq 0$ \cite{AveryTumor}.}

\subsection{Discussion}
{\color{black} We now comment on potential extensions of our results, and challenges therein. 

\noindent \textbf{Global versus local aspects.} As mentioned in the introduction, in scalar equations global convergence results to fronts for non-negative initial data are often available through the use of the comparison principle. Even in scalar equations, the situation becomes more subtle when one allows for sign-changing initial data. Considering the Allen-Cahn equation $u_t = u_{xx} + u - u^3$, there are two pulled fronts which connect to $u = 0 $ in the leading edge: one positive front which converges to $u = 1$ in the wake, and a negative front which converges to $u = -1$ in the wake. Since pulled fronts are governed by their tail dynamics, initial data which appear positive to the naked eye, converging to $u = 1$ in the wake, but have a small, visibly imperceptible negative tail may in fact converge locally uniformly to the negative front. A kink between the $u = 1$ and $u = -1$ states then develops in the wake and slowly propagates to the left. This phenomenon can be observed in direct numerical simulations. 

This example demonstrates that when one no longer relies on comparison principles and preservation of positivity, global aspects of front convergence become much subtler. It is then quite natural that, exploiting only spectral properties and not the specific structure of the equation, our convergence result is local in nature. 
}

\noindent \textbf{Towards modulated and pattern-forming fronts.} {\color{black} Invasion dynamics are often observed shortly after the onset of instability in some bifurcation. In \cite{AverySmallAmpFronts}, we show that pulled fronts emerge and govern spreading dynamics near the onset of instability in a transcritical, saddle-node, or supercritical pitchfork bifurcation}. On the other hand, instabilities in spatially extended systems may develop in more complicated manners, such as in a Turing bifurcation, which leads to creation of periodic patterns with a selected wavenumber in the wake of the invasion process \cite{vanSaarloosReview}. There are two main difficulties in extending our methods to invasion processes governed by Turing instabilities. The first is that the dynamics in the leading edge now become time-periodic rather than stationary in the co-moving frame, leading to the formation of \emph{modulated} invasion fronts \cite{EbertvanSaarloos, EckmannWayne} which are also time-periodic in the moving frame. The linear analysis of Section \ref{s: linear estimates} would then need to be adapted to handle time-periodic coefficients. We expect this can be done, for instance, by extending work in \cite{MargaretBjornKevinTimePeriodic} which develops a pointwise semigroup approach for nonlinear stability of time-periodic Lax shocks. The other difficulty is that the patterns in the wake are only diffusively stable, and so the stability of the patterns in the wake themselves is itself a difficult problem; see \cite{jnrz1, jnrz2, zumbruninventiones, uecker}. Reconciling this delicate stability argument with the lack of decay resulting from the matching with the diffusive tail in the leading edge is the most difficult part of establishing front selection in the wake of a Turing instability, and we leave further discussion to future work. {\color{black} Recent work on local stability of pulled pattern-forming fronts in the FitzHugh-Nagumo system gives some progress in this direction \cite{AveryCarterScheeldeRijkNonlinear}.}

\noindent \textbf{Declarations.} This work was supported by the National Science Foundation through NSF-DMS-2202714. Any opinions, findings, and conclusions or recommendations expressed in this material are those of the author and do not necessarily reflect the views of the National Science Foundation. The authors have no other
competing interests that are relevant to the content of this article.

\appendix

\section{Higher order terms in the diffusive equation --- $u_0, u_1$ linearly independent}\label{app: lin ind}

\subsection{Expressions in original variables}
We first define the projections
\begin{align*}
\PI \begin{pmatrix}
f_1 \\ \vdots \\ f_n 
\end{pmatrix} = f_1, 
\quad 
\PII \begin{pmatrix}
f_1 \\ 
f_2 \\ 
\vdots \\ f_n
\end{pmatrix} = f_2, \quad 
\Ph \begin{pmatrix} f_1 \\ \vdots \\ f_n 
\end{pmatrix}
= \begin{pmatrix}
f_3 \\ \vdots \\ f_n 
\end{pmatrix},
\end{align*}
and the coefficients $s_{ij}$ 
\begin{align}
SQ \begin{pmatrix}
\phiI \\ \phiII \\ \phih 
\end{pmatrix} = \begin{pmatrix}
\phiI + s_{12} \phiII + s_{13}^T \phih \\
s_{21} \phiI + s_{22} \phiII + s_{23}^T \phih \\ 
s_{31} \phiI + s_{32} \phiII + s_{33} \phih
\end{pmatrix}. 
\end{align}

The higher order corrections $\FI_j = \FI_j(\phiI, \phiII, \phih, t+T), j = 1,2,$ are then given by
\begin{align*}
\FI_1 &= \frac{3}{2 (t+T)} s_{12} \phiII - \frac{3}{2 \eta_* (t+T)} \partial_y \phiI + b_{12}^{10} \partial_t \phiII + b_{12}^{02} \partial_y^2 \phiII  + b_{13}^{01} \partial_y \phih, \\
\FI_2  &= \tilde{b}_{13} (\partial_t, \partial_y) \phih + \frac{3} {2(t+T)} s_{13}^T \phih - \frac{3}{2 \eta_* (t+T)} \partial_y( s_{12} \phiII + s_{13}^T\phih)  + \PI \omega N (\omega^{-1} Q \Phi),
\end{align*}
while $\FII_j = \FII_j (\phiI, \phiII,\phih, t+T), j = 1, 2, 3,$ are defined by
\begin{align*}
\FII_1 &= - b_{21}^{10} \partial_t \phiI - b_{21}^{02} \partial_y^2 \phiI - b_{22}^{01} \partial_y \phiII - \frac{3}{2 (t+T)} s_{21} \phiI, \\
\FII_2 &=  - b_{22}^{10} \partial_t \phiII - b_{22}^{02} \partial_y^2 \phiII - b_{23}^{01} \partial_y \phih - \frac{3}{2(t+T)} s_{22} \phiII + \frac{3}{2 \eta_*(t+T)} s_{21} \partial_y \phiI, \\
\end{align*}
and 
\begin{align*}
\FII_3 = - \tilde{b}_{23}(\partial_t, \partial_y)\phih - \frac{3}{2 (t+T)} \left( s_{23}^T \phih - \eta_*^{-1} s_{22}\partial_y \phiII- \eta_*^{-1} s_{23}^T \partial_y \phih \right)  - \PII S \omega N (\omega^{-1} Q \Phi). 
\end{align*}
Finally, $\Fh_j, j = 0, 1,2$ are defined as 
\begin{align*}
\Fh_0 &= - b_{31}^{10} \partial_t \phiI - b_{31}^{02} \partial_y^2 \phiI - b_{32}^{01} \partial_y \phiII - \frac{3}{2(t+T)} s_{31} \phiI, \\
\Fh_1 &=  \frac{3}{2 \eta_* (t+T)} s_{31} \partial_y \phiI - \left( b_{32}^{10} \partial_t + b_{32}^{02} \partial_y^2 + \frac{3}{2(t+T)} s_{32} \right) \phiII - b_{33}^{01} \partial_y \phih, 
\end{align*}
and
\begin{align*}
\Fh_2 = \frac{3}{2 \eta_* (t+T)} s_{32} \partial_y  \phiII + \left( -\tilde{b}_{33} (\partial_t, \partial_y) - \frac{3}{2(t+T)} s_{33} + \frac{3}{2\eta_*(t+T)} s_{33} \partial_y \right) \phih 
- \Ph \omega (N \omega^{-1} Q \Phi). 
\end{align*}

\subsection{Expressions in self-similar variables}\label{app: lin ind self sim}

Define for $\Psi = (\psiI, \psiII, \psih)^T$
\begin{align}
\tilde{N}(\xi, \tau, \psiI, \psiII, \psih) = \omega(y(\xi, \tau)) N \left( (\omega(y(\xi, \tau))^{-1} Q \Psi \right).
\end{align}
The higher order corrections in self-similar variables $\tilde{\mcf}^\mathrm{I/II/h}_j (\psiI, \psiII, \psih, \xi, \tau)$ are given by
\begin{multline}
\tfi_1 =  - \frac{3}{2\eta_*} e^{-\tau/2} \Deff^{-1/2} \partial_\xi\psiI + \left( \frac{3}{2} s_{12} + b_{12}^{10} (\partial_\tau - \frac{1}{2} \xi \partial_\xi ) + b_{12}^{02} \Deff^{-1} \partial_\xi^2 \right) \psiII + e^{\tau/2} b_{13}^{01} \Deff^{-1/2} \partial_\xi \psih, \label{e: tfi 1 expression}
\end{multline}
\begin{multline}
\tfi_2 =  - \frac{3}{2 \eta_*} e^{-\tau/2} \Deff^{-1/2} \partial_\xi \psiII  + \bigg[ e^\tau \tilde{b}_{13} \left( e^{-\tau} (\partial_\tau - \frac{1}{2} \xi \partial_\xi), e^{-\tau/2} \Deff^{-1/2} \partial_\xi \right) + \frac{3}{2} s_{13}^T \\ - \frac{3}{2 \eta_*} e^{-\tau/2} s_{13}^T \Deff^{-1/2} \partial_\xi \bigg] \psih  + e^\tau \PI \tilde{N}(\xi, \tau, \psiI, \psiII, \psih), 
\end{multline}
\begin{align}
\tfii_1 = \left( - b_{21}^{10} (\partial_\tau - \frac{1}{2} \xi \partial_\xi ) - b_{21}^{02} \Deff^{-1} \partial_\xi^2 - \frac{3}{2} s_{21}\right) \psiI - e^{\tau/2} b_{22}^{01} \Deff^{-1/2} \partial_\xi \psiII, \label{e: tf ii 1 expression}
\end{align}
\begin{multline}
\tfii_2 = \frac{3}{2 \eta_*} e^{-\tau/2} s_{21} \Deff^{-1/2} \partial_\xi  \psiI + \left( -b_{22}^{10} (\partial_\tau - \frac{1}{2} \xi \partial_\xi) - b_{22}^{02} \Deff^{-1} \partial_\xi^2 - \frac{3}{2} s_{22} \right) \psiII - e^{\tau/2} b_{23}^{01} \Deff^{-1/2} \partial_\xi \psih 
\end{multline}
\begin{multline}
\tfii_3 =  \left( - e^{\tau} \tilde{b}_{23} \left( e^{-\tau} (\partial_\tau - \frac{1}{2} \xi \partial_\xi), e^{-\tau/2} \Deff^{-1/2} \partial_\xi \right) - \frac{3}{2} s_{23}^T + \frac{3}{2 \eta_*} e^{-\tau/2} s_{23}^T \Deff^{-1/2} \partial_\xi \right) \psih \\ + \frac{3}{2 \eta_*} e^{-\tau/2} s_{22} \Deff^{-1/2} \partial_\xi \psiII
 - e^\tau \PII \tilde{N}(\xi, \tau, \psiI, \psiII, \psih) 
\end{multline}
\begin{align}
\tfh_0 = \left(-b_{31}^{10} (\partial_\tau - \frac{1}{2} \xi \partial_\xi ) - b_{31}^{02} \Deff^{-1} \partial_\xi^2 - \frac{3}{2} s_{31} \right) \psiI - b_{32}^{01} e^{\tau/2} \Deff^{-1/2} \partial_\xi \psiII. \label{e: tilde F h 0 expression}
\end{align}
\begin{multline}
\tfh_1 = \frac{3}{2 \eta_*} s_{31} e^{-\tau/2} \Deff^{-1/2} \partial_\xi \psiI - \left(  b_{32}^{10} (\partial_\tau - \frac{1}{2} \xi \partial_\xi) + b_{32}^{02} \Deff^{-1} \partial_\xi^2 + \frac{3}{2} s_{32} \right) \psiII - b_{33}^{01} e^{\tau/2}  \Deff^{-1/2} \partial_\xi \psih. \label{e: tfh 1 expression}
\end{multline}
\begin{multline}
\tfh_2 =  \left( - e^\tau \tilde{b}_{33} \left( e^{-\tau} (\partial_\tau - \frac{1}{2} \xi \partial_\xi), e^{-\tau/2} \Deff^{-1/2} \partial_\xi \right) - \frac{3}{2} s_{33} + \frac{3}{2 \eta_*} s_{33} e^{-\tau/2} \Deff^{-1/2} \partial_\xi \right) \psih \\ +  \frac{3}{2 \eta_*} s_{32} e^{-\tau/2} \Deff^{-1/2} \partial_\xi  \psiII - e^\tau \Ph \tilde{N}(\xi, \tau, \psiI, \psiII, \psih). 
\end{multline}

\section{Sharpness of resolvent estimates}\label{app: sharpness}
\begin{proof}[Proof of Lemma \ref{l: sharpness}]
	In the proof of Lemma \ref{l: center projections pole}, given in \cite[Lemma 2.2]{SIMA}, the projections $P^\mathrm{cs/cu}(\gamma)$ are constructed as a polynomial in $\mathcal{M}(\gamma^2)$, with coefficients involving its eigenvalues, via Lagrange interpolation. We now use a different characterization. Since the eigenspaces associated to $\nu^\pm(\gamma)$ are simple for $\gamma \neq 0$, we can write the associated spectral projection as
	\begin{align}
	P^\mathrm{cs}(\gamma) v = \frac{\langle v, \varrho^-_\mathrm{ad}(\gamma)\rangle}{\langle \varrho^-(\gamma), \varrho^-_\mathrm{ad}(\gamma)} \varrho^-(\gamma), \label{e: center stable projection}
	\end{align}
	where $\varrho^-(\gamma)$ spans the kernel of $\mathcal{M}(\gamma^2) - \nu^-(\gamma)$, and $\varrho^-_\mathrm{ad}(\gamma)$ spans the cokernel, with an analogous formula for $P^\mathrm{cu}(\gamma)$. Standard spectral perturbation theory \cite[Chapter 2]{Kato} implies that $\varrho^-(\gamma)$ and $\varrho^-_\mathrm{ad}(\gamma)$ are analytic in $\gamma$ in a neighborhood of the origin, with expansions
	\begin{align*}
	\varrho^-(\gamma) &= \varrho^-_0 + \varrho^-_1 \gamma + \mathrm{O}(\gamma^2), \\
	\varrho^-_\mathrm{ad}(\gamma) &= \varrho^-_{\mathrm{ad},0} + \varrho^-_{\mathrm{ad},1} \gamma + \mathrm{O}(\gamma^2). 
	\end{align*}
	Lemma \ref{l: center projections pole} implies that, since $P^\mathrm{cs}(\gamma)$ has a pole of order 1, we must have $\langle \varrho_0^-, \varrho^-_{\mathrm{ad},0} \rangle = 0$, but 
	\begin{align*}
	\langle \varrho^-_1, \varrho^-_{\mathrm{ad},0}\rangle + \langle \varrho^-_0, \varrho^-_{\mathrm{ad},1} \rangle \neq 0. 
	\end{align*}
	Expanding the formula \eqref{e: center stable projection}, we obtain
	\begin{align}
	P_\mathrm{pole} v = \frac{\langle v, \varrho^-_{\mathrm{ad},0} \rangle}{\langle \varrho^-_1, \varrho^-_{\mathrm{ad},0}\rangle + \langle \varrho^-_0, \varrho^-_{\mathrm{ad},1} \rangle} \varrho^-_0. 
	\end{align}
	{}	It follows from a short direct calculation, writing the explicit form of $\mathcal{M}(0)$ and observing that $\ker \mathcal{M}(0) = \spn ( \varrho^-_0), \ker \mathcal{M}(0)^* = \spn (\varrho^-_{\mathrm{ad,0}}))$, that 
	\begin{align}
	\varrho^-_0 = \begin{pmatrix} u_0 \\ 0 \end{pmatrix} \in \R^{2n},
	\end{align}
	and that there exists $v \in \mathrm{Rg} (\Lambda_1)$ such that $\langle v, \varrho^-_{\mathrm{ad},0} \rangle \neq 0$. Together these implies the desired result.  
\end{proof}

\bibliographystyle{abbrv}

\bibliography{systems}

\end{document}